\documentclass[11pt,letterpaper]{amsart}
\usepackage[margin=1.1in]{geometry}
\usepackage[utf8x]{inputenc}
\usepackage[toc,page]{appendix}
\usepackage{microtype}
\usepackage[T1]{fontenc}
\usepackage{enumitem}
\usepackage{tikz-cd}
\usetikzlibrary{positioning,fit,backgrounds}
\usepackage{ae,aecompl}
\usepackage{times}
\usepackage{float}
\usepackage{scalerel}

\usepackage{stmaryrd}
\usepackage{verbatim,amsmath,amsthm,amsfonts,amssymb,latexsym,graphicx,mathtools,color,mathabx,soul}
\usepackage[all]{xy}
\usepackage{graphicx}
\usepackage{caption}
\usepackage{subcaption}
\DeclareMathAlphabet{\mathpzc}{OT1}{pzc}{m}{it}
\usepackage[colorlinks,pagebackref,hypertexnames=false]{hyperref} 
\usepackage{hyperref}
\hypersetup{
  colorlinks = true,
  linkcolor  = black,
  citecolor=blue
} 

\usepackage[alphabetic,backrefs,msc-links]{amsrefs}
\usepackage{accents}

\graphicspath{./images/}
\usepackage{aliascnt}
\numberwithin{equation}{section}

\DeclareMathOperator{\tw}{tw}
   \newcommand{\HMb}{\overline{\mathit{HM}}}

\newtheorem{theorem}{Theorem}[section]
\newtheorem{prop}[theorem]{Proposition}
\newtheorem{lemma}[theorem]{Lemma}

\newtheorem{quest}{Question}
\newtheorem{cor}[theorem]{Corollary}

\newtheorem{theoremint}{Theorem}

\newtheorem{questint}[theoremint]{Question}

\theoremstyle{definition}
\newtheorem{definition}[theorem]{Definition}
\newtheorem*{definitionint}{Definition}

\theoremstyle{remark}
\newtheorem{remark}{Remark}[section]
\newtheorem{example}{Example}[section]
\newtheorem*{remarkint}{Remark}

\def\precube{\scaleobj{0.6}{\Box}}

\def\cube{{\mathrel{\raisebox{1pt}{$\precube$}}}}

\begin{document}

\title{Monopoles, twisted integral homology, and Hirsch algebras}

\author{Francesco Lin}
\address{Department of Mathematics, Columbia University} 
\email{flin@math.columbia.edu}

\author{Mike Miller Eismeier}
\address{Department of Mathematics, Columbia University} 
\email{smm2344@columbia.edu}

\begin{abstract}We provide an explicit computation over the integers of the bar version $\overline{HM}_*$ of the monopole Floer homology of a three-manifold in terms of a new invariant associated to its triple cup product called extended cup homology. This refines previous computations over fields of characteristic zero by Kronheimer and Mrowka, who established a relationship to Atiyah and Segal's twisted de Rham cohomology, and characteristic two by Lidman using surgery techniques in Heegaard Floer theory.

In order to do so, we first develop a general framework to study the homotopical properties of the cohomology of a dga twisted with respect a particular kind of Maurer-Cartan element called a twisting sequence. Then, for dgas equipped with the additional structure of a Hirsch algebra (which consists of certain higher operations that measure the failure of strict commutativity, and related associativity properties), we develop a product on twisting sequences and a theory of rational characteristic classes. These are inspired by Kraines' classical construction of higher Massey products and may be of independent interest. 

We then compute the most important infinite family of such higher operations explicitly for the minimal cubical realization of the torus. Building on the work of Kronheimer and Mrowka, the determination of $\overline{HM}_*$ follows from these computations and certain functoriality properties of the rational characteristic classes.
\end{abstract}

\maketitle

\tableofcontents

\section*{Introduction}

\textbf{Results in monopole Floer homology.}
In \cite{KM}, the authors use Seiberg-Witten theory to define several invariants of three-manifolds, collectively referred to as \textit{monopole Floer homology groups}. Among these, the simplest version assigns to each three-manifold $Y$ equipped with a spin$^c$ structure $\mathfrak{s}$ the invariant $\HMb_*(Y,\mathfrak{s})$, read \textit{HM-bar}.
\par
While this invariant vanishes when $\mathfrak{s}$ is not torsion, it is much more interesting in the case of a torsion spin$^c$ structure. In that case, we have that $\HMb_*(Y,\mathfrak{s})$ is a finitely generated relatively $\mathbb{Z}$-graded module over the ring of Laurent polynomials $\mathbb{Z}[U^{-1},U]$ in a variable $U$ of degree $-2$. In \cite[Chapter $35$]{KM} the authors prove several fundamental results, which we briefly summarize:
\begin{enumerate}[label=(\roman*)]
\item The invariant $\HMb_*(Y,\mathfrak s)$ only depends on the cohomology ring of $Y$, and in particular on the triple cup product
\begin{align*}
\Lambda^3 H^1(Y;\mathbb{Z})&\rightarrow \mathbb{Z}\\
\alpha_1\wedge \alpha_2\wedge \alpha_3&\mapsto \langle \alpha_1\cup \alpha_2\cup \alpha_3,[Y]\rangle,
\end{align*}
which we denote by $\cup_Y^3\in\left(\Lambda^3 H^1(Y;\mathbb{Z})\right)^*$.
\item There exists a spectral sequence $(E^r,d^r)$ converging to $\HMb_*(Y,\mathfrak{s})$ for which
\begin{equation*}
E^3=\Lambda^* H^1(Y;\mathbb{Z})\otimes \mathbb{Z}[U^{-1},U]
\end{equation*}
and the differential $d^3$ is given by
\begin{equation}\label{d3}
d^3(\omega\otimes U^n)= \iota_{\cup_Y^3}\omega\otimes U^{n-1},
\end{equation}
where $\iota_{\cup_Y^3}$ is the contraction with the triple cup product $\cup_Y^3$ sending
$\alpha_1\wedge\dots\wedge\alpha_k$ to
\begin{equation*}
\sum_{i_1<i_2<i_3} (-1)^{i_1+i_2+i_3}\langle \alpha_{i_1}\cup \alpha_{i_2}\cup \alpha_{i_3},[Y]\rangle\cdot \alpha_1 \wedge\dots  \wedge\hat{\alpha}_{i_1} \wedge\dots \wedge\hat{\alpha}_{i_2} \wedge\dots \wedge\hat{\alpha}_{i_3}\wedge\dots\wedge\alpha_k.
\end{equation*}
Up to signs, $\iota_{\cup_Y^3}$ is Poincar\'e dual to cup product with the $3$-form $\cup_Y^3$. The differential $d^3$ has degree $-1$ once we declare the elements of $H^1(Y;\mathbb{Z})$ to have degree $1$.
\item After tensoring with $\mathbb{R}$ (or any field with characteristic zero), the spectral sequence collapses at the $E^4$ page.
\end{enumerate}
Using the last point, the authors were able to conclude that the invariant $\HMb_*(Y,\mathfrak{s})$ is non-vanishing when $\mathfrak{s}$ is torsion (and in fact an infinite rank $\mathbb{Z}$-module). This provided a fundamental step in Taubes' celebrated proof of the Weinstein conjecture in dimension three \cite{Tau}.
\\ 
\par
The main goal of the present paper is to refine the results of \cite{KM} by providing a computation of $\HMb_*(Y,\mathfrak s)$ over $\mathbb{Z}$, or more generally with local coefficients (over any ring). We will refer to the $E^4$ page of the spectral sequence in $(\mathrm{ii})$ as the \textit{cup homology} of $Y$. This was introduced and thoroughly studied in \cite{Mar}, where it is denoted by $HC_*^{\infty}(Y)$; in this notation, Kronheimer and Mrowka's result is that $$\HMb_*(Y,\mathfrak s)\otimes \mathbb{Q}=HC_*^{\infty}(Y)\otimes\mathbb{Q}.$$
Our main theorem gives an isomorphism over $\mathbb{Z}$ from $\HMb_*(Y, \mathfrak s)$ to a more intricate variant of the cup homology groups. To state it, suppose $a \in \Lambda^3(\mathbb Z^n)$ may be written as a sum of basis monomials as $a = \sum_{I = \{i < j < k\}} a_I e^i e^j e^k$; it will be useful to write $e^I = e^i e^j e^k$. Associated to $a$ are odd-degree classes called its \emph{insertion powers}, defined precisely in Definition \ref{def:insertionpower}. We content ourselves with a precise definition of the first here: $$a^{\circ 2} = -\sum_{\substack{I, J \text{ with } |I \cap J| = 1}} a_I a_J e^{I_0} \wedge e^J \wedge e^{I_1}\in\Lambda^5(\mathbb{Z}^n),$$
where $I_0$ and $I_1$ are the (possibly empty) substrings of $I$ appearing before and after the element $I\cap J$ respectively. In words, if the (unordered) pair $I, J$ meets at exactly one point, we insert $J$ where they meet, and otherwise ignore it. One has, for instance, 
\begin{align*}
(e^{123} + e^{245})^{\circ 2}&=-e^{12453}\\
(e^{123} + e^{345} + e^{567})^{\circ 2} &= -e^{12345} -e^{34567}.
\end{align*}
The higher insertion powers $a^{\circ k}\in \Lambda^{2k+1}(\mathbb Z^n)$ essentially iterate this process; we have for example $$(e^{123} + e^{345} + e^{567})^{\circ 3} = +e^{1234567}\in\Lambda^{7}(\mathbb Z^n),$$
the plus sign arising from two canceling minus signs.
\textbf{Choose} a basis for $H^1(Y)=\mathbb{Z}^n$ so that $\cup^3_Y$ is identified with a class $a \in \Lambda^3(\mathbb Z^n)$. Via duality, one then identifies $\iota_{\cup^3_Y}=\iota_{a}$; the latter naturally generalizes to the contractions $\iota_{a^{\circ k}}$ with the classes $a^{\circ k}\in \Lambda^{2k+1}(\mathbb Z^n)$. In this notation, the cup homology $HC_*^\infty(Y)$ is isomorphic to the homology of $\Lambda^*(\mathbb{Z}^n)[U, U^{-1}]$ with respect to the differential given by $x\otimes U^{n}\mapsto \iota_{a}(x)\otimes U^{n-1}$.
\begin{definitionint}
The \emph{extended cup homology} $\overline{HC}_*^{\infty}(Y)$ of $Y$ with respect to the given basis is the homology of $\Lambda^*(\mathbb{Z}^n)[U, U^{-1}]$ with respect to the degree $-1$ differential given by
\begin{equation*}
x\otimes U^{n}\mapsto \iota_{a}(x)\otimes U^{n-1} + \iota_{a^{\circ 2}}(x)\otimes U^{n-2} + \iota_{a^{\circ 3}}( x) \otimes U^{n-3} + \cdots
\end{equation*}
where $a=\cup^3_Y\in \Lambda^3(\mathbb Z^n)$.
\end{definitionint}
\begin{remarkint}
The natural filtration of $\Lambda^*(\mathbb{Z}^n)$ by degree gives rise to a spectral sequence abutting to $\overline{HC}_*^\infty(Y)$ with $E^4$-page $HC_*^\infty(Y)$; the higher differentials can be described in terms of $a$ in a purely algebraic fashion.
\end{remarkint}

Our main result is the following \textit{completely explicit} computation of $\HMb_*$ over the integers.
\begin{theoremint}\label{thmmain}
If $\mathfrak{s}$ is a torsion spin$^c$ structure, the monopole Floer homology group $\HMb_*(Y,\mathfrak s)$ is isomorphic to the extended cup homology $\overline{HC}_*^{\infty}(Y)$ as relatively $\mathbb{Z}$-graded $\mathbb Z[U, U^{-1}]$-modules.
\end{theoremint}
\begin{remarkint}
For the reader more familiar with Heegaard Floer homology, our result directly translates to a computation of the group $HF^{\infty}(Y,\mathfrak{s})$ over the integers for a torsion spin$^c$ structure via the isomorphism between the theories (see \cite{KLT}, \cite{CGH} and subsequent papers). In that setting, Lidman showed using surgery techniques that the invariant coincides with the usual cup homology over $\mathbb F_2$ \cite{Lid}.
\end{remarkint}

Our main result is somewhat curious, especially given the fact that there exists an isomorphism to the usual cup homology $HC^\infty$ when the ground ring is $\mathbb F_2$ or $\mathbb Q$. This follows from the aforementioned computations of the invariants with those coefficients due to Lidman and Kronheimer--Mrowka respectively. One can show (using ideas related to the content of the paper) that this implies that the spectral sequence with $E^4=HC^\infty$ and converging to $\overline{HC}^\infty$ has $d_5 = 0$ after tensoring with any field, but for example leaves open the possibility that $d_7$ is nonzero even after tensoring with a field of characteristic $3$. 

To determine whether the non-vanishing of some higher differential was plausible, we computed both $\overline{HC}^\infty$ and $HC^\infty$ for $n \le 12$ and about $1000$ random choices of $a$, where $a$ was given as a sum of $30$ monomials with random coefficients at most $10$. In all cases, the underlying $\mathbb Z$-graded abelian groups were isomorphic. This leads us to the following completely algebraic question.

\begin{questint}\label{conj-cup}
Let $a \in \Lambda^3(\mathbb Z^n)$. Are the extended cup homology $\overline{HC}^\infty$ and the cup homology $HC^\infty$ associated to $a$ isomorphic (possibly noncanonically) as $\mathbb Z[U, U^{-1}]$-modules?
\end{questint}

\begin{remarkint}
Because both cup homology and its extended version have lots of integral torsion, a positive answer to this question is strictly stronger than the vanishing of the differentials in the spectral sequence; the analogue of the latter in Heegaard Floer homology was conjectured to hold by Ozsv\'ath and Szab\'o \cite{OSz}. Also, classical work of Sullivan \cite{Sul} shows that any $a \in \Lambda^3(\mathbb Z^n)$ can be realized as the triple cup product of some three-manifold; therefore the question above is asking whether $\overline{HC}^{\infty}(Y)\cong HC^{\infty}(Y)$ for all $Y$.
\end{remarkint}

\begin{remarkint}
The extended cup complex depends on the choice of basis in an essential way; it is only clear that the resulting homology groups are independent of basis as a result of the isomorphism to $\HMb_*(Y;\mathfrak s)$. As discussed in the final section of this introduction, the choice of basis arises naturally in the course of the proof, yet cup homology itself has no such dependence on basis, nor does $\HMb_*$. In this light, the dependence of the extended cup complex on a choice of basis appears rather strange.
\end{remarkint}

The monopole Floer homology group $\HMb_*(Y, \mathfrak s)$ vanishes for a non-torsion spin$^c$ structure. A more interesting version, denoted by $\HMb_*(Y, \mathfrak s, c_b)$, arises by looking at the equations perturbed by a \textit{balanced non-exact} perturbation $c_b$ \cite[Chapter $30$]{KM}. This invariant is only relatively $\mathbb{Z}/2N\mathbb{Z}$-graded, where
\begin{equation*}
N=\frac{1}{2}\mathrm{gcd}\left\{\langle a\cup{c_1(\mathfrak{s}}),[Y]\rangle\text{ $\lvert$ } a\in H^1(Y;\mathbb{Z})\right\},
\end{equation*}
and corresponds in Heegaard Floer homology to the group $HF^{\infty}(Y,\mathfrak{s})$ under the isomorphism between the theories. In direct analogy to Theorem \ref{thmmain}, we will provide a complete computation of this invariant in terms of suitable `twisted' extended cup homology groups, see Section $6.2$ for the exact statement. In fact, our approach readily generalizes also to compute the invariants twisted by a local coefficient system $\Gamma_0$ in the blown-up configuration space \cite[Section $3.7$]{KM}, for both torsion and non-torsion spin$^c$ structures. 

\begin{remarkint}
The proof of Theorem \ref{thmmain} may either be phrased in terms of homology or cohomology. In our discussion we prefer the latter, so that we may focus our discussion on algebras (as opposed to algebras and their modules). In particular, in the proof of Theorem \ref{thmmain} we will mostly focus on the cohomological version of the invariant $\HMb^*(Y,
\mathfrak{s})$. When we deal with local systems and non-torsion spin$^c$ structures, and are forced to use modules, we return to the homological version.
\end{remarkint}

\vspace{0.3cm}

\textbf{Coupled Morse theory and relations with twisted de Rham cohomology. }In \cite[Chapter $33$]{KM}, Kronheimer and Mrowka introduce the \textit{coupled Morse complex} $CMC^*(M, L)$ for a compact smooth manifold $M$ with a Morse function $f$ and a family of self-adjoint Fredholm operators $L$ over $M$, classified (up to homotopy) by a map $\zeta_L: M \to SU(\infty)$. Its cohomology is a module over $\mathbb{Z}[U, U^{-1}]$. Furthermore, Kronheimer and Mrowka prove (\cite[Section $35.1$]{KM}) that when $L$ is the family of Dirac operators $\{D_B\}$ parameterized by the torus $\mathbb{T} = \mathbb T(Y, \mathfrak s)$ of flat spin$^c$ connections of a torsion spin$^c$ structure $\mathfrak s$, the coupled Morse cohomology $CMH^*(\mathbb T, D_B)$ recovers the Floer cohomology group $\HMb^*(Y,\mathfrak s)$, together with its module structure over $\mathbb{Z}[U, U^{-1}]$.

Our proof will build on the simplicial model constructed in \cite[Section $34.3$]{KM} for the coupled Morse cohomology complex $C^*(M,L)$ of a family of self-adjoint operators $L$ on a smooth manifold $M$ classified (up to homotopy) by a map $$\zeta:M\rightarrow SU(2).$$ The authors show that for a sufficiently fine $\Delta$-complex structure on $M$, there is a simplicial $3$-cocycle $x_3$ for which the chain complex
\begin{equation}\label{twistcomp}
C^*_{\tw}(M;x_3)=C^*_{\Delta}(M)\otimes \mathbb{Z}[T^{-1},T]
\end{equation}
equipped with the differential
\begin{equation}\label{twistdiff}
\sigma \mapsto d\sigma +(x_3\cup \sigma)T,
\end{equation}
is quasi-isomorphic to $CMC^*(M,L)$. Here $T$ is a formal variable of degree $-2$; this quasi-isomorphism sends the action of $U^{-1}$ in coupled Morse cohomology to $T$ in twisted singular homology. The two variables are related in a more complicated way when one allows $\zeta$ to factor through $U(2)$ instead, and the exact relationship remains open; again see Section 6 for more details.
\par
We will refer to this model as the simplicial cohomology of $M$ twisted by the $3$-cocycle $\xi_3$, and denote it by $H^*_{\tw}(M;\xi_3)$. Notice that it is implicit in the construction that $\xi_3^2=0$ at the cochain level, as otherwise (\ref{twistdiff}) does not square to zero. We call such cocycles \emph{twisting 3-cycles}. The class $[\xi_3]\in H^3(M;\mathbb{Z})$ is the pullback via the classifying map $\zeta$ of the generator of $H^3(SU(2);\mathbb{Z})=\mathbb{Z}$. In the case of the family of Dirac operators associated to $(Y,\mathfrak{s})$, the corresponding element in $H^3(\mathbb{T};\mathbb{Z})=\left(\Lambda^3 H^1(Y;\mathbb{Z})\right)^*$ is naturally identified with $\cup_Y^3$, the triple cup product of $Y$, \cite[Lemma $35.1.2$]{KM}.
\par
In \cite{KM} the authors then show that after tensoring with $\mathbb{R}$, this construction is equivalent to a version of Atiyah and Segal's twisted de Rham cohomology \cite{AS}, where the twisting is provided by a closed $3$-form representing $[x_3]$ in de Rham cohomology. The result in (iii) then follows immediately from the formality of the de Rham complex of the torus. Furthermore, this shows that when $L: M \to SU(\infty)$ factors through $SU(2)$, the coupled Morse cohomology with real coefficients $CMH^*(M,L;\mathbb R)$ only depends on the real cohomology class $\zeta_L^*[SU(2)].$
\\
\par
\textbf{Integral twisted cohomology.} By contrast, the analogous story is significantly more complicated over the integers. In the final remark of \cite[Section $34.3$]{KM}, the authors point out that cohomologous twisting $3$-cycles $x_3$ and $x_3'$ might lead to non-isomorphic twisted cohomology groups, but do not provide explicit examples of this phenomenon.
\par
In fact, already in the simplest case in which $[x_3]=0\in H^3(M;\mathbb{Z})$, the twisted cohomology group $H^*_{\tw}(M;x_3)$ might differ from $H^*_{\tw}(M;0)=H^*(M)\otimes \mathbb{Z}[T^{-1},T]$. Indeed, if $d h_2 + x_3 = 0$, $x_3\cup h_2$ defines a cohomology class in $H^5(M;\mathbb{Z})$ as $x_3^2=0$, and one shows that the $E_5$ page of the spectral sequence associated to the grading filtration is equivalent to the chain complex with the same underlying group $H^*(M)\otimes \mathbb{Z}[T^{-1},T]$ and differential
\begin{equation}\label{twistdiff2}
\sigma \mapsto \left([x_3\cup h_2]\cup \sigma\right)T^2.
\end{equation}
In some cases of interest, we will see that $[x_3\cup h_2]$ is a possibly non-vanishing 2-torsion class, so that while the differential $d_3$ of the associated spectral sequence vanishes, as $[x_3]=0$, the differential $d_5$ might be non-zero. We give an explicit example where this occurs in Section \ref{nonformalsu2} below, and prove the following result (providing a concrete example of the possible phenomenon described in \cite[Section $34.3$]{KM}) as a consequence.

\begin{theoremint}\label{cohono}
There exists a closed smooth manifold $M$ equipped with two families of self-adjoint operators $L_i$ ($i=0,1$) with non-isomorphic coupled Morse homologies which are classified by maps $\zeta_i: M\rightarrow SU(2)$ for which $\zeta_0^*[SU(2)]=\zeta_1^*[SU(2)]\in H^3(M;\mathbb{Z})$.
\end{theoremint}

This discussion suggests that even if we are only interested in studying the cohomology twisted by a square-zero $3$-cocycle, higher-degree cochains (such as $x_3\cup h_2$) necessarily enter the theory. In fact, the latter provide a more natural and invariant framework to deal with twisted cohomology from the point of view of homotopy theory, and naturally lead to the following definition.
\begin{definitionint}\label{twistintro} Let $A$ be a dg-algebra. A \textbf{twisting sequence} for $A$ is a sequence $x_{\bullet}=(x_3, x_5,\dots)$ of odd degree cochains satisfying the relation
\begin{equation*}
d x_{2n+1}+\sum_{\substack{n=i+j \\ i,j \geq 1}} x_{2i+1}\cup x_{2j+1}=0.
\end{equation*}
The twisting sequence $x_{\bullet}$ defines the differential on $A\llbracket T, T^{-1}]$ given by
\begin{equation*}
\sigma\mapsto d \sigma+\sum_{i\geq 1} (x_{2i+1} \cup \sigma) T^i,
\end{equation*}
and we denote the corresponding twisted cohomology group by $H^*_{\tw}(A;x_{\bullet})$. When $A$ is bounded above in degree, we may use Laurent polynomials $A[T, T^{-1}]$ instead of Laurent series as above.

A $3$-cocycle $x_3$ with $x_3^2=0$ naturally defines a twisting sequence, namely $(x_3,0,0,\dots)$; we will call the latter the twisting sequence associated to $x_3$.
\end{definitionint}
\begin{remarkint}
The notions of twisting sequence and twisted cohomology make sense in much more general contexts, where they are known under the name \emph{twisting cochain}, \emph{twisting element}, or sometimes \emph{Maurer-Cartan element}. These first appeared in \cite{Brown} in a description of the cochain algebra of a fiber bundle as a twisted tensor product, twisted by such a twisting cochain. The above is essentially the special case of a Maurer-Cartan element in $A\llbracket T\rrbracket$ contained in $TA\llbracket T\rrbracket$. One might then denote the set of twisting sequences by $MC_+(A\llbracket T\rrbracket)$; this is rather a mouthful, and we will later use the notation $TS(A)$. 

While some of our algebraic results might hold in more general contexts, we only pursue them in the level of generality relevant to our main application.
\end{remarkint}

\medskip

\textbf{Twisting sequences and Hirsch algebras.} 
Twisting sequences enjoy a notion of homotopy, so that if $f: A \to B$ is a quasi-isomorphism of dg-algebras, it induces a bijection between homotopy classes of twisting sequences (see Proposition \ref{qiso-transfer}); furthermore, homotopic twisting sequences give rise to isomorphic twisted (co)homology groups. Because Kronheimer and Mrowka's result computes that $\HMb^*(Y, \mathfrak s)$ is isomorphic to a twisted cohomology group of the algebra $C^*_\Delta(\mathbb T)$ with respect to an appropriate twisting sequence $\xi_{KM} = (\xi_3, 0, \cdots)$, our goal is to transfer this twisting sequence to the exterior algebra algebra $H^*(\mathbb T)$ by using the fact that cochains on the torus give a \emph{formal} dga: there is a zig-zag of quasi-isomorphisms from $C^*(\mathbb T)$ to $H^*(\mathbb T)$. 

Unfortunately, this transferred twisting sequence is \emph{completely inexplicit}, because the proof that quasi-isomorphisms induce a bijection on homotopy classes of twisting sequence is inexplicit. We need something stronger to determine what twisting sequence $\xi_{KM}$ is transferred to. Observe that if $x_\bullet$ is a twisting sequence, $[x_3]$ gives a cohomology class which is homotopy invariant and natural under pushforward of twisting sequences. Therefore in the above transfer process we can at least recover $[x_3]$. Ideally we would now say that $[x_5], [x_7]$ and so on play the same role in higher degrees. But $x_5$ is not a cycle, so does not define a cohomology class! Indeed we have $dx_5 = -x_3^2$. If we want to construct a homology class using $x_5$, we will need a \emph{canonical} reason that $x_3^2$ is null-homotopic.
\par
This leads us to study \emph{Hirsch algebras} (see \cite{Saneblidze}) in Section 2. These are dgas equipped with an operation $$E_{1,1}: A^p \otimes A^q \to A^{p+q-1}$$ which demonstrates that the product on $A$ is homotopy-commutative (playing the same role as Steenrod's cup-$1$ product on simplicial cochains), together with additional operations $E_{p,q}$ which assert that $E_{1,1}$ is associative up to coherent homotopy and gives a derivation of the cup product up to coherent homotopy. These are relevant because $dE_{1,1}(x_3, x_3) = -2x_3^2$. Much of Section \ref{examplesHirsch} is devoted to giving explicit constructions of Hirsch structures in two cases of interest, simplicial and cubical cochains $C^*_\Delta$ and $C^*_\cube$ of simplicial/cubical sets, using recent work of Medina-Mardones on $E_\infty$ operads \cite{MM1} but attempting to remain as explicit as possible (which is necessary for our later explicit computations on the minimal torus).
\par
In a Hirsch algebra $A$, we will define \textit{rational characteristic classes} for a twisting sequence $x_{\bullet}$ $$F_n(x_\bullet) \in H^{2n+1}(A)_{\mathbb Q}$$ which are homotopy invariants and natural under maps of Hirsch algebras. When $A$ has torsion-free cohomology, these characteristic classes are enough to recover the original twisting sequence up to homotopy (Theorem \ref{charclass-facts}). They are exactly what we need to compute the transferred twisting sequence in $H^*(\mathbb T)$ above.

We construct these characteristic classes via two constructions, valid for arbitrary Hirsch algebras $A$ and possibly of independent interest. The inspiration for both of these is the fact (Proposition \ref{gplike-cocycle}) that elements of $A\llbracket T\rrbracket$ with $dx + x^2 = 0$ correspond bijectively to elements $g(x)$ of the bar construction $BA\llbracket T\rrbracket$ which are both:
\begin{itemize}
    \item \emph{grouplike}, i.e. $\Delta g(x) = g(x) \otimes g(x)$,
    \item \emph{cocycles}, i.e. $d_{BA} g(x) = 0$.
\end{itemize}
The reason a Hirsch algebra structure is relevant is that it gives rise to a product $\mu: BA \otimes BA \to BA$ so that $BA$ becomes a (possibly nonassociative) dg-bialgebra. Then products of grouplike cocycles are again grouplike cocycles, and we have the following result (Corollary \ref{ProdOp}).

\begin{theoremint}\label{hT-prod}
Let $A$ be a Hirsch algebra. Then the set $TS(A)$ of twisting sequences in $A$ has an explicit unital product $\mu: TS(A) \times TS(A) \to TS(A)$, which is natural for Hirsch algebra maps. Writing $ts(A)$ for the set of homotopy classes of twisting sequences, this descends to a product $h\mu: ts(A) \times ts(A) \to ts(A)$. 
\end{theoremint}
\begin{remarkint}
It seems likely that $h\mu$ is associative (though the map $\mu: BA \otimes BA \to BA$ need not be, it seems plausible that one should be able to choose a homotopy equivalent model in which it is indeed associative).
\end{remarkint}

Our second construction is inspired by the construction of \cite[Lemma 16]{Kra} and its extension to Hirsch algebras in \cite[Section 3.3]{Saneblidze}.

\begin{theoremint}\label{twistK}
{Let $A$ be a Hirsch algebra.} For each odd cocycle $a \in Z(A)^{2n+1}$, there is a \textbf{canonical} twisting sequence $K(a)$ in $A_{\mathbb Q}$ with $K(a)_{2i+1} = 0$ for $i < n$ and $K(a)_{2n+1} = a$. These twisting sequences are natural for Hirsch algebra maps.
\end{theoremint}

Putting these together, for any twisting sequence $x_\bullet$ in a Hirsch algebra $A$, we can construct a canonical approximation $K_{(n)}(x_\bullet)$ in $A_{\mathbb Q}$ which agrees with $x_\bullet$ through degree $2n-1$; the $n$'th characteristic class is given by $$F_n(x_\bullet) = [x_{2n+1} - K_{(n)}(x_\bullet)_{2n+1}],$$ and may be understood as the first obstruction to finding a homotopy from $x_\bullet$ to $K_{(n)}(x_\bullet)$.
Let us point out that the theory of characteristic classes is significantly simpler in the case in which $A=C^*_{\Delta}(X)$; this is because all higher operations $E_{p,q}$ with $p\geq 2$ vanish.\footnote{More precisely, our characteristic classes rely heavily on the operations $E_{1,p}$; one may obtain characteristic classes which only depend on the cup-1 product if one sets up the theory with the operations $E_{p,1}$ instead. This is discussed in more detail in Remark \ref{rmk:right-mult}.} On the other hand, we will need the more general machinery because we will work with cubical cochains, for which the analogue of the classical Hirsch formula \cite{Hir} only holds up to homotopy (rather than on the nose as in the simplicial case).
\\
\par
\textbf{Computations on the minimal torus.} 
To apply the machinery of the previous subsection to our problem, we need to find a zig-zag of \emph{Hirsch algebra} quasi-isomorphisms from $C^*_\Delta(\mathbb T)$ to $H^*(\mathbb T)$. The most obvious approach is to use combinations of the classical Eilenberg-Zilber and Alexander-Whitney maps to reduce this to a tensor product of circles, but these maps are \textbf{not} Hirsch algebra maps in an obvious way. 

Instead, we observe that there is a geometric model for $\mathbb T$ whose cochain algebra is isomorphic to $H^*(\mathbb T)$ and carries the structure of a Hirsch algebra. This is not possible simplicially: the minimal triangulation of $T^n$ certainly requires more than one $n$-simplex, and the standard small triangulation has $n!$ of them. Instead, we use a cubical model $T^n_1$, which we call the \emph{minimal torus}. This is a cubical set obtained as a quotient of the standard cube $\Box^n$ by pasting together opposite sides. 

There is a zigzag of comparison maps between $H^*(T^n) = C^*_\cube(T^n_1)$ and the simplicial cochain algebra $C^*_\Delta(\mathbb T)$, all of which are Hirsch algebra quasi-isomorphisms. Choosing such a zig-zag is essentially equivalent to choosing a basis for $H^1(\Bbb T;\Bbb Z)$, which is why the definition of the extended cup complex depends on a choice of basis.. The twisting sequence we are interested in --- $\xi_{KM} = (\xi_3, 0, \cdots)$ --- is pulled back to $C^*_\Delta(\mathbb T)$ from $C^*_\Delta (SU_2)$ by a simplicial map. By naturality, the characteristic classes have $F_n(\xi_{KM}) = 0$ for all $n > 1$, while $F_1 (\xi_{KM}) = [\xi_3] = \cup^3_Y$ is known to be the triple cup product. 

The twisted cohomology of $C^*_\Delta(\mathbb T)$ with respect to $\xi_{KM}$ is isomorphic to the twisted cohomology of $H^*(T^n) = \Lambda^*(\mathbb Z^n)$ with respect to an appropriate transferred twisting sequence $z_* \xi_{KM}$. This transferred twisting sequence must have $F_1(z_* \xi_{KM}) = \cup^3_Y$ and $F_n(z_* \xi_{KM}) = 0$ for $n > 1$. We will see that the canonical the twisting sequence $K(\cup^3_Y)$ associated to $\cup^3_Y$ in Theorem \ref{twistK} has these properties (cf. Theorem \ref{charclass-facts} below); because the torus has bounded and torsion-free cohomology, this must be homotopic to $z_* \xi_{KM}$ (and therefore equal because the notion of homotopy degenerates on any algebra with trivial differential). 
\par
We thus have an isomorphism from $\overline{HM}^*(Y, \mathfrak s)$ to the twisted cohomology of $H^*(T^n)$ with respect to $K(\cup^3_Y)$. What we need to do is \emph{determine this twisting sequence}. We carry this out in Section \ref{mintoruscomp} by giving a complete calculation of the operations $E_{1,p}$ for the minimal torus, where we show that in $H^*(T^n)$, the element $K(a) = (a, a^{\circ 2}, a^{\circ 3}, \dots)$, as in Theorem \ref{thmmain}. This is not a full calculation of the Hirsch algebra structure on the minimal torus --- for instance, the operation $E_{2,2}$ is nonzero. However, these higher operations are irrelevant to our calculations. Finally, in Section \ref{sec:localsys}, we explain why this result also implies an isomorphism with respect to various local systems, and in Section \ref{HMblocsys} we carefully state the version of main theorem in the setting of local systems and non-torsion spin$^c$ structures.

\vspace{0.5cm}
\textbf{Acknowledgements. } The authors would like to thank Andrew Blumberg, Anibal Medina-Mardones, John Morgan and Boyu Zhang for some helpful conversations, as well as Pedro Tamaroff for notational suggestions; they are also grateful to the anonymous referee whose feedback helped greatly improve the manuscript. The authors are especially thankful for the keen eyes of Stefan Behrens and Thomas Kragh, who alerted us to an error in a previous version of this article. The first author was partially supported by NSF Grant DMS-1948820 and the Alfred P. Sloan Foundation.

\bigskip
\section{Twisting sequences for general dg-algebras}\label{twistseq}
The first part of this section deals with definitions as well as functoriality and invariance properties of twisting sequences and the associated twisted cohomology. In the second part, we discuss an obstruction theory for twisting sequences and its consequences.
\subsection{Generalities on twisting sequences}
In this section we discuss the notion of \emph{integral twisted (co)homology}. This is inspired by Atiyah and Segal's twisted de Rham cohomology \cite{AS}. In the de Rham setting, wedge squares of odd-degree forms are automatically zero, and given a closed odd-degree form $\omega$ the map $$(d + m_\omega)(\eta) = d\eta + \omega \wedge \eta$$ is a square-zero operator, with respect to which we can take cohomology. 
\par
Working integrally, we have no such luck, because the cup product does not commute on the nose: in fact, if $x$ is an odd-degree cycle, $x^2$ might even define a non-zero (2-torsion) class in cohomology. To prevent this, we should ask that there is a chain $y$ with $dy + x^2 = 0$, which we then need to incorporate into our twisted differential. We then need additional chains to cancel out the contribution from $xy + yx$ and $y^2$, and so on.

Starting with a degree-3 cycle as our basic twist, this leads us to the definition of twisting sequence for a dg-algebra. First, let us set conventions. In what follows, all dg-algebras are graded over $\Bbb Z$ and the differential has degree $+1$, satisfying $d(ab) = (da)b + (-1)^{|a|} a (db)$.

\begin{definition}
A \textbf{cohomological} dg-module is a $\Bbb Z$-graded chain complex $M$ with differential of degree $+1$, together with an action $A \otimes M \to M$ satisfying $a(bm) = (ab)m$, $|am| = |a| + |m|$, and $d(am) = (da)m + (-1)^{|a|} a (dm)$.

\noindent A \textbf{homological} dg-module is a $\Bbb Z$-graded chain complex $M$ with differential of degree $-1$, together with an action $A \otimes M \to M$ satisfying $b(am) = (ab)m$, $|am| = |m| - |a|$, and $d(am) = (-1)^{|a|}( a(dm) - (da)m).$
\end{definition}

The standard example of a cohomological dg-module is $C^*(X;\Bbb Z)$ as a module over itself with the cup-product action. The standard example of a homological dg-module is $C_*(X;\Bbb Z)$ as a module over $C^*(X;\Bbb Z)$ with the cap-product action. 

Henceforth, we will almost exclusively work with cohomological dg-modules to avoid writing two nearly identical proofs. All of the results below still apply for homological dg-modules, with the same signs. We discuss the distinction in the rare occasions it is important. Furthermore, we will also assume that all dg-algebras and modules are free as $\mathbb{Z}$-modules and have finitely generated cohomology in each dimension.

\begin{definition}
Let $A$ be a dg-algebra. A twisting sequence $x_{\bullet}$ in $A$ is a sequence $$x_{\bullet}=(x_{2n+1}) \in \prod_{n \geq 1} A^{2n+1}(X;\mathbb Z)$$ of odd-degree elements of $A$ so that $$dx_{2n+1} + \sum_{\substack{i + j = n \\ i,j \geq 1}} x_{2i+1} x_{2j+1} = 0 $$ for all $n$. When $x_{2i+1} = 0$ for all $i \neq n$, we say that $x_{2n+1}$ is a twisting $(2n+1)$-cycle, or simply a square-zero cocycle. 

We write the set of twisting sequences in $A$ as $TS(A)$.
\end{definition}

\medskip

This is precisely what is needed to define a (generalized) twisted cohomology group. 

\begin{definition}\label{twistHdef}
Consider a dg-algebra $A$ with twisting sequence $x_{\bullet}$. For a dg module $M$ over $A$, we define the twisted chain complex of $M$ as $$C_{\textup{tw}}^*(M;x_\bullet) = \left(M \otimes \mathbb Z\llbracket T, T^{-1}], d + L_3 T + L_5 T^2 + \cdots\right),$$ where $$L_{2n+1} m = x_{2n+1} \cdot m,$$ the formal variable $T$ has degree $-2$, and $\mathbb Z\llbracket T, T^{-1}]$ denotes the Laurent series in $T$. When $M$ is bounded above in degree, this is identical to $M \otimes \mathbb Z[T, T^{-1}]$, and we may use Laurent polynomials instead.\footnote{This is the version we are interested in when studying $\HMb_*$, in particular when we study local systems.}

We say that the resulting homology group $H_{\textup{tw}}^*(M; x_\bullet)$, considered as a module over $\mathbb Z\llbracket T, T^{-1}]$, is the \textup{twisted cohomology} of $(M; x_\bullet)$. 
\end{definition}


\begin{remark}\label{twistnontors}
In proving Theorem \ref{thmmain} we will only need to take $A = M$, and our algebras will be various cubical and simplicial cochain algebras of a space (or simplicial/cubical set) $X$. In this case, we write the twisted cochain complex as $C_{\textup{tw}}^*(X; x_\bullet)$. To study the monopole Floer homology for a torsion spin$^c$ structure $\mathfrak{s}$ with coefficients in a local system $\Gamma_0$ on $\mathcal B^\sigma(Y, \mathfrak s)$, it is natural to choose $M$ to be the simplicial or cubical chains with coefficients in $\Gamma_0$, which we denote as $C_*^\Delta(X; \Gamma_0)$ and $C_*^{\cube}(X; \Gamma_0)$ respectively. Finally, when studying non-torsion spin$^c$ structures, we will replace $M \otimes \mathbb Z[T, T^{-1}]$ by a twisted version in order to take into account monodromies that act as multiplying by $T^n$. 
\end{remark}

\vspace{0.5cm}

It is readily checked that the differential above indeed squares to zero: 

\begin{align*}(d+L_3 T + \cdots)^2 m &= d^2 m + \sum_{n=1}^\infty \left(dL_{2n+1} + L_{2n+1} d + \sum_{i+j=n} L_{2i+1} L_{2j+1}\right)m T^n \\
&= \sum_{n=1}^\infty \left(d(x_{2n+1} m) + x_{2n+1} dm + \sum_{i+j=n} x_{2i+1} x_{2j+1} m \right) T^n \\
&= \sum_{n=1}^\infty \left((dx_{2n+1}) m   + \sum_{i+j = n} x_{2i+1} x_{2j+1} m\right) T^n \\
&= \sum_{n=1}^\infty \left(dx_{2n+1} + \sum_{i+j=n} x_{2i+1} x_{2j+1} \right) m T^n \\&= 0.
\end{align*}

The twisted differential has degree $+1$, so that the twisted cohomology group is naturally $\mathbb{Z}$-graded. Furthermore, it breaks up into a sum of terms $L_{2n+1} T^n$, where $L_{2n+1}$ increases the $M$-degree by $2n+1$. It is thus compatible with the filtration on $M\llbracket T, T^{-1}]$ so that $F_p$ is the set of Laurent series $\sum_{i \ge p} m_i T^i$, and gives rise to a spectral sequence $$H^*(M)\llbracket T, T^{-1}] \implies H^*_{\text{tw}}(M; x_\bullet)$$ whose $d^3$-differential is multiplication by $[x_3]T$. If $$x_3 = \dots = x_{2n-1} = 0$$ as chains, then the first possibly nonzero differential of this spectral sequence is $d^{2n+1}$, which is given by multiplication by $[x_{2n+1}] T^n$.

As before, when $M$ is bounded above in degree, Laurent series simplify here to being Laurent polynomials.

Because the filtration by $T$ is complete, this spectral sequence is strongly convergent, which means that $H^*_{\text{tw}}(M; x_\bullet)$ has associated graded group isomorphic to the $E_\infty$ page of the spectral sequence; furthermore, filtered maps between these complexes which induce an isomorphism on the $E_2$ page induce an isomorphism on homology. 

\par
\bigskip
The fundamental notion for our purposes is the following.

\begin{definition}
Consider a dg-algebra $A$. Two twisting sequences $x_\bullet$ and $y_\bullet$ in $A$ are \textit{homotopic} if there is a sequence $$h_{\bullet}=(h_{2n}) \in \prod_{n \geq 1} A^{2n}$$ so that the identity
\begin{equation*}\label{homotopytwist}
y_{2n+1} - x_{2n+1} + dh_{2n} + \sum_{i+j = n} (y_{2j+1} h_{2i} - h_{2i}x_{2j+1}) = 0
\end{equation*}
holds for all $n\geq 1$. We write $ts(A)$ for the set of homotopy classes of twisting sequences on $A$.
\end{definition}

\begin{remark}
It is not obvious from the definition that being homotopic is an equivalence relation on twisting sequences; this will be shown in the next subsection.
\end{remark}

The relevance of the notion of homotopic twisting sequences is the following.
\begin{lemma}\label{isom}
Consider two homotopic twisting sequences $x_\bullet$ and $y_\bullet$ in a dg-algebra $A$. Then for every dg module $M$ over $A$, the twisted cohomologies $H^*_{\tw}(M, x_{\bullet})$ and $H^*_{\tw}(M, y_{\bullet})$ are isomorphic.
\end{lemma}

\begin{proof}
Consider the map $$h^*_{\tw}: C^*_{\tw}(M; x_\bullet) \to C^*_{\tw}(M; y_\bullet)$$ given by $$h^*_{\tw}(m\otimes T^j) = m\otimes T^j + \sum_{n \geq 1} h_{2n}\cdot m\otimes T^{j+n}.$$ It is easily verified that $h^*_{\tw}$ is a chain map if and only if $h_\bullet$ is a homotopy of twisting sequences. Further, it is a filtered \emph{isomorphism}, because it is a filtered map (with respect to the complete filtration described above) whose associated graded map is the identity.
\end{proof}

\bigskip

Of course, the isomorphism between $H^*_{\tw}(M; x_{\bullet})$ and $H^*_{\tw}(M; y_{\bullet})$ might depend in general on the choice of homotopy $h_{\bullet}$. 

If we want to understand functoriality properties for twisted cohomology, we must use twisting sequences themselves (as opposed to twisting sequences considered up to homotopy), as follows.

Suppose one is given a dg-algebra homomorphism $f: A \to B$, and suppose that $M$ and $N$ are dg-modules over $A$ and $B$, respectively, either both homological or both cohomological. The definition of a module map depends on whether 

\begin{enumerate}[label=(\alph*)]
\item Cohomological module maps over $f$ are chain maps $g: M \to N$ with $g(am) = f(a) g(m)$.

\item Homological module maps over $f$ are chain maps $g: N \to M$ with $g(f(a) n) = a g(n)$.
\end{enumerate}

The key examples to have in mind are the following. Suppose we have a commutative diagram of simplicial complexes
\begin{center}
\begin{tikzcd}
    X \arrow{r}{p} \arrow{d} & Y \arrow{d} \\
    K \arrow{r}{q} & L
  \end{tikzcd}
\end{center}
If we consider $A=C^*_{\Delta}(L)$, $B=C^*_{\Delta}(K)$ and $f=q^*$, then the maps
\begin{align*}
p^*:C^*_{\Delta}(Y)\rightarrow C^*_{\Delta}(X)\\
p_*:C_*^{\Delta}(X)\rightarrow C_*^{\Delta}(Y)
\end{align*}
are cohomological and homological module maps, respectively. This language will be especially useful when dealing with local coefficients.
\\
\par
We have the following functoriality properties.

\begin{lemma}\label{typeab}
In the setting above, consider a twisting sequence $x_{\bullet}$ for $A$. Then:
\begin{enumerate}[label=(\alph*)]
\item If $g:M\rightarrow N$ is a cohomological module map, there is an induced map 
$$g_{\tw}: H^*_{\tw}(M; x_\bullet) \to H^*_{\tw}(N; f(x_\bullet)).$$ 
\item If $g:N\rightarrow M$ is a homological module map, there is an induced map $$g_{\tw}: H^*_{\tw}(N; f(x_\bullet)) \to H^*_{\tw}(M; x_\bullet).$$
\end{enumerate}
In both cases, if $g$ is a quasi-isomorphism, then $g_{\tw}$ is an isomorphism.
\end{lemma}

\begin{proof}
That such maps induce chain maps on the corresponding twisted chain complexes follows from the given formulas; these chain maps are filtered maps, which induce $g^*: H^*(M) \to H^*(N)$ and $g_*: H^*(N) \to H^*(M)$ on the $E_2$ page of the corresponding spectral sequence. Because the twisted cohomology spectral sequence converges, if $g^*$ (resp $g_*$) is an isomorphism, so is $g_{\text{tw}}$. 
\end{proof}

\begin{remark}
It may be surprising to some that our definition of twisting sequence starts in degree 3 rather than degree 1. In fact, this is essential for several reasons: 
\begin{itemize}
\item If $x_1 \ne 0$, then the $E^1$ page of the spectral sequence abutting to $H^*_{\tw}(M)$ is now $(M, d + L_1)$,and we do not have simple tools to compute its homology. 
\item In the next section we will use inductive arguments to show that quasi-isomorphisms induce bijections on homotopy classes of twisting sequences. If we allowed $x_1 \ne 0$, the base case in these inductive arguments fails, and indeed the claim is no longer true. 
\item The definition of characteristic classes we introduce in Section 4 would also fail to introduce even a $0$'th class $F_0(x_\bullet)$. 
\end{itemize}
\end{remark}

\begin{remark}\label{univ}
There is a space (or rather, a simplicial set) $\mathcal U_T$ so that maps $X \to \mathcal U_T$ of simplicial sets correspond to twisting sequences in $C^*_\Delta(X;\mathbb Z)$, and homotopies between maps correspond to homotopies between twisting sequences. In fact, taking $X = \Delta^n$, this gives a formula for the $n$-cycles of this simplicial set. Similar simplicial models are discussed in \cite[Section 1.5]{BM}.

One may think of twisted singular homology as being a parameterized homology theory over $\mathcal U_T$ in the sense of \cite[Section 20.1]{MS}.
One can show that this space has
\begin{equation}\label{htpy}
\pi_{i} \left(\mathcal U_T \right) =
\begin{cases}
\mathbb{Z}\quad\text{ if $i\geq3$ is odd;}\\
0\quad\text{ otherwise.}
\end{cases}
\end{equation}
Furthermore, one can interpret each relation $dx_{2n+1} + \sum_{i+j = n} x_{2i+1} x_{2j+1} = 0$ as giving the $k$-invariant for the next stage in the Postnikov tower of $ \mathcal U_T$. In particular, the first $k$-invariant is nonzero.

The computation in (\ref{htpy}) suggests a comparison to two spaces: the product $K = \prod_{n \geq 1} K(\mathbb Z, 2n+1)$ and the special unitary group $SU$. However, $\mathcal U_T$ is equivalent to neither: the first $k$-invariant of $\mathcal U_T$ is nonzero and its second $k$-invariant has order $3$, while the $k$-invariants of $K$ are trivial and the second $k$-invariant of $SU$ has order $6$. 
\end{remark} 

\vspace{0.5cm}
\subsection{Obstruction theory for twisting sequences and homotopies}
In what follows, we will need to know that given a quasi-isomorphism $f: A \to B$, the induced map $ts(f): ts(A) \to ts(B)$ on homotopy classes of twisting sequences is a bijection. 

In our argument we will want to extend partially defined twisting sequences, as well as extend partially defined homotopies. Before doing so we should define the \emph{obstruction classes} of partially defined twisting sequences and homotopies, and check that these classes are themselves well-defined up to homotopy.

\begin{definition}
A \emph{twisting $n$-sequence} in $A$ is a sequence $$(x_3, \cdots, x_{2n-1}) \in \prod_{i=1}^{n-1} A^{2i+1}$$ satisfying the relations $dx_{2m+1} + \sum_{i+j = m} x_{2i+1} x_{2j+1} = 0$ for all $m < n$. The \emph{obstruction class} to extending this to a twisting $(n+1)$-sequence is $$o_n(x_\bullet) = \left[\sum_{i+j=n} x_{2i+1} x_{2j+1}\right] \in H^{2n+2}(A).$$
An $n$-\emph{homotopy} $h_\bullet: x_\bullet \to y_\bullet$ between a pair of twisting $n$-sequences $x_\bullet$ and $y_\bullet$ is a sequence $(h_2, \cdots, h_{2n-2})$ with $$dh_{2m} + y_{2m+1} - x_{2m+1} + \sum_{\substack{i+j=m \\ i,j \geq 1}} y_{2i+1} h_{2j} - h_{2j} x_{2i+1} = 0$$ for all $m < n$. If $x_\bullet$ and $y_\bullet$ are extended to twisting $(n+1)$-sequences, the obstruction class to extending $h_\bullet$ to an $(n+1)$-homotopy is $$o_n(h_\bullet) = o_n(h_\bullet; x_\bullet, y_\bullet) = [y_{2n+1} - x_{2n+1} + \sum_{i+j=n} y_{2i+1} h_{2j} - h_{2j} x_{2i+1}] \in H^{2n+1}(A).$$
Finally, an $n$-\emph{modification} between two homotopies $h_\bullet, h'_\bullet: x_\bullet \to y_\bullet$ is a sequence $(z_1, \cdots, z_{2n-3})$ in $A$ with $$dz_{2m-1} + h_{2m} - h'_{2m} + \sum_{i+j=m} z_{2i-1} x_{2j+1} + y_{2j+1} z_{2i-1} = 0,$$ for all $m < n$. 
\end{definition}

The latter two concepts are nearly special cases of the first. An $n$-homotopy $h_\bullet$ between twisting $n$-sequences $x_\bullet$ and $y_\bullet$ is the same data as a twisting $n$-sequence $$\overline h_{2m+1} = e^0 \otimes y_{2m+1} + e^I \otimes h_{2m} + e^1 \otimes x_{2m+1}$$ on $I \otimes A$. Here $I$ is the algebra of simplicial cochains on the $1$-simplex, where we denote the generators by $e^0,e^1$ and $e^I$ respectively. Write $r_0: I \to \mathbb Z[0]$ for the map with $r_0(e^0) = 1$ and $r_0(e^I) = r_0(e^1) = 0$ and similarly for $r_1$; these maps are dg-algebra maps and quasi-isomorphisms, and $r = (r_0, r_1)$ is surjective onto $\mathbb Z^2$. The above construction amounts to saying that an $n$-homotopy $h_\bullet$ between twisting $n$-sequences is equivalent to the data of a twisting $n$-sequence $\overline h_\bullet$ on $I \otimes A$ with $r_0(\overline h_\bullet) = y_\bullet$ and $r_1(\overline h_\bullet) = x_\bullet$.



Given homotopies $h_\bullet$ and $h'_\bullet$, and passing to the associated twisting sequences $\overline h_\bullet$ and $\overline h'_\bullet$ on $I \otimes A$, the notion of a modification is precisely a \emph{relative} homotopy between these twisting sequences: a homotopy which vanishes on the boundary $\partial I \otimes A$. One may also view $z_\bullet$ as arising from the twisting sequence on $I \otimes I \otimes A$ given by $$\overline z_{2m+1} = e^0_1 e^0_2 y_{2m+1} + e^0_1 e^1_2 y_{2m+1} + e^1_1 e^0_2 x_{2m+1} + e^1_1 e^1_2 x_{2m+1} + e^I_1 e^0_2 h'_{2m} + e^I_1 e^1_2 h_{2m} + e^I_1 e^I_2 z_{2m-1}.$$

Write $i: \mathbb Z[0] \to I$ for the map $i(1) = e^0 + e^1$; we will denote the induced map $A\rightarrow I\times A$ by the same letter. One further useful perspective arising from the explicit formula here is that an $n$-modification between two homotopies $h,h': x_\bullet \to y_\bullet$ is the same data as a homotopy $\hat z_\bullet: i(x)_\bullet \to i(y)_\bullet$ so that $$r_0(\hat z_\bullet) = h'_\bullet: x_\bullet \to y_\bullet$$ and similarly $$r_1(\hat z_\bullet) = h_\bullet : x_\bullet \to y_\bullet.$$ That is, one may understand an $n$-modification as a relative homotopy between two homotopies.  

The definition of twisting $n$-sequence is exactly the same as that of a \emph{defining sequence for the Massey power} $\langle x_3\rangle^n$ as in \cite{Kra}; one says that $\langle x_3\rangle^n$ is the set of all cohomology classes produced as $o_n(x_\bullet)$, where $x_\bullet$ is a twisting $n$-sequence beginning with $x_3$. 

The fact that Massey powers are well-defined and vanish rationally --- and thus the obstruction to extending a twisting sequence is rationally always zero --- is one of the main inspirations for the results of this article. The obstruction class associated to a homotopy and to a modification are relative versions of these constructions. These, plus naturality properties of obstruction classes, give us the following result.

\begin{lemma}
If $x_\bullet$ and $y_\bullet$ are homotopic twisting $n$-sequences, then $o_n(x_\bullet) = o_n(y_\bullet)$. Similarly, if there exists a modification $z_\bullet$ between two $n$-homotopies $h_\bullet, h'_\bullet: x_\bullet \to y_\bullet$, the associated obstruction classes are equal: $$o_n(h_\bullet) = o_n(h'_\bullet).$$
\end{lemma}

\begin{proof}
It is clear that if $f: A \to B$ is a dg-algebra homomorphism and $x$ is a twisting $n$-sequence in $A$, then $o_n(f(x_\bullet)) = f_* o_n(x_\bullet)$. Now if $x$ and $y$ are $n$-homotopic by a homotopy $h$, consider the twisting $n$-sequence $\overline h$ on $I \otimes A$. By naturality, we have $$o_n(x_\bullet) = (r_1)_* o_n(\overline h_\bullet) = (r_0)_* o_n(\overline h_\bullet) = o_n(y_\bullet),$$ where the middle equality holds because $r_0$ and $r_1$ induce the same map on cohomology.

A similar naturality property holds for the obstruction class of a homotopy; we suppress bullets from notation for legibility. If $f: A \to B$ is a dg-algebra homomorphism and $h: x \to y$ is an $n$-homotopy in $A$, then $f(h): f(x) \to f(y)$ is an $n$-homotopy in $B$ and we have $$f_* o_n(h; x, y) = o_n(f(h); f(x), f(y)).$$ 
As discussed above, a modification $z$ gives a homotopy $\hat z_\bullet$ between $i(x)$ and $i(y)$. Then $$o_n(h; x, y) = (r_1)_* o(\hat z; i(x), i(y)) = (r_0)_* o(\hat z; i(x), i(y)) = o_n(h'; x, y).$$

The outer two equalities hold because $r_0 \hat z = h$ and $r_1 \hat z = h'$, while the maps $r_0 i = r_1 i$ are both equal to the identity. Using once more that $r_0$ and $r_1$ induce the same maps in homology we have $o_n(h; x, y) = o_n(h'; x, y)$ as desired. 
\end{proof}

These in hand, we can finally explain why homotopy classes of twisting sequences are quasi-isomorphism invariants.

\begin{prop}\label{qiso-transfer}
Let $f: A \to B$ be a dg-algebra homomorphism which induces an isomorphism on homology. Then the induced map on homotopy classes of twisting sequence $ts(A) \to ts(B)$ is a bijection. 
\end{prop}

\begin{proof}
We will prove both directions by induction. 

First we show surjectivity. Suppose $b_\bullet = (b_3, \cdots)$ is a twisting sequence in $B$. Our goal is to construct a twisting $n$-sequence $a_\bullet$ in $A$ and an $n$-homotopy $f(a_\bullet) \to b_\bullet$ via induction on $n$; doing so for all $n$ gives us the desired lift-up-to-homotopy.

For the base step, use that $f$ is a quasi-isomorphism to choose $a_3 \in A$ and $h_2 \in B$ with $dh_2 = f(a_3) - b_3$. 

Inductively, suppose we have a twisting $n$-sequence $a_\bullet = (a_3, \cdots, a_{2n-1})$ in $A$ and an $n$-homotopy $(h_2, \cdots, h_{2n-2})$ from $f(a_\bullet)$ to $b_\bullet$. The obstruction to extending $a_\bullet$ is given by $o(a_\bullet)$. We know $$f_* o(a_\bullet) = o(f(a_\bullet)) = o(b_\bullet),$$ because the two twisting $n$-sequences are homotopic. But $b_\bullet$ is a twisting sequence, defined for all $n$; there is no obstruction to extending it. Thus $f_* o(a_\bullet) = 0$, and because $f$ is a quasi-isomorphism $o(a_\bullet) = 0$. Thus we may choose an element $a'_{2n+1} \in A$ extending $a_\bullet$ to a twisting $(n+1)$-sequence. However, after choosing such an extension, it may be the case that the obstruction to extending the homotopy $$o(h_\bullet; a'_\bullet, b_\bullet) = [f(a'_{2n+1}) - b_{2n+1} + \sum_{i+j=n} h_{2j} f(a_{2i+1}) - b_{2i+1} h_{2j}]$$ is nonzero. If so, pick a cocycle $c \in A$ so that $f(c)$ is homologous to this obstruction class, and set $a_{2n+1} = a'_{2n+1} - c$. Because $c$ is a cocycle it is clear this is still a twisting $(n+1)$-sequence, and by the explicit formula for the obstruction class, the obstruction class has $$o(h_\bullet; a_\bullet, b_\bullet) = o(h_\bullet; a'_\bullet, b_\bullet) - [c] = 0,$$ so there is no obstruction to extending the homotopy. Choosing an appropriate $h_{2n}$, this completes the induction.

Injectivity falls to a similar argument: one supposes $x_\bullet$ and $y_\bullet$ are twisting sequences in $A$, that $b_\bullet$ is a homotopy between $f(x_\bullet)$ and $f(y_\bullet)$, and inductively constructs a homotopy $b'_\bullet$ from $x_\bullet$ to $y_\bullet$ and a modification $z_\bullet: f(b'_\bullet) \to b_\bullet$. The only novelty is that we use modifications between homotopies, instead of homotopies between twisting sequences.
\end{proof}

It is important to point out that even though $f$ induces a bijection on homotopy classes of twisting sequences, the bijection is not explicit, and it does not preserve the property of being representable by a twisting $3$-cycle $(x_3, 0, \cdots)$. In Section $2$ we will see an example of a zig-zag of quasi-isomorphisms so that a twisting $3$-cycle $(\zeta_3, 0, \cdots)$ with $[\zeta_3] = 0$ is transferred to a twisting $5$-cycle $(0, \zeta_5', 0, \cdots)$ with $[\zeta_5'] \ne 0$. This is a 2-torsion phenomenon, as $\zeta_5'$ is necessarily a 2-torsion class. In Section 3, our rational characteristic classes --- only valid for Hirsch algebras $A$ --- will be used to control this phenomenon.

\medskip

We conclude this section by reformulating the homotopy relation. As stated, it is not even clear that this relation is either reflexive or transitive: in the former case the issue is that there is no algebra map $I \to I$ which `swaps the endpoints' of the interval, and in the latter case if one has a pair of homotopies $h_1: x \to y$ and $h_2: y \to z$, these define a twisting sequence on the algebra $I_2 \otimes A$ (where $I_2$ is the algebra of cochains on the simplicial interval with two edges), but there is no clear way to induce from this a homotopy from $x$ to $z$.

To remedy these, we show that the notion of `homotopy' may be defined with respect to any algebra which behaves sufficiently well like cochains on the interval, and that this agrees with our original notion of homotopy. This will quickly show that homotopy of twisting sequences is an equivalence relation. It will be important later that, rationally, $I$ may be replaced with a \emph{commutative} algebra. First we describe a suitable class of algebras which can be used in place of $I$; in Remark \ref{commutativeintalg} we show that a commutative example exists over the rationals, which is used in the proof of Theorem \ref{charclass-facts}.

\begin{definition}
An \emph{interval algebra} $J$ is a torsion-free, non-negatively graded, unital dg-algebra equipped with two dg-algebra quasi-isomorphisms $r_0, r_1: J \to \mathbb Z[0]$ so that $$r = (r_0, r_1): J \to \mathbb Z^2$$ is \textit{surjective} and so that the $r_i$'s induce the same map in homology.
\par
A \emph{homomorphism of interval algebras} is a dg-algebra homomorphism $f: J \to J'$ so that $r'f = r$; a quasi-isomorphism of interval algebras is a homomorphism of interval algebras which induces an isomorphism in homology.
\par
Finally, given an interval algebra $J$ and a dg-algebra $A$, a \emph{$J$-homotopy} between twisting sequences $x_\bullet$ and $y_\bullet$ on $A$ is a twisting sequence $\overline h_\bullet$ on $J \otimes A$ so that $r_0(\overline h_\bullet) = y_\bullet$ and $r_1(\overline h_\bullet) = x_\bullet$. 

If $x_\bullet$ and $y_\bullet$ are $J$-homotopic, we write $x_\bullet \sim_J y_\bullet$.
\end{definition}

\begin{example}
The algebra $I = C^*_\Delta(\Delta^1)$ is the standard and simplest inteval algebra; an $I$-homotopy is simply a homotopy between twisting sequences.

The algebra $I_N = C^*_\Delta(\Delta^1_N)$, the simplicial cochain algebra of the simplicial interval with $N$ edges, is an interval algebra when equipped with $r_i$ the restriction maps to $\{i\}$. An $I_N$-homotopy between twisting sequences is a sequence of $N$ composable homotopies $x_\bullet \to x^1_\bullet \cdots \to x^{n-1}_\bullet \to y_\bullet$.

One may also take $I_{sing}$, the singular cochain algebra of the unit interval, or even the singular cochain algebra of any acyclic space equipped with two distinct points. An $I_{sing}$-homotopy is hard to describe in terms of $A$ itself.
\end{example}

Using the same ideas as the last argument, we can now prove that $J$-homotopy is independent of the choice of $J$. We do this in two steps: first we show that the notion of $J$-homotopy is independent of $J$ up to quasi-isomorphism, and then we give a zig-zag of interval algebra quasi-isomorphisms between any two interval algebras.

\begin{lemma}
Let $A$ be a dg-algebra. If $x_\bullet$ and $y_\bullet$ are twisting sequences in $A$, and $f: J \to J'$ is a quasi-isomorphism of interval algebras (so $r_i' f = r_i$), then $x_\bullet \sim_J y_\bullet$ if and only if $x_\bullet \sim_{J'} y_\bullet$.
\end{lemma}

\begin{proof}
If $x_\bullet \sim_J y_\bullet$ then (by definition) there exists a twisting sequence $\overline h_\bullet$ on $J \otimes A$ so that $r_0(\overline h_\bullet) = y_\bullet$ and $r_1(\overline h_\bullet) = x_\bullet$. Then because $r'f = r$, we see that $f(\overline h_\bullet)$ is a $J'$-homotopy between these twisting sequences, so that $x_\bullet \sim_{J'} y_\bullet$.

The other direction is more difficult. We are in the situation of the following diagram: \[\begin{tikzcd}
	{\overline h_\bullet?} & {J \otimes A} \\
	&&&& {(x_\bullet, y_\bullet)} \\
	{\overline h'_\bullet} & {J' \otimes A} &&& {A^2} \\
	&&&& {}
	\arrow["r_A"', from=1-2, to=3-5]
	\arrow["{r'_A}"', from=3-2, to=3-5]
	\arrow["{f \otimes 1}"', from=1-2, to=3-2]
	\arrow[hook, from=3-1, to=3-2]
	\arrow[dashed, hook, from=1-1, to=1-2]
	\arrow[hook', from=2-5, to=3-5]
\end{tikzcd}\]

We have twisting sequences $x_\bullet$ and $y_\bullet$ on $A$, a twisting sequence $\overline h'_\bullet$ which restricts to these on $J' \otimes A$, and our goal is to inductively construct a twisting sequence $\overline h_\bullet$ on $J \otimes A$ and a \emph{relative homotopy} $z_\bullet: f(\overline h_\bullet) \to \overline h'_\bullet$, as in the injectivity part of Proposition \ref{qiso-transfer}. 

Before we can carry out this induction, we need to mention some algebraic preliminaries. First, the map $f \otimes 1$ is a quasi-isomorphism because $J$ and $J'$ are both torsion-free, hence $\mathbb Z$-flat; this guarantees that $f \otimes 1$ is a quasi-isomorphism as soon as $f$ is, by a spectral sequence argument as in \cite[Part II, Lemma 2.2]{KrizMay} (the key point is that submodules of $\mathbb Z$-flat modules are $\mathbb Z$-flat).

Next, the map $f \otimes 1: \text{ker}(r_A) \to \text{ker}(r'_A)$ is also a quasi-isomorphism. This is because $r_A$ and $r'_A$ are surjective, so induce long exact triangles relating the cohomology of $\text{ker}(r_A), J \otimes A$, and $A^2$. The map $f$ induces a map between these long exact triangles; the map $A^2 \to A^2$ is the identity, and the map $J \otimes A \to J' \otimes A$ is $f \otimes 1$, both of which induce isomorphisms on cohomology. The claim now follows from the five lemma.
\par
We are now prepared to set up our induction. For the base case, arbitarily choose a cycle $\overline h_3 \in J \otimes A$ and a chain $z_2 \in J' \otimes A$ so that $r_A(\overline h_3) = (x_3, y_3)$, and $dz_2 = f(\overline h_3) - \overline h'_3$, while $r'_A(z_2) = 0$. Both of these steps require some justification.

\begin{itemize}
\item Because $f \otimes 1$ is a quasi-isomorphism and $r'_A(\overline h'_3) = (x_3, y_3)$, we see that there exists some cycle $H_3 \in J \otimes A$ so that $r(H_3)$ is homologous to $(x_3, y_3)$; let's say $r_A(H_3) + dw = (x_3, y_3)$. Now because $r_A$ is surjective, there exists some $W \in J \otimes A$ with $r_A(W) = w$. It follows that $\overline h_3 = H_3 + dW$ satifies the desired properties.\\
\item Now, $f(\overline h_3) - \overline h'_3$ is a cycle in $\text{ker}(r'_A)$; because $$f \otimes 1: \text{ker}(r_A) \to \text{ker}(r'_A)$$ is a quasi-isomorphism, we may find a cochain $\ell \in \text{ker}(r_A)$ so that $f(\overline h_3 + \ell) - \overline h'_3$ is a coboundary in $\text{ker}(r'_A)$. Thus replacing $\overline h_3$ with $\overline h_3 + \ell$ we may choose $z_2$ with the desired properties.
\end{itemize}
The induction step is similar. Supposing we have chosen a twisting $n$-sequence $\overline h_\bullet$ on $J \otimes A$ and a relative $n$-homotopy $f(\overline h_\bullet) \xrightarrow{z_\bullet} \overline h'_\bullet$. 

Then the obstruction to extending $\overline h_\bullet$ to a twisting $(n+1)$-sequence (which restricts to the desired twisting sequences via $r_0$ and $r_1$) is a relative cohomology class $o(\overline h_\bullet; r) \in H^*(\text{ker}(r_A))$. There are similar obstructions $o(f(\overline h_\bullet); r')$ and $o(\overline h'_\bullet; r')$ in $H^*(\text{ker}(r'_A))$. The latter obstruction vanishes (because this sequence is extendable), and the former obstruction is equal to the latter obstruction because the two twisting $n$-sequences $f(\overline h_\bullet)$ and $\overline h'_\bullet$ are homotopic relative to $r'$. Because $f$ is a quasi-isomorphism and $$f_* o(\overline h_\bullet; r) = o(f\overline h_\bullet; r') = o(\overline h'_\bullet; r') = 0,$$ we see that $o(\overline h_\bullet; r) = 0$, and we may extend our twisting $n$-sequence $\overline h_\bullet$ to a twisting $(n+1)$-sequence which still has the desired restrictions.

Now we must extend the homotopy. Just as before, the obstruction to doing so may be adjusted arbitrarily by changing the newly-constructed $\overline h_{2n+1}$ by adding a cocycle in $\text{ker}(r_A)$; doing so if necessary, we may extend the homotopy. This completes the induction.
\end{proof}

Now let $J, J'$ be two \emph{arbitrary} interval algebras. Then $J \otimes J'$ is again an interval algebra, with endpoint maps $r_0 \otimes r_0'$ and $r_1 \otimes r_1'$. Furthermore, the maps $i: J \to J \otimes J'$ and $i': J' \to J \otimes J'$, given by $i(x) = x \otimes 1$ and $i'(y) = 1 \otimes y$, are maps of interval algebras and quasi-isomorphisms by the same argument as for $f\otimes 1$ above.

It follows that given any interval algebra, there is a zig-zag of quasi-isomorphisms of interval algebras between the standard interval algebra $I = C^*_\Delta(\Delta^1)$ and $J$. This gives us the following statement.

\begin{cor}
Given two twisting sequences $x_\bullet$ and $y_\bullet$ in $A$, the following are equivalent.
\begin{itemize}
\item $x_\bullet$ and $y_\bullet$ are homotopic.
\item There exists an interval algebra $J$ so that $x_\bullet$ and $y_\bullet$ are $J$-homotopic.
\item For all interval algebras $J$, the twisting sequences $x_\bullet$ and $y_\bullet$ are $J$-homotopic.
\end{itemize}
\end{cor}

\begin{cor}
The homotopy relation $x_\bullet \sim y_\bullet$ on twisting sequences is in fact an equivalence relation.
\end{cor}

\begin{proof}
Reflexivity is clear (set $h_{2n} = 0$ for all $n$). 

For symmetry, choose an interval algebra $J$ for which there exists a dg-algebra automorphism $f: J \to J$ so that $r_1 f = r_0$ and $r_0 f = r_1$. Then if $x_\bullet \sim_J y_\bullet$ via a twisting sequence $\overline h_\bullet$ on $J \otimes A$, it follows that $y_\bullet \sim_J x_\bullet$ via the twisting sequence $f(\overline h_\bullet)$. As an example of such a $J$ one may take $C^*_{sing}([0,1])$, with $f$ induced by the continuous map $t \mapsto 1-t$.

For transitivity, observe that if $a_\bullet \xrightarrow{h_\bullet} b_\bullet \xrightarrow{j_\bullet} c_\bullet$ are a pair of homotopies between twisting sequences, then these define a twisting sequence $$\overline h_{2n+1} = e^0 c_{2n+1} + e^{1/2} b_{2n+1} + e^1 a_{2n+1} + e^{[0,1/2]} j_{2n} + e^{[1/2, 1]} h_{2n}$$ on the algebra $$I_2 = C^*_\Delta(\Delta^1_2),$$ where $\Delta^1_2$ denotes the simplicial structure on the unit interval with two edges and three vertices. This is again an interval algebra with $r_0$ given by restriction to $\{0\}$ and $r_1$ given by restriction to $\{1\}$, and $\overline h_\bullet$ has $r_0(\overline h_\bullet) = c_\bullet$ and $r_1(\overline h_\bullet) = a_\bullet$.

Thus if $a_\bullet \sim b_\bullet$ and $b_\bullet \sim c_\bullet$, then $a_\bullet \sim_{I_2} c_\bullet$ (essentially by definition). By the previous corollary it follows that $a_\bullet \sim c_\bullet$ as desired.
\end{proof}

\begin{remark}\label{commutativeintalg}
So far we have implicitly worked with algebras over the ground ring $\mathbb Z$. If one works over the rationals $\mathbb Q$, then in fact there is a \emph{commutative} interval algebra, given by the rational polynomial differential forms on an interval: $\mathbb Q[t,dt]/(dt)^2$ with $d(t^n) = nt^{n-1}$ and $|t| = 1$, with restriction maps given by restricting these differential forms to $\{0\}$ and $\{1\}$ respectively (so that $r_0 (p(t) + q(t) dt) = p(0)$ gives the constant term of the 0-form term, and $r_1(p(t) + q(t) dt) = p(1)$ sums the coefficients of the 0-form term). 

This construction does not work integrally: the algebra $\mathbb Z[t,dt]/(dt)^2$ this not acyclic, and if one takes a divided power algebra $\mathbb Z[t, t^2/2, t^3/6, \cdots, dt]/(dt^2)$ then the second evaluation map lands in the rationals, not the integers. It seems unlikely to the authors that there is a commutative interval algebra over the integers.
\end{remark}

We conclude with a technical lemma which will be useful later.

\begin{lemma}\label{lemma:add-a-cbdry}
Let $A$ be a dg-algebra. Given a twisting sequence $x_\bullet$ in $A$ and a coboundary $dz = a \in A^{2s+1}$, there is a twisting sequence $x'_\bullet$ homotopic to $x_\bullet$ with $x_{2i+1} = x'_{2i+1}$ for $i < s$, while $x_{2s+1} = x'_{2s+1} + a$.
\end{lemma}

\begin{proof}
Construct $x'_\bullet$ and the homotopy by induction. For the base case, we have $x'_\bullet$ in degrees up to $x_{2s+1}$, and may choose $h_{2i} = 0$ for $i < s$ and $h_{2s} = z$ (so that $dh_{2s} = x_{2s+1} - x'_{2s+1}$). 

Inductively, we have an $m$-sequence $x'_\bullet$ with the desired properties and an $m$-homotopy $h_\bullet$ from $x_\bullet$ to $x'_\bullet$, and we want to extend these to $m+1$. The obstruction theory argument is now exactly the same as in Proposition \ref{qiso-transfer}: the obstruction to extending $x'_\bullet$ is identified with the obstruction to extending $x_\bullet$, and hence is zero; choosing $x'_{2m+1}$ the obstruction to extending $h_\bullet$ is possibly nonzero, but may be made zero by adjusting $x'_{2m+1}$ by a cocycle if necessary.
\end{proof}

\medskip

\section{Some higher differentials in $SU(2)$-coupled Morse homology}\label{nonformalsu2}
As a brief aside, in this section we answer a question of Kronheimer and Mrowka by showing that in general, the integral coupled Morse homology for a family $(M,L)$ classified by $\zeta: M\rightarrow SU(2)$ is not determined by the cohomology class $\zeta^*[SU(2)]\in H^3(M;\mathbb{Z})$. 

Explicitly, we show that there exists such a map so that $\zeta^* [SU(2)] = 0$, but the $d^5$ differential on the twisted homology spectral sequence (equivalently, the coupled Morse homology spectral sequence) is nonzero. This will also give us an example of a phenomenon discussed in the previous section: the existence of a twisting sequence $(\zeta_3, 0, 0, \cdots)$ in an algebra $B$ so that $[\zeta_3] = 0$, but so that there exists a quasi-isomorphism $f: A \to B$ so that this twisting sequence is pulled back to one of the form $(0, \zeta'_5, 0, \cdots)$, where $[\zeta'_5]$ is nonzero.

In proving Theorem \ref{cohono}, we will interpret the differential on the $E^5$ page in terms of a mod $2$ Hopf invariant for the map $M\rightarrow SU(2)$. This Hopf invariant is precisely the cohomology class $[\zeta'_5]$ discussed above.

We will mostly work in the simplicial model, and comment on the equivalence with the Morse model at the end of the section. In this section all (co)homology is taken with $\Bbb Z$ coefficients unless specfied otherwise.

\begin{definition}
Let $X$ be a simplicial complex equipped with a simplicial map $$\phi: X \to SU(2)=S^3$$ which has $\phi^* = 0$ on third cohomology; write $C_\phi$ for the mapping cone of $\phi$. Let $x_\phi$ be any class in $H^3(C_\phi)$ so that under the map $i: S^3 \to C_\phi$ we have $$i^*x_\phi = 1 \in \mathbb Z \cong H^3(S^3).$$ The graded-commutativity of cup the product implies that $x_{\phi}^2$ is $2$-torsion. Consider the composite isomorphism $$H^6(C_\phi) \xrightarrow{(p^*)^{-1}} H^6(\Sigma X) \xrightarrow{S} H^5(X)$$
where $p: C_\phi \to \Sigma X$ is the collapse map and $S$ is the suspension isomorphism. We define the \textit{mod $2$ Hopf invariant} $h(\phi)$ to be the 2-torsion class $S(p^* x_\phi^2)\in H^5(X;\mathbb Z)$.
\end{definition}

The Hopf invariant $h(\phi)$ is readily seen to be an invariant of the homotopy class of $\phi$, because $x_\phi^2$ is independent of the choice of lift $x_\phi$. Indeed, if $i^* y_\phi = i^* x_\phi$ then $$y_\phi^2 - x_\phi^2 = (y_\phi - x_\phi)(y_\phi + x_\phi).$$ Now $y_\phi - x_\phi$ represents an element of $H^3(C_\phi, S^3) \cong H^3(\Sigma X)$, and all cup products $$H^3(C_\phi, S^3) \times H^3(C_\phi) \to H^6(C_\phi, S^3)$$ are zero (the usual proof that cup products of a suspension vanish applies), hence the cohomology class $y_\phi^2 - x_\phi^2$ is zero.

\medskip
Write $x_3^0$ for a simplicial cocycle on $S^3$ giving the oriented generator of third cohomology. Because $C^6_\Delta(S^3) = 0$, we tautologically have that $x_3^0$ is a twisting 3-cycle. We write $\phi_3$ for the twisting $3$-cycle $\phi^* x_3$ on $X$.

\begin{lemma}\label{hopf}Let $X$ be a finite-dimensional simplicial complex with a simplicial map $\phi: X \to S^3$ so that the induced map $\phi^*$ is zero on third cohomology. Then the spectral sequence for $H^*_{\textup{tw}}(X;\phi_3)$ has differential $d_3 = 0$, but differential $d_5$ given by $$d_5([a] \otimes 1) = h(\phi) \cup [a] \otimes T^2 \in H^5(X;\mathbb Z) \otimes \mathbb Z[T, T^{-1}],$$so that in particular
$d_5(1) = h(\phi) T^2$.
\end{lemma}
\begin{remark} 
One should understand this Hopf invariant $h(\phi)$ as being the second obstruction to null-homotoping a twisting sequence $(\phi_3, 0, 0, \cdots)$ with $[\phi_3] = 0$.
It follows that $h(\phi)$ depends on the map $\phi:X\rightarrow S^3$ only through the homotopy class of the twisting $3$-cycle $\phi_3$ (or, equivalently, the homotopy class of the composite map to the universal space $X \to S^3 \to \mathcal U_T$, see Remark \ref{univ}). 
\end{remark}

\begin{proof}
The differential $d_3$ is zero by the assumption that $[\phi_3] = 0$. This means that there is a cochain $h_2$ with $dh_2 + \phi_3 = 0$. Following the notation of \cite[Theorem 2.6]{McC}, the class $$[a] \otimes 1 \in H^*(X) \otimes \mathbb Z[T, T^{-1}] = E^{*,*}_5$$ is represented by the cochain $a \otimes 1 + h_2 a \otimes T$. The differential $d + L_3 T$, applied to this chain, gives $\phi_3 h_2 a \otimes T^2$. It follows that $$d_5([a] \otimes 1) = [\phi_3 h_2 a] \otimes T^2.$$

We will identify the cohomology class $[\phi_3 h_2]$ with the Hopf invariant $h(\phi)$. Notice that $C_\phi$ is naturally a simplicial set (though not a simplicial complex if $\phi$ is not injective). For this reason, we will work with the subcomplex $C_*$ of the singular chain complex $C_*^{\mathrm{sing}}(C_\phi)$ consisting of simplices on $S^3$ and linear cones on simplices in $X$, together with an extra vertex to serve as the cone point.\footnote{Equivalently, $C_*$ is the normalized simplicial chain complex $C_*^\Delta(C_\phi)$ on the simplicial mapping cone.} The linear cone is ordered so that the cone point is the last point in the simplex, while the ordering on the earlier vertices coincides with that of $X$. Then, if $\sigma$ is one of the simplices listed above, for every $i$ the front face $\sigma_{[0, i]}$ and the back face $\sigma_{[i, k]}$ are simplices in the list too.

It follows that there is a well-defined cup-product on the dual $C^* = \text{Hom}(C_*, \Bbb Z)$, so that the restriction map $C^*_{\mathrm{sing}}(C_\phi) \to C^*$ is a dg-algebra quasi-isomomorphism. We may thus compute our $x_\phi^2$ in
$$C^* \cong C^{*-1}_\Delta(X) \oplus C^*_\Delta(S^3),$$where the differential is the mapping cone differential  $$\overline d \sigma = \begin{pmatrix} d_X \sigma & (-1)^{|\sigma| + 1}\phi^*\sigma \\ 0 & d_{S^3}\sigma\end{pmatrix}$$ and the cup product is $$(Cx, y) \cdot (Cx', y') = (C(\phi^*y \cup_X x'), y \cup_{S^3} y').$$ We use our chosen cochain with $dh_2 + \phi_3 = 0$, and set $x_\phi = (C(h_2), x^0_3)$; this is a cocycle such that $i^* x_\phi = 1 \in H^3 (S^3)$. Then $x_\phi^2 = (C(\phi_3 h_2), 0)$. The Hopf invariant is obtained by pulling this back under the maps $C^5(X) \to C^6(\Sigma X) \to H^6(C_\phi)$; doing so, we obtain the desired result that $[\phi_3 h_2] = h(\phi)$. 
\end{proof}

Using this, we can provide an example of the phenomenon describe in Theorem \ref{cohono} in the category of simplicial complexes.

\begin{prop}
The 5-dimensional complex $X = \Sigma^3 \mathbb{RP}^2$ has a map $\phi: X \to S^3$ so that $h(\phi) \neq 0$. In particular, as $$H^*(X; \mathbb Z) = \mathbb Z[0] \oplus (\mathbb Z/2)[5],$$ from the spectral sequence for twisted cohomology it follows that $$H^*_{\textup{tw}}(X;\phi_3) \cong \mathbb Z[T,T^{-1}] \not\cong H^*(X;\mathbb Z)[T,T^{-1}]\cong H^*_{\textup{tw}}(X;0),$$ even though $[\phi_3]=0$.
\end{prop}
\begin{proof}
Consider the set of homotopy classes $[X, S^3]$, which we write as $\pi^3(X)$. This set has a group structure, coming from the group structure on $S^3=SU(2)$; maps $X \to Y$ induce group homomorphisms between the mapping sets $\pi^3(Y) \to \pi^3(X)$. Whenever $X$ is a suspension this group is abelian.

Thinking of $\mathbb{RP}^2$ as the mapping cone of $z^2: S^1 \to S^1$, we get a long cofibration sequence $$S^1 \to S^1 \to \mathbb{RP}^2 \to S^2 \to S^2 \to \Sigma\mathbb{RP}^2 \to \cdots$$ where the first map is the squaring map of degree 2, and all further maps $S^n \to S^n$ are suspensions of this (so also of degree 2).This gives rise to an exact sequence of mapping sets $$\cdots\rightarrow \pi^3(S^5) \to \pi^3(S^5) \to \pi^3(\Sigma^3 \mathbb{RP}^2) \to \pi^3(S^4) \to \pi^3(S^4) \to \cdots$$
The groups $\pi^3(S^4) = \pi_4(S^3)$ and $\pi^3(S^5) = \pi_5(S^3)$ are both isomorphic to $\mathbb Z/2$. The outer maps are the maps induced by the map $S^n \to S^n$ of degree 2, hence induce multiplication by 2 (so, zero) on the above homotopy groups. We thus get a short exact sequence $$0 \to \mathbb Z/2 \to \pi^3(X) \to \mathbb Z/2 \to 0,$$ so that in particular there exists a $\phi \in \pi^3(X)$ which restricts to the suspension of the Hopf map $\Sigma \eta = f: S^4 \to S^3$.

Now the mapping cone $C_\phi$ has a CW structure with a single cell of each dimension $0, 3, 5, 6$. The 5-cell is attached along $\Sigma \eta = \phi|_{S^4}\in \pi_4(S^3)$; in particular, the 5-skeleton is homotopy equivalent to $\Sigma \mathbb{CP}^2$. By the suspension invariance and naturality of Steenrod squares, it follows that $$\text{Sq}^2: H^3(C_\phi;\mathbb Z/2) \to H^5(C_\phi; \mathbb Z/2)$$ is an isomorphism (see for example \cite[Section $4.L$]{Hat}). 

Further, there is a collapse map $C_\phi \to \Sigma^4 \mathbb{RP}^2$ which is an isomorphism on the cohomology groups of degree $5$ and $6$ (regardless of coefficients); because the integral Bockstein $\overline \beta: H^1(\mathbb{RP}^2;\mathbb Z/2) \to H^2(\mathbb{RP}^2; \mathbb Z)$ is an isomorphism, the same is true for the integral Bockstein $$\overline \beta:H^5(C_\phi;\mathbb Z/2)\rightarrow H^6(C_\phi;\mathbb Z).$$ Thus the integral Steenrod square $$\overline{\text{Sq}}^3 = \overline \beta\circ \text{Sq}^2\circ r:H^3(C_\phi;\mathbb{Z})\rightarrow H^6(C_\phi;\mathbb{Z}),$$ where $r$ is the reduction mod $2$, is given by $$H^3(C_\phi;\mathbb{Z}) \cong \mathbb Z \xrightarrow{\!\!\!\!\!\!\mod 2} \mathbb Z/2 \cong H^6(C_\phi;\mathbb{Z}).$$
Thus if $x_\phi \in H^3(C_\phi)$ has $i^*(x_\phi) = 1 \in \mathbb Z \cong H^3(S^3)$, we have $$x_\phi^2 = \overline{\text{Sq}}^3 x_\phi \neq 0 \in H^6(C_\phi;\mathbb Z),$$ as $x_\phi$ is a degree 3 class and $\overline{\text{Sq}}^{2n+1}(x) = x^2$ for any class $x$ of degree $2n+1$. To see this, write $\overline x$ for the cocycle $x$ taken mod 2. Then by Steenrod's original definition via cup-$i$ products (see for example \cite[Chapter $2$]{MT}), when $x$ is degree $2n+1$, we have $\text{Sq}^{2n} [\overline x] = [\overline x \cup_1 \overline x]$. To calculate the integral Bockstein, observe that we have an integral lift given by $x \cup_1 x$ and that $d(x \cup_1 x) = -2x^2$; the Bockstein is half the boundary of an integral lift, and hence $$\overline{\text{Sq}}^{2n+1} x := \overline \beta \text{Sq}^{2n} \overline x = -x^2 = x^2,$$ as $x^2$ is a 2-torsion cohomology class. Finally, because the mod-2 Hopf invariant is given by the image of $x_\phi^2$ under the inverse of the isomorphisms $$H^5(X) \to H^6(\Sigma X) \to H^6(C_\phi),$$ it follows that $h(\phi) \neq 0$.
\end{proof}

\begin{remark}
The example above is the minimal possible example because a twisting sequence on a 4-dimensional simplicial complex is zero in degrees 5 and above for degree reasons, hence null-homotopic if and only if $[x_3] = 0$. It was found with the observation that a map $X \to S^3$ which is trivial in cohomology factors through the homotopy fiber $\tau_{\geq 4} S^3$ of the map $S^3 \to K(\mathbb Z, 3)$ picking out a generator of its top cohomology. The 5-skeleton of a minimal cell structure on the space $\tau_{\geq 4} S^3$ is precisely $\Sigma^3 \mathbb{RP}^2$.
\end{remark}

To see the connection with twisting sequences, let us point out the following.
\begin{prop}\label{qisoRP3}
There exists a zig-zag of dga quasi-isomorphisms between $\widetilde C^*_\Delta(\Sigma^3 \mathbb{RP}^2)$ and its homology $(\mathbb Z/2)[5]$.
\end{prop}
Now, this zig-zag must transfer $(\zeta_3, 0, \cdots)$ to some $(0, \zeta_5', 0, \cdots)$ (because the only degree which could possibly be nonzero is degree 5). Because the twisted cohomology with respect to this twisting sequence disagrees with the untwisted homology, it follows that $[\zeta_5'] \ne 0$. This is well-defined up to homotopy because in $(\mathbb Z/2)[5]$ there is no differential. We thus have proved the following.
\begin{cor}
There is a zig-zag of dga quasi-isomorphisms for which the homotopy class of some twisting $3$-cycle does not correspond (under the natural bijection) to a class represented by a twisting $3$-cycle.
\end{cor}

\begin{proof}[Proof of Proposition \ref{qisoRP3}] We begin by making some simplifications to the cochain complex of a suspension. In the following, if $A$ is a nonnegatively-graded dg-algebra, write $\Sigma A$ for the algebra with $(\Sigma A)^n = A^{n-1}$, differential, $d_{\Sigma A} = -d_A$, and identically zero product.

For any simplicial set $X$ there exists a dg-algebra quasi-isomorphism $$\Sigma^*: \widetilde C^*_\Delta(\Sigma X) \to \Sigma\widetilde C^*_\Delta(X),$$ where the latter is equipped with trivial product and negative the differential on $C^*_\Delta(X)$. If one prefers to work with unital algebras, there is a corresponding homomorphism $C^*_\Delta(\Sigma X) \to \mathbb Z[0] \oplus \Sigma C^*_\Delta(X)$, with trivial product on the codomain except that $1 \in \mathbb Z[0]$ acts as a unit.

If $X$ is a simplicial set and $I = [0,1]$ the 1-simplex, give $I \times X$ the product simplicial structure and $x \in X_0$ a chosen 0-simplex. Then the quotient $\Sigma X$ of $[0, 1] \times X$, obtained by collapsing $\{0, 1\} \times X \cup I \times \{x\}$ to a point, is again naturally a simplicial set.

Consider the composite map $$\widetilde C^*_\Delta(\Sigma X) \xrightarrow{p^*} C^*_\Delta(I \times X, \partial I \times X \cup I \times \{x\}) \xrightarrow{EZ^*} C^*_\Delta(I, \partial I) \otimes C^*_\Delta(X, x) \cong \Sigma C^*_\Delta(X, x) \cong \widetilde \Sigma C^*_\Delta(X).$$ The first map is induced by the simplicial map of pairs $$p: (I \times X, \partial I \times X \cup I \times \{x\}) \to (\Sigma X, \Sigma x),$$ hence is a dg-algebra map; in fact, it is a dg-algebra isomorphism (not merely quasi-isomorphism). The second is the dual of the Eilenberg-Zilber map (sending $\Delta^i \otimes \Delta^j$ to a triangulation of $\Delta^i \times \Delta^j$). Giving the tensor product the differential $d(a \otimes b) = da \otimes b + (-1)^{|a|} a \otimes db$ and product $$(a \otimes b)(a' \otimes b') = (-1)^{|a'| |b|} aa' \otimes bb',$$ $EZ^*$ is a dg-algebra map and a quasi-isomorphism, as computed in \cite[Assertion (17.6)]{EM}. 

Now $C^*_\Delta(I, \partial I) \cong \mathbb Z[1]$ is a 1-dimension algebra with generator $e^I$ in degree 1 and trivial product. The second-to-last isomorphism simplifies $\mathbb Z[1] \otimes A = \Sigma A$, and the final isomorphism is simply $C^*(X, x) \cong \widetilde C^*(X)$. Every map given above is a dg-algebra quasi-isomorphism, hence the composite is as well.

To prove the result, set $A = \widetilde C^*_\Delta(\Sigma^3 \mathbb{RP}^2)$. By the above discussion, there exists a dg-algebra quasi-isomorphism $\widetilde C^*_\Delta(\Sigma^3 \mathbb{RP}^2) \to \Sigma \widetilde C^*_\Delta(\Sigma^2 \mathbb{RP}^2)$, where the codomain is equipped with trivial product. Now $\Sigma^2 \mathbb{RP}^2$ has reduced homology $\mathbb Z/2$ in degree 3 and zero otherwise. Choose a simplicial generator $x_3$ and an element $x_4$ with $dx_4 = -2x_3$. Shifting up one in degree, write $A_R = \mathbb Z[4] \oplus \mathbb Z[5]$ with differential $dy_5 = 2y_4$ and trivial producut; there is a dg-algebra quasi-isomorphism $A_R \to \Sigma \widetilde C^*_\Delta(\Sigma^2 \mathbb{RP}^2)$, given by sending $y_4$ to $x_3[1]$ and $y_5$ to $x_4[1]$. 

Finally, there is a dg-algebra quasi-isomorphism $A_R \to (\mathbb Z/2)[5]$. We have thus constructed a zig-zag of quasi-isomorphisms between $\widetilde C^*_\Delta(\Sigma^3 \mathbb{RP}^2)$ and its homology $(\mathbb Z/2)[5]$; passing to unital versions of the previous constructions, we have a zig-zag of dg-algebra quasi-isomorphisms from $C^*(\Sigma^3 \mathbb{RP}^2)$ to $\mathbb Z[0] \oplus (\mathbb Z/2)[5] = H^*(\Sigma^3 \mathbb{RP}^2)$.
\end{proof}

\begin{proof}[Proof of Theorem \ref{cohono}]
Going back to coupled Morse theory, we can consider $M_0$ to be a smooth manifold with boundary, homotopy equivalent to $\Sigma^3\mathbb{R}P^2$ (e.g. the regular neighborhood of some embedding $\Sigma^3 \mathbb{RP}P^2 \hookrightarrow \mathbb{R}^n$ for $n \ge 11$); the definitions and constructions in \cite{KM} readily generalize to the case of manifolds with boundary, and provide examples of the desired phenomenon in this setting.
\par
To obtain a closed example, we can simply take the double manifold $DM_0$, using the fact that the map $\phi:M_0\rightarrow SU(2)$ naturally extends to $\overline \phi: DM_0\rightarrow SU(2)$. It follows quickly from the cohomological description of the Hopf invariant above that it is natural under pullback: if $f: X \to S^3$ is zero in cohomology and $g: Y \to X$ is a map, then $h(fg) = g^* h(f)$. Because we have an inclusion $i: \Sigma^3 \mathbb{RP}^2 \to DM_0$ and $\phi = \overline \phi\circ i$, it follows that $i^* h(\overline \phi) = h(\phi)$ is nontrivial, so in particular $h(\overline \phi)$ is nontrivial. It follows from Lemma \ref{hopf} that the $E_5$ page of the twisted cohomology spectral sequence for $(DM_0, \overline \phi)$ has nonzero differential.
\par
Notice that the above examples show that the coupled Morse \emph{cohomology} spectral sequence has a nontrivial differential even though $d_3 = 0$. But the smooth manifold $(DM_0, \overline \phi)$ also provides an example where the homology spectral sequence has a nontrivial higher differential but $d^3 = 0$, because the Poincar\'e duality map $C^*(DM_0) \to C_*(DM_0)$ given by $x \mapsto x \cap [DM_0]$ is a quasi-isomorphism (by Poincar\'e duality) and a module map (by elementary properties of the cap product). It follows that the twisted (co)homology spectral sequences also satisfy Poincar\'e duality, and in particular, $d^5$ is also nonvanishing in the twisted homology spectral sequence.
\end{proof}

\bigskip

\section{Hirsch algebras and higher structure on twisting sequences}\label{sec:hirsch-def}
Given a dga map $f: A \to B$, we will want to understand the induced map $ts(f): ts(A) \to ts(B)$ on homotopy classes of twisting sequences. One traditional way to understand such maps is to produce \emph{characteristic classes}, elements $F_n(x_\bullet)$ in cohomology associated to each homotopy class of twisting sequence which are natural under dga maps.

One of these is easy to produce; because $dx_3 = 0$, we may take $F_1(x_\bullet) = [x_3]$. This is certainly natural for dga maps, as $F_1(f(x_\bullet)) = [f(x_3)] = f_* [x_3]$. In degree 5, this is not so easy; now $dx_5 = -x_3^2$, and to produce some natural \emph{cocycle} we would need a canonical element $e(x_3)$ with $de(x_3) = -x_3^2$. 

Suppose $A$ is \emph{homotopy commutative}; this means that $A$ is equipped with an operator $\cup_1$ which gives a null-homotopy of the graded commutator $[x,y] = xy - (-1)^{|x||y|} yx$. The original such product was Steenrod's product on simplicial cochains \cite{Ste}. Given such an operation, we have $d(x_3 \cup_1 x_3) = 2x_3^2$, and one may take (with rational coefficients) $e(x_3) = -\frac 12 x_3 \cup_1 x_3$.

To extend this to higher degrees one needs a compatibility relation between the cup-1 product and the product called a \emph{Hirsch formula}. A left Hirsch formula holds for simplicial cochains \cite{Hir}, but the corresponding right Hirsch formula fails. On the other hand, neither Hirsch formula holds for cubical cochains. A suitable setting for higher characteristic classes is given by \emph{Hirsch algebras}, which include a cup-1 product and coherent homotopies correcting for the failure of the Hirsch formula.

\subsection{Definitions}
The following definition is from \cite{Saneblidze}.

\begin{definition}\label{def:Hirsch}
A \emph{Hirsch algebra} is an associative differential-graded algebra $A$ equipped with the additional structure of maps $$E_{p,q}: A^{\otimes p} \otimes A^{\otimes q} \to A$$ of degree $1-p-q$ for each $p,q \ge 0$ with $p+q \ge 1$; we demand that $E_{1,0} = E_{0,1} = \textup{Id}$, while $E_{0,p} = E_{p,0} = 0$ for $p > 1$, and we also demand that the following identities hold:
\begin{align*}
dE_{p,q}(a_1,\dots,a_p;b_1,\dots, b_q)&=\sum_{1\leq i\leq p}(-1)^{\epsilon^{a}_{i-1}}E_{p,q}(a_1,\dots,da_i,\dots,a_p;b_1,\dots,b_q)\\
&+\sum_{1\leq j\leq q}(-1)^{\epsilon^{a}_{p}+\epsilon^b_{j-1}}E_{p,q}(a_1,\dots,a_p;b_1,\dots,db_j,\dots,b_q)\\
&+\sum_{1\leq i<p }(-1)^{\epsilon^{a}_{i}}E_{p-1,q}(a_1,\dots,a_ia_{i+1},\dots,a_p;b_1,\dots,b_q)\\
&+\sum_{1\leq j< q}(-1)^{\epsilon^{a}_{p}+\epsilon^b_{j}}E_{p,q-1}(a_1,\dots,a_p;b_1,\dots,b_jb_{j+1},\dots,b_q)\\
+\sum_{\substack{0\leq i\leq p\\ 0\leq j\leq q\\(i,j)\neq(0,0)}}(-1)^{\epsilon_{i,j}}&E_{i,j}(a_1,\dots, a_i;b_1,\dots, b_j)\cdot E_{p-i,q-j}(a_{i+1},\dots, a_p;b_{j+1},\dots, b_q)
\end{align*}
where $\epsilon^x_i=|x_1|+\dots +|x_i|+i$ and $\epsilon_{i,j}=\epsilon^a_i+\epsilon^b_j+(\epsilon^a_i+\epsilon^a_p)\epsilon^b_j+1$, where all appearance of $\epsilon_i^a$ and $\epsilon_j^b$ above refer to the strings $(a_1, \cdots, a_p)$ and $(b_1, \cdots, b_q)$. Here $|a|$ refers to the degree in $A$, not the degree in $A[-1]$.\footnote{If one instead preferred to write this in terms of the degree of $[a]$ as an element of $A[-1]$, given by $|a|_{-1} = |a| - 1$, one could simplify the expression to $\epsilon^a_i = |a_1|_{-1} + \cdots + |a_i|_{-1}$.}
\end{definition}

The two main examples are the simplicial and cubical cochain algebras. We describe explicitly these operations in Section \ref{examplesHirsch} below using the operadic technology of \cite{MM1}; the reader might find enlightening to get familiar with these before looking at the discussion of twisting sequences later in this section.

There are two useful ways of understanding this structure. First, it asserts that the product on $A$ is \textit{homotopy commutative} in a way which is homotopy-coherently associative.  The operation $E_{1,1}$ satisfies
\begin{equation}\label{E11id}
dE_{1,1}(a;b)-E_{1,1}(da;b)+(-1)^{|a|}E_{1,1}(a;db)=(-1)^{|a|}ab-(-1)^{|a|(|b|+1)}ba
\end{equation}
hence it defines a homotopy between the two maps $A \otimes A \to A$ given by $(a,b) \mapsto ab$ and $(a,b) \mapsto (-1)^{|a| |b|} ba$. In the simplicial world, $E_{1,1}$ behaves up to an overall sign (depending on the grading of the entries) as Steenrod's cup-1 product. Steenrod's cup-1 product satisfies a useful additional property, the \emph{left Hirsch formula} \cite{Hir}: $$(ab) \cup_1 c = a(b \cup_1 c) + (-1)^{|b|(|c|+1)} (a \cup_1 c)b.$$ This formula does not hold for $E_{1,1}$ in general, and in particular fails in the setting of cubical cochains, which we will need below. 

However, $E_{2,1}$ is a homotopy between the two sides of the given equation, while $E_{1,2}$ is a homotopy which demonstrates that a similar \emph{right Hirsch formula} (which does not hold in the simplicial setting) at least holds up to homotopy. Very explicitly, we have the homotopy-Hirsch formulas
\begin{align*}
dE_{2,1}(a,b;c)&=E_{2,1}(da,b;c)-(-1)^{|a|}E_{2,1}(a,db;c)+(-1)^{|a|+|b|}E_{2,1}(a,b;dc)\\
&-(-1)^{|a|}E_{1,1}(ab;c)+(-1)^{|a|+|b|(|c|+1)}E_{1,1}(a;c)b+(-1)^{|a|}aE_{1,1}(b;c)
\end{align*}
and
\begin{align*}
dE_{1,2}(a;b,c)&=E_{1,2}(da;b,c)-(-1)^{|a|}E_{1,2}(a;db,c)+(-1)^{|a|+|b|}E_{1,2}(a;b,dc)\\
&+(-1)^{|a|+|b|}E_{1,1}(a;bc)-(-1)^{|a|+|b|}E_{1,1}(a;b)c-(-1)^{|a|(|b|+1)}bE_{1,1}(a;c).
\end{align*}
The operations $E_{p,q}$ encode the higher homotopy-associativity of these operations.
\\
\par
A second more algebraic perspective on Hirsch algebras will be very useful to us below: it is the structure of a dg-bialgebra (to be defined below) on the bar construction $BA$. We shall make this explicit. 

Given a graded abelian group $A$, the \emph{bar construction} is given as $$BA = \bigoplus_{n \ge 0} A[-1]^{\otimes n}.$$ We write a generic element of this space as $[a_1 | \cdots | a_n]$, and write $1 = []$ for the empty string. The bar construction $BA$ is a coalgebra with comultiplication $$\Delta[a_1 | \cdots | a_n] = \sum_{0 \le i \le n} [a_1 | \cdots | a_i] \otimes [a_{i+1} | \cdots | a_n].$$ This is coassociative and counital. It will be sometimes convenient to write $(BA)_n = A[-1]^{\otimes n}$ for the summand corresponding to the tensor product of $n$ copies of $A$.

When $A$ is also given the structure of a differential graded algebra, $BA$ then inherits a differential making it into a dg-coalgebra, given as $$d[a_1 | \cdots | a_n] = -\sum_{1 \le i \le n} (-1)^{\epsilon_{i-1}^a} [a_1 | \cdots | da_i | \cdots | a_n] - \sum_{1 \le i \le n-1} (-1)^{\epsilon_i^a} [a_1 | \cdots | a_i a_{i+1} | \cdots | a_n],$$
where the signs $\epsilon$ are those from Definition \ref{def:Hirsch}, with the same conventions.

In fact, the structure of `a differential on $BA$ giving it the structure of a dg-coalgebra' is equivalent the structure of an $A_\infty$-algebra on $A$; up to an overall sign one may recover the $n$'th structure operation $m_n$ among the $A_\infty$-operations as the component $$A^{\otimes n}[-n] = (BA)_n \to (BA)_1 = A[-1]$$ of $d: BA \to BA$. All of our examples will be dg-algebras, so we restrict to that setting. The crucial observation for us is that a Hirsch algebra takes this one level further. 

\begin{theorem}\label{HirschBA}
Let $A$ be an associative dg-algebra, so that $BA$ carries the structure of a dg-coalgebra. 

Then the data of a Hirsch algebra structure on $A$ is equivalent to the data of a choice of multiplication $\mu: BA \otimes BA \to BA$ making it into a dg-bialgebra \textbf{with not-necessarily-associative product} and for which $[]$ is a unit. 

Given such a multiplication $\mu$, one may recover the operations $E_{p,q}$ as $$E_{p,q}(a_1, \cdots, a_p; b_1, \cdots, b_q) = \mu^1([a_1 | \cdots | a_p], [b_1 | \cdots | b_q]),$$ where $\mu^1$ is the composite $BA \otimes BA \to BA \xrightarrow{p_1} A$, where the final map projects to $(BA)_1$ and follows the degree $-1$ isomorphism $(BA)_1 \cong A$.
\end{theorem}

This is proved, among other places, in \cite[Section 2]{Voronov}. In fact, Voronov discusses a mild generalization: a `$B_\infty$ structure', which precisely corresponds to the structure of a dg-bialgebra on $BA$ extending its natural coalgebra structure. This amounts to saying that $A$ is an $A_\infty$-algebra and carries a set of operations $E_{p,q}$ satisfying a mild modification of those written above for Hirsch algebras. Voronov furthermore determines when this map $\mu$ is associative, but we will not need associativity and so do not discuss it further.

The demand that this is a dg-bialgebra means that $\mu: BA \otimes BA \to BA$ is a chain map, and that the following diagram commutes: 

\[\begin{tikzcd}
	{(BA \otimes BA) \otimes (BA \otimes BA)} && {(BA \otimes BA) \otimes (BA \otimes BA)} \\
	{BA \otimes BA} && {BA \otimes BA} \\
	& BA
	\arrow["{1 \otimes \tau \otimes 1}", from=1-1, to=1-3]
	\arrow["{\Delta \otimes \Delta}", from=2-1, to=1-1]
	\arrow["{\mu \otimes \mu}", from=1-3, to=2-3]
	\arrow["\mu"', from=2-1, to=3-2]
	\arrow["\Delta"', from=3-2, to=2-3]
\end{tikzcd}\]

Here $\tau: BA \otimes BA \to BA \otimes BA$ sends $x \otimes y \mapsto (-1)^{|x||y|} y \otimes x$. 

The fact that this is a dg-bialgebra allows us to compute the component $\mu^k: (BA \otimes BA) \to (BA)_k$ from the compute $\mu^1$, as follows. We define $$\nabla_k: (BA \otimes BA) \to (BA \otimes BA)^{\otimes k}$$inductively as follows. Set $\nabla_1 = \text{Id}$ and let $\nabla_2 = \nabla_{BA \otimes BA}$ be the coproduct given by $$\nabla_{BA \otimes BA} = (1 \otimes \tau \otimes 1)(\Delta \otimes \Delta),$$ where the map $\tau$ is the swap map given by $\tau(x \otimes y) = (-1)^{|x| |y|} y \otimes x$. Then define $$\nabla_{k+1} = (\nabla_2 \otimes 1_{(BA)^{\otimes 2k}}) \circ \nabla_k.$$ 
Then the component $\mu^k: (BA \otimes BA) \to (BA)_k \cong A^{\otimes k}$ is given by the composite $(\mu^1)^{\otimes k} \circ \nabla_k.$ 

\subsection{Twisting sequences in Hirsch algebras}
Now that we have seen the relationship between dg-algebra and Hirsch algebra structures on $A$ to various structures on $BA$, we can exploit that structure to get new results on twisting sequences. Recall that $g\in BA$ is grouplike, if $\Delta g = g \otimes g$.

\begin{prop}\label{gplike-cocycle}
Let $A$ be a dg-algebra, and say a twisting element of $A$ is a degree 1 element with $dx + x^2 = 0$. There is a canonical bijection from the set of twisting elements in $A$ to the set of (degree 0) grouplike cocycles in $BA$ whose $(BA)_0$ component is $1$.

If $A$ is a dg-algebra, twisting sequences in $A$ are twisting elements of $A\llbracket T\rrbracket$ of the form $x_3 T + x_5 T^2 + \cdots$; that is, they are twisting elements which lie in $T A\llbracket T\rrbracket$. Here $|T| = -2$.

Then the above construction gives a bijection between $TS(A)$ and grouplike elements of $B(A\llbracket T\rrbracket)$ lying in $1 + T BA\llbracket T\rrbracket$ whose $(BA)_0$-component is $1$. 
\end{prop}

If one is being careful, one should interpret the expression $B(A\llbracket T\rrbracket)$ as the bar construction in the category of $\mathbb Z\llbracket T\rrbracket$-algebras; it is the sum $$\bigoplus_{n \ge 0} A^{\otimes n}\llbracket T\rrbracket[-1];$$ we allow ourselves to distribute $T$ across tensor summands.

\begin{proof}
For any graded abelian group $A$, there is a canonical bijection between grouplike elements of $BA$ with first term $1$ and elements of $A$. For if $a \in A$, then $$g(a) = 1 + [a] + [a|a] + \cdots$$ is grouplike; on the other hand, suppose $g$ is a grouplike element with lowest term $1$. Inductively assume that $$g = 1 + [a] + \cdots + [a|\cdots|a] + \sum_{i=1}^N [a_1^i | \cdots | a_n^i] + \cdots,$$ where we have written the term with $n$ tensor factors as a sum of as few elementary tensors as possible. Computing $\Delta g$ and comparing its part with $n$ tensor summands to $g \otimes g$, we see that $$[a] \otimes [a| \cdots | a] = \sum_{i=1}^N [a_1^i] \otimes [a_2^i | \cdots | a_n^i].$$ Because the right-hand-side is written as a sum of as few elementary tensors as possible, but the left-hand side is a single elementary tensor, we see that in fact $N = 1$ and $[a_1^i | \cdots | a_n^i] = [a | \cdots | a]$.  By induction, we see that every grouplike element with first term $1$ arises as $g(a)$ for some unique $a$.

Next, we investigate the condition that $dg(a) = 0$. It is straightforward to see that $$dg(a) = -[da] - [a^2] - [da | a] - [a | da] - [a^2 | a] - [a | a^2] + \cdots;$$ from this, we see immediately that if $dg(a) = 0$ we have $da + a^2 = 0$ (because the $(BA)_1$ component of $dg(a)$ is up to sign $da + a^2$), while conversely if $da + a^2 = 0$ one may see explicitly that $dg(a) = 0$. 

It is now easy to see that the condition $$d(x_3 T + x_5 T^2 + \cdots) + (x_3 T + x_5 T^2 + \cdots)^2 = 0$$ gives, for each power of $T$, the conditions $dx_3 = 0, dx_5 + x_3^2 = 0, \cdots$ which define twisting sequences, giving the stated relationship between twisting sequences in $A$ and twisting elements of $A\llbracket T\rrbracket$.

To conclude, observe that the condition that our twisting sequence take the form $x_3 T + x_5 T^2 + \cdots$ corresponds to asking our grouplike cocycle in $BA\llbracket T\rrbracket$ to be of the form $$1 + [x_3] T + ([x_5] + [x_3 | x_3]) T^2 + \cdots$$
as an element in $1 + T BA\llbracket T\rrbracket$.\end{proof}

Proposition \ref{gplike-cocycle} is quite surprising because there is no obvious product operation on twisting sequences. However, given any dg-bialgebra, there \emph{is} a product operation on its grouplike cocycles: if $g, h$ are grouplike cocycles (of degree 0) then $$d\mu(g,h) = \mu(d[g\otimes h]) = \mu(0,0) = 0,$$ while \begin{align*}\Delta \mu(g,h) &= (\mu \otimes \mu)(1 \otimes \tau \otimes 1)(\Delta \otimes \Delta)(g,h) \\
&= (\mu \otimes \mu)(1 \otimes \tau \otimes 1)(g \otimes g \otimes h \otimes h) \\
&= (\mu \otimes \mu)(g \otimes h \otimes g \otimes h) \\
&= \mu(g,h) \otimes \mu(g,h),
\end{align*}
so the product of grouplike elements is also again grouplike. Applying this to the previous proposition and using the explicit formula for the $(BA)_1$ component of $\mu(g(x), g(y))$, we immediately get the following corollary.

\begin{cor}\label{ProdOp}
For any Hirsch algebra $A$, there is a product operation $\mu: TS(A) \times TS(A) \to TS(A)$. This product is natural, and gives rise to a functor from the category of Hirsch algebras to the category of unital magmas.

Explicitly, if $x_\bullet$ and $y_\bullet$ are twisting sequences, suppose $y_{2i+1} = 0$ for all $i < n$. Then $\mu(x_\bullet, y_\bullet)_{2i+1} = x_{2i+1}$ for $i < n$ and $\mu(x_\bullet,y_\bullet)_{2n+1} = x_{2n+1} + y_{2n+1}$. In general, we have $$\mu(x_\bullet,y_\bullet)_{2k+1} = \sum_{\substack{(m,n) > (0,0) \\ i_1, \cdots, i_m, j_1, \cdots, j_n \ge 1 \\ i_1 + \cdots + i_m + j_1 + \cdots + j_n = k}} E_{m,n}(x_{2i_1 + 1}, \cdots, x_{2i_m+1}; y_{2j_1+1}, \cdots, y_{2j_n+1}).$$
In particular, $\mu(0, x_\bullet) = x_\bullet = \mu(x_\bullet, 0).$
\end{cor}

Recall here that for a Hirsch algebra $A$ the product on $BA$ need not be associative, so this does not promote $TS(A)$ into a monoid. When $A = C^*(X;\mathbb Z)$, the product $\mu$ on $BA$ is indeed associative, so $TS(A)$ has an associative and unital product, but for our purposes below we will need to work with cubical cochains where $\mu$ is certainly not associative. 

\begin{remark}
In fact this product descends to a well-defined product homotopy classes
$$h\mu: ts(A) \times ts(A) \to ts(A)$$
To see why this is true, one needs to prove that there is a $B_\infty$-algebra modeling the unit interval $I$, with two $B_\infty$-maps to $\mathbb Z$ with trivial $B_\infty$-structure, and further that there is a suitable notion of $B_\infty$-tensor product structure on $I \otimes A$. Because a homotopy between twisting sequences is a twisting sequence on $I \otimes A$ extending the two, the claim follows.

Further, it should be true that every $B_\infty$-algebra may be strictified to one with associative product; this would imply that for any $B_\infty$-algebra $A$, the homotopy classes of twisting sequences $ts(A)$ inherit the structure of an associative monoid. Lastly, it should also be true that every element of $ts(A)$ is invertible, by performing a sort of iterated killing construction with the product $\mu$. It takes some effort to give a careful proof of this fact, including a proof that there is a meaningful notion of limit of elements of $ts(A)$ (coming from an appropriate complete Hausdorff filtration). Therefore $ts(A)$ should actually be a group for an arbitrary $B_\infty$-algebra $A$. Because we will not make use of these in our arguments below, we are content to leave these as remarks.
\end{remark}

The last general result we will make use of is the following, which asserts that in a \emph{rational} Hirsch algebra, there is no obstruction to extending a given cocycle $a \in Z(A)^{2n+1}$ to a twisting sequence beginning at $a$. In fact, there is a \emph{canonical} such extension.

\begin{prop}[The Kraines construction]
Let $A$ be a Hirsch $\mathbb Q$-algebra; write $$Z(A)^{\textup{odd}} = \bigsqcup_{n \ge 1} Z(A)^{2n+1}$$ for the set of homogeneous odd-degree cocycles of degree at least 3. Then there is a canonical map $$K: Z(A)^{\textup{odd}} \to TS(A)$$ so that if $a \in Z(A)^{2n+1}$, then $K(a)_{2i+1} = 0$ for $i < n$ and $K(a)_{2n+1} = a$. The function $K$ is natural under Hirsch algebra maps.
\end{prop}

\begin{proof}
We take the perspective of grouplike cocycles in $BA\llbracket T\rrbracket$. The element $a \in A^{2n+1}$ corresponds to an element $[a] T^n$ in $BA\llbracket T\rrbracket$. Define even-degree elements $a_m$ inductively by iterated left-multiplication by $[a]$. Precisely, set $a_0 = 1$ and $a_m = \mu([a], a_{m-1})$ (so in particular $a_1 = [a]$). Then define $$\exp(a) = 1 + \sum_{m=1}^\infty \frac{a_m}{m!} T^{nm}.$$
We suppress $T$ from notation in what follows. To get a sense for these $a_i$, observe that one can explicitly compute the first couple of $a_i$ directly from the definition of the operation $\mu$ (and in particular the components $\mu^k$):
\begin{align*}a_2 &= 2[a|a] + [E_{1,1}(a; a)]\\
a_3 &= 6[a|a|a] + 2[E_{1,1}(a;a) | a] + 2[a | E_{1,1}(a;a)] + 2[E_{1,2}(a;a,a)] + [E_{1,1}(a;E_{1,1}(a;a))].
\end{align*}
Lemma \ref{mult-formula} below gives a formula for $a_m$ in general; for the purposes of this argument, we need not be so explicit.

Let us show that $\exp(a)$ is a grouplike cocycle. To argue this, we will inductively compute $da_m = 0$ and $\Delta a_m = \sum_{i=0}^m \binom{m}{i} a_i \otimes a_{m-i}$. For the first, observe that $d[a] = 0$ and $d\mu(x,y) = \mu(dx,y) + (-1)^{|x|} \mu(x, dy)$, so $d\mu([a], a_{m-1}) = \mu(0, a_{m-1}) \pm \mu([a], 0) = 0$ by induction.

For the second fact, we have $\Delta a_0 = a_0 \otimes a_0$ by definition, and by induction \begin{align*}\Delta \mu([a], a_{m-1}) &= (\mu \otimes \mu)(1 \otimes \tau \otimes 1)(\Delta \otimes \Delta)([a] \otimes a_{m-1}) \\
&= (\mu \otimes \mu)(1 \otimes \tau \otimes 1)\bigg((1 \otimes [a] + [a] \otimes 1) \otimes \left(\sum \binom{m-1}{i} a_i \otimes a_{m-i-1}\right)\bigg)\\
&= (\mu \otimes \mu)\Big(\sum_{i=0}^{m-1} \binom{m-1}{i} (1 \otimes a_i) \otimes ([a] \otimes a_{m-i-1}) + ([a] \otimes a_i) \otimes (1 \otimes a_{m-i-1})\Big) \\
&= \sum_{i=0}^{m-1} \binom{m-1}{i} (a_i \otimes a_{m-i} + a_{i+1} \otimes a_{m-i-1}) \\
&= \sum_{i=0}^{m} \left(\binom{m-1}{i} + \binom{m-1}{i-1}\right) a_i \otimes a_{m-i}.
\end{align*}
Combining this with the recurrence relation $\binom{m-1}{i-1} + \binom{m-1}{i} = \binom{m}{i}$ we see that this reduces to $$\Delta a_m = \sum_{i=0}^m \binom{m}{i} a_i \otimes a_{m-i},$$ completing the desired induction.

This in hand, we see that \begin{align*}\Delta \exp(a) &= \sum_{m=0}^\infty \frac{\Delta(a_m)}{m!} = \sum_{m=0}^\infty \frac{\sum_{i=0}^m \binom{m}{i} a_i \otimes a_{m-i}}{m!} \\
&= \sum_{\substack{i,j \ge 0}} \frac{1}{(i+j)!} \binom{i+j}{i} a_i \otimes a_j \\
&= \sum_{i,j \ge 0} \frac{a_i \otimes a_j}{i! j!} \\
&= \exp(a) \otimes \exp(a),
\end{align*}
so $\exp(a)$ is indeed a grouplike cocycle, as desired. Now we define $K(a)$ to be the twisting with $\exp(a) = g(K(a))$; equivalently, take $K(a)$ to be the $(BA)_1$ component of $\exp(a)$.

Because the only tool used in this construction is the product $\mu$ on $BA$ (given to us by the Hirsch structure on $A$), it is clear that $K$ is natural for maps of Hirsch algebras.
\end{proof}

\begin{remark}
One can show that $K$ descends to a map from odd cohomology to $ts(A)$. However, it does not seem to be a group homomorphism in general on any $H^{2n+1}(A)$; it seems that one needs to assume some further homotopy-commutativity of $A$ (more precisely, the operation $E_{1,1}$ should itself be homotopy-commutative, in a way which is associative up to higher homotopies. This will not be relevant to us and so we do not explore it further.
\end{remark}

\begin{remark}\label{rmk:right-mult}
When $A = C^*_\Delta(X;\mathbb Z)$ for a simplicial set $X$, the definition of $K(a)$ simplifies if one uses \emph{right multiplication} instead of left multiplication. This is because the relevant higher operations $E_{m,1}$ are all zero on $C^*_\Delta(X)$. 

When these higher operations vanish and we define $K(a)$ using right-multiplication, write $E_{1,1}(a,b) = a \cup_1 b$. Then $K(a)$ has the form $K(a)_{2n+1} = a$, $K(a)_{4n+1} = \frac 12 a \cup_1 a$, and $$K(a)_{6n+1} = \frac 16 (a \cup_1 a) \cup_1 a,$$ and so on; all higher terms are given by iterated cup-1 products. All other $K(a)_{2i+1}$ are zero. 

This construction via right-multiplication appears in the proof of \cite[Lemma 16]{Kra}, and this argument led (directly or indirectly) to many of the ideas of this paper. We call the class $K(a)$ the \emph{Kraines construction} on $a$ in reference to this Lemma, though it might rightfully be called the Kraines--Saneblidze construction: the extension to Hirsch algebras described above first appears in \cite[Section 3.3]{Saneblidze}.

In our application, it will be more convenient to use left-bracketed multiplication, which is why we prefer that convention.
\end{remark}

The general formula for the higher terms in $K(a)$ is complicated, because the formula for the multiplication $\mu$ is complicated. It is given as a sum over iterated applications of the operations $p!E_{1,p}$ to copies of $a$, where the term after the semicolon is always $a$; for instance, $$2 E_{1,2}(a; E_{1,1}(a;a), E_{1,1}(a;a))$$ appears in $K(a)_{10n+1}$. It is difficult to make all of this explicit. Instead of proving the above (which is not the most useful way to phrase it), we will extract exactly what will be useful for us later. 

The following lemma is quite tedious and technical. A reader may prefer to skip it for now and return later when we find use for it in studying the minimal torus.

\begin{lemma}\label{mult-formula}
Let $A$ be a Hirsch algebra. Define $\mu_n: A^{\otimes n} \to BA$ recursively, by $\mu_1(a) = [a]$ and $$\mu_n(a_1, \dots, a_n) = \mu([a_1], \mu_{n-1}(a_2, \dots, a_n)).$$ Denote by $$\mu_n^k: A^{\otimes n} \to (BA)_k \cong A^{\otimes k}$$the $k$'th component of this map. \textbf{Suppose $a_1, \cdots, a_n$ are all odd-degree elements} in $A$. Then we have $$\mu_n^k(a_1, \dots, a_n) = \sum_{\substack{\textup{partitions of } \{1, \dots, n\} \\ \textup{into an ordered list of } k \textup{ nonempty sets} \\ I_1, \dots, I_k \subset \{1, \dots, n\}}} [\mu^1(a_{I_1}) \mid \dots \mid \mu^1(a_{I_k})].$$Here if $I_j = \{i_1 < \cdots < i_m\}$ then one interprets $\mu^1(a_{I_k}) = \mu^1_m(a_{i_1}, \cdots, a_{i_m})$ with $\mu^1_m:A^{\otimes m}\rightarrow A$.
\end{lemma}
\begin{remark}
In general, if not all $a_i$ are odd-degree elements, there are some additional signs. 
\end{remark}
Notice that, for instance, if $n = 3$, this includes the entire symmetric expression $$[a_1|a_2|a_3] + [a_1 | a_3 | a_2] + [a_2 | a_1 | a_3] + [a_2 | a_3 | a_1] + [a_3 | a_1 | a_2] + [a_3 | a_2 | a_1];$$ this is implicit in the statement `ordered list of $k$ nonempty sets', which means we know the order we list the subsets $(I_1, \dots, I_k)$.

\begin{proof}
The map $\mu_n^k$ is defined (inductively) as the composition
\begin{equation*}
A^{\otimes n}=A\otimes A^{\otimes (n-1)}\xrightarrow{s\otimes \mu_{n-1}} BA\otimes BA \stackrel{\mu^k}{\longrightarrow} (BA)_k=A^{\otimes k}
\end{equation*}
with $s(a)=[a]$ and $\mu^k$ the component of $\mu:BA\otimes BA\rightarrow BA$ landing in $(BA)_k$. Recall from the end of Section 3.1 that the component $\mu^k: (BA \otimes BA) \to (BA)_k \cong A^{\otimes k}$ is given by the composite $(\mu^1)^{\otimes k} \circ \nabla_k,$ where $\nabla_k$ is an iterated composite of the coproduct on $BA \otimes BA$. There is a sign appearing in the definition of $\nabla_k$ via the swap symmetry, but in the case we are interested in all terms which appear are even-degree in $BA$, so the sign in the swap symmetry is completely irrelevant to us. 

Writing $\vec a = [a_1 | \cdots | a_m]$ and $\vec b = [b_1 | \cdots | b_n]$, we have the explicit formula $$\nabla_k(\vec a \otimes \vec b) = \sum_{\substack{0 \le i_1 \le \dots \le i_k \le m \\ 0 \le j_1 \le \dots \le j_k \le n}} (\vec a_{[1, i_1]} \otimes \vec b_{[1, j_1]}) \otimes \cdots \otimes (\vec a_{[i_{k}+1, m]} \otimes \vec b_{[j_{k}+1, n]}).$$ Here we write $\vec a_{[i,j]} = [a_i | a_{i+1} | \cdots | a_j]$, interpreting this as the unit $1 \in (BA)_0$ when $i = j+1$. In the special case of interest, this becomes $$\nabla_k([a_1], [a_2 | \cdots | a_n]) = \sum_{\substack{1 \le i_1 \le \cdots \le i_k \le n \\ 0 \le j \le k}} (1 \otimes \vec a_{[2, i_1]}) \otimes \cdots \otimes ([a_1] \otimes \vec a_{[i_j+1, i_{j+1}]}) \otimes \cdots \otimes (1 \otimes \vec a_{[i_k + 1, n]}).$$ 
That is, we break the second vector into $k$ sequential pieces, and insert $[a_1]$ next to one of them (leaving $1$ next to everything else). 

Using these formulae, let us prove the desired claim by induction on $n$. For $n = 1$ and any $k$, the formula holds tautologically. Suppose the given formula holds for all $\mu^k_m$ for all $m < n$ and all $k$. We will verify the formula for $\mu^k_n$. For sake of space we make two further notational simplifications: if $$I = \{i_1 < \dots < i_k\} \subset \{1, \dots, n\}$$ we write $$\mu^1(I) = \mu^1(a_I) = \mu^1(a_{i_1}, \dots, a_{i_k}).$$ Further, if $I_1, \dots, I_m \subset \{1, \dots, n\}$ is a list of subsets, we write $$\mu^1(I_1 \mid \cdots \mid I_m) = [\mu^1(I_1) \mid \cdots \mid \mu^1(I_m)].$$

Our inductive hypothesis gives that $$\mu_{n-1}(a_2, \cdots, a_n) = \sum_{\substack{\textup{partitions of } \{2, \cdots, n\} \\ \textup{into an ordered list of nonempty sets} \\ I_1, \cdots, I_k \subset \{2, \cdots, n\}}} \mu^1(I_1 \mid \cdots \mid I_k),$$ where here we sum over all $k$. To compute the $k$'th component of $\mu([a_1],  \mu_{n-1}(a_2, \cdots, a_n))$, we must first compute $$\nabla_k\big([a_1] \otimes \mu_{n-1}(a_2,\cdots, a_n)\big).$$ By the descriptions of $\nabla_k$ above and the inductive formula for $\mu_{n-1}$, we see that this is given by the following expression. 
\begin{align*}
    &\nabla_k([a_1] \otimes \mu_{n-1}) = \nabla_k\Bigg({\sum_{\substack{\textup{partitions of } \{2, \cdots, n\} \\ \textup{into an ordered list of nonempty sets} \\ I_1, \cdots, I_m \subset \{2, \cdots, n\}}}} [a_1] \otimes \mu^1(I_1 \mid \cdots \mid I_m)\Bigg)\\
    = &\sum_{\substack{m \ge 1 \\ I_1, \cdots, I_m \text{ as above} \\0 \le i_1 \le \cdots \le i_k \le m \\ 0 \le j \le k}} \left(1 \otimes \mu^1(I_1 \mid \cdots \mid I_{i_1})\right) \otimes \cdots \otimes \left([a_1] \otimes \mu^1(I_{i_j+1} \mid \cdots \mid I_{i_{j+1}})\right) \otimes \cdots \otimes \left(1 \otimes \mu^1(I_{i_k+1} \mid \cdots \mid  I_m)\right).
\end{align*}
Finally, applying $(\mu^1)^{\otimes k}$ does nothing to most of these components. Any component $1 \otimes 1$ is sent to zero (so we may reindex the sum over $0 < i_1 < \cdots < i_k < m$); any component $1 \otimes a$ is sent to $a$. The remaining interesting component is the one with $a_1$. Renaming the sets $I_{i_j+1}, \cdots, I_{i_{j+1}}$ to $J_1, \cdots, J_k$ for convenience, write $J = J_1 \cup \cdots \cup J_k$. Consider $$\sum_{\substack{\text{partitions of } J \text{ into an ordered list of nonempty subsets} \\ J'_1, \cdots, J'_m}} \mu([a_1], \mu_m(J'_1 \mid \cdots \mid J'_m)).$$ By the inductive hypothesis, this is precisely $\mu([a_1], \mu^1(J))$, and by the recursive definition of $\mu$ this is precisely what we call $\mu^1(\{1\} \cup J)$. Setting $$m' = m - (i_{j+1}-1 - i_j)$$ and $$I'_k = \begin{cases} I_k & k \le i_j \\ \{1\} \cup I_{i_j+1} \cup \cdots \cup I_{i_{j+1}} & k = i_j+1 \\ I_{k+i_{j+1} -1 -i_j} & k > i_j+1 \end{cases}$$
we may thus reindex the above sum as $$\sum_{\substack{\text{partitions of } \{1, \cdots, n\} \\ \text{into an ordered list of nonempty subsets} \\ I'_1, \cdots, I'_{m'}}} \mu^1(I'_1 \mid \cdots \mid I'_{m'}).$$ 
This completes the induction.
\end{proof}
\medskip 
\subsection{Characteristic classes of twisting sequences}
By combining the previous sections, we are able to construct characteristic classes of twisting sequences which completely characterize their homotopy classes --- so long as the algebra has torsion-free cohomology. In the case of spaces, these should be intuitively thought as the pullbacks of certain odd degree elements in the rational cohomology of the classifying space $\mathcal{U}_T$ in Remark \ref{htpy}, hence the name.

\begin{theorem}\label{charclass-facts}
Let $A$ be a Hirsch algebra. There are maps $F_n: ts(A) \to H^{2n+1}(A) \otimes \mathbb Q$ which are natural for Hirsch algebra maps $f: A \to B$. These satisfy the following properties:
\begin{enumerate}
\item if $x_{2i+1} = 0$ for all $i < n$, then $F_n(x_\bullet) = [x_{2n+1}]$. 
\item if $a \in A^{2n+1}$ is a cocycle, then $F_n K(a) = [a]$ and $F_m K(a) = 0$ for all $m \ne n$.
\item if $A$ has torsion-free cohomology and $F_i(x_\bullet) = F_i(y_\bullet)$ for all $i \le n$, then $x_\bullet$ is homotopic to a twisting sequence $x'_\bullet$ with $x'_{2i+1} = y_{2i+1}$ for all $i \le n$. If the cohomology of $A$ is torsion-free, supported in bounded degrees, and $F_n(x_\bullet) = F_n(y_\bullet)$ for all $n$, then $x_\bullet$ is homotopic to $y_\bullet$.
\end{enumerate}
\end{theorem}

\begin{proof}
We will outline a construction of $F_n(x)$ given a twisting sequence $x$; this construction will only use the operations given by a Hirsch algebra structure, and so will transparently be natural. It will remain to show that these are homotopy invariants satisfying the given properties.

Given a twisting sequence $x_\bullet$, we will iteratively construct twisting sequences $K_{(n)}(x_\bullet)$ in $A_{\mathbb Q}$ with $K_{(n)}(x_\bullet)_{2i+1} = x_{2i+1}$ for all $i < n$. These will be called the \emph{Kraines approximations} to $x_\bullet$. 

The element $c_{2n+1} = x_{2n+1} - K_{(n)}(x_\bullet)_{2n+1}$ is a cocycle, because $$dc_{2n+1} = -\sum_{\substack{i+j=n \\ i,j \ge 1}} x_{2i+1} x_{2j+1} + \sum_{\substack{i+j=n \\ i,j \ge 1}} K_{(n)}(x_\bullet)_{2i+1}K_{(n)}(x_\bullet)_{2j+1} = 0,$$ because $x_{2i+1} = K_{(n)}(x_\bullet)_{2i+1}$ for $i < n$. We will set $$F_n(x_\bullet) = [c_{2n+1}] = \big[x_{2n+1} - K_{(n)}(x_\bullet)_{2n+1}\big].$$ That is, each characteristic class is the obstruction to finding a homotopy $x_\bullet \to K_{(n)}(x_\bullet)$.  

As for the construction, set $K_{(1)}(x_\bullet) = 0$, so that $F_1(x_\bullet) = [x_3]$. Inductively, using the product operation of Corollary \ref{ProdOp}, set $$K_{(n+1)}(x_\bullet) = \mu(K(c_{2n+1}), K_{(n)}(x_\bullet)).$$ The twisting sequences $K_{(n+1)}(x_\bullet)$ and $K_{(n)}(x_\bullet)$ agree through degree $2n-1$ because $K(c_{2n+1})_{2i+1} = 0$ for $i < n$, but is equal to $$K_{(n)}(x_\bullet)_{2n+1} + c_{2n+1} = x_{2n+1}$$ in degree $2n+1$, the equality following by definition of $c_{2n+1}$. It follows that $K_{(n+1)}(x_\bullet)_{2i+1} = x_{2i+1}$ for $i \le n$. 

Therefore, we have constructed $K_{(n)}(x_\bullet)$ with the desired properties, and thus we have also constructed the characteristic classes of twisting sequences. 

Notice that because the Kraines construction $K(a)$ is natural for Hirsch algebra maps $f: A \to B$, and the product $\mu$ of twisting sequences is natural for Hirsch algebra maps, these characteristic classes are also natural for Hirsch algebra maps $f: A \to B$.\\

To see that these are homotopy invariants, first observe that because this construction factors through $A_{\mathbb Q}$, it suffices to show that they are homotopy invariants for rational Hirsch algebras. Over the rationals, we exhibited a \textit{commutative} interval algebra in Remark \ref{commutativeintalg}; recall that this is $\mathbb{Q}[t, dt]/(dt)^2$, where $|t| = 0$ and $|dt| = 1$ and $d(t^n) = n t^{n-1} dt$. It then suffices to observe that if $A$ is a rational Hirsch algebra, so is $A\otimes \mathbb{Q}[t, dt]/(dt)^2$. Explicitly, \begin{align*}E_{p,q}&(a_1 t^{n_1} (dt)^{m_1}, \cdots, a_p t^{n_p} (dt)^{m_p}; b_1 t^{n'_1} (dt)^{m'_1}, \cdots, b_q t^{n'_q} (dt)^{m'_q}) \\ &= \pm E_{p,q}(a_1, \cdots, a_p; b_1, \cdots, b_q) t^{\sum n + n'} (dt)^{\sum m + m'}.\end{align*} The signs arise from the Koszul sign rule when commuting the single nonzero copy of $dt$ to the righthand side of the expression (if two appear, the whole expression is zero). A straightforward but tedious computation shows that this defines a Hirsch algebra structure on $A[t,dt]/(dt)^2$. Two twisting sequences in $A$ are homotopic if and only if there is a twisting sequence on $A[t,dt]/(dt)^2$ restricting to the two. It follows from naturality that if $x,y$ are homotopic, then $F_n(x) = F_n(y)$ for all $n$.\\

Now we move on to verifying the three claimed properties.

(1) If $x_{2i+1} = 0$ for all $i < n$, then the first stages of this construction are given by multiplication with $K(0) = 0$, which changes nothing; so $K_{(n)}(x_\bullet) = 0$. It follows that $F_n(x) = [x_{2n+1} - 0] = [x_{2n+1}]$, as claimed.

(2) If $a \in A^{2n+1}$ is an odd cocycle, then the first $n$ Kraines approximations have $K_{(n)}K(a) = 0$, so that $F_mK(a) = 0$ for $m < n$ and $F_n K(a) = [a - 0] = [a]$. On the other hand, $$K_{(n+1)}(K(a)) := \mu(K(a), K_{(n)}(K(a)) = \mu(K(a), 0) = K(a).$$ Because $K(a) = K_{(n+1)} K(a)$, this approximation is perfect, and there are no obstructions to finding a homotopy between these. All higher characteristic classes are zero.

(3) We will prove the final claim by induction. It is tautologous for the case $n = 0$. Inductively, suppose that $F_i(x_\bullet) = F_i(y_\bullet)$ for all $i \le n$ and that $x_\bullet$ is homotopic to a twisting sequence $x'_\bullet$ so that $x'_{2i+1} = y_{2i+1}$ for all $i < n$; because the $F_i$ are homotopy invariants, $F_i(x'_\bullet) = F_i(y_\bullet)$ for all $i \le n$ as well. 

Because $F_n(x_\bullet) = [x_{2n+1} + \text{terms with smaller indices}]$, and the lower-degree terms of $x'_\bullet$ and $y_\bullet$ agree, it follows that $$F_n(x'_\bullet) - F_n(y_\bullet) = [x'_{2n+1} - y_{2n+1}]_{\mathbb Q}.$$ But by hypothesis, $F_n(x'_\bullet) = F_n(y_\bullet)$, so $x'_{2n+1} - y_{2n+1}$ represents zero in \textbf{rational cohomology}.  Because $A$ has torsion-free cohomology, the map $$H^*(A) \to H^*(A) \otimes \mathbb Q$$ is injective. It follows that $y_{2n+1} - x'_{2n+1}$ is a coboundary in $A$; say $dh_{2n} + x'_{2n+1} = y_{2n+1}$. Now apply Lemma \ref{lemma:add-a-cbdry} to see that we may find a new twisting sequence $x''$ which agrees with $y$ through degree $2n+1$, so that $x'$ is homotopic to $x''$. This completes the induction.
When $A$ has bounded cohomology, if $x$ is always homotopic to some $x'$ which agrees with $y$ through degree $2n+1$, then $x$ is in fact homotopic to $y$; simply choose $n$ so large that $x'$ agrees with $y$ except possibly in degrees where the cohomology of $A$ is zero. 

Then one may construct a homotopy from $x'$ to $y$ inductively; choose $h_{2i} = 0$ for $i < n$. The obstruction to extending this homotopy at each stage $m$ lies in some cohomology group $H^{2m+2}(A)$, but we are in the range where these are zero, so there is no obstruction to extending the homotopy. Thus $x$ is homotopic to $x'$, which is homotopic to $y$; because homotopy is an equivalence relation, $x$ is homotopic to $y$ as desired.
\end{proof}

\begin{remark}In fact, the final statement should hold without the assumption that $A$ has nonzero cohomology in bounded degrees; we will not use this more general statement.
\end{remark}

The first handful of these characteristic classes can be defined explicitly. For instance, $$F_2(x_\bullet) = [x_5 - \frac 12 x_3 \cup_1 x_3],$$ where here $a \cup_1 b = E_{1,1}(a,b)$, and $$F_3(x_\bullet) = [x_7 - x_5 \cup_1 x_3 + \frac{1}{3} x_3 \cup_1 (x_3 \cup_1 x_3) - \frac{1}{3} E_{1,2}(x_3; x_3, x_3)].$$ The formulas are much more complicated for $F_4$, because one must define this by comparing to $K(x_3, x_5 - \frac{1}{2} x_3 \cup_1 x_3)$, and this twisting sequence has a very complicated formula (albeit one which is natural for Hirsch algebra maps).

\begin{remark}
If one clears denominators on the formulas for $F_n(x_\bullet)$, one obtains integral cohomology classes. A more intricate version of the above argument establishes that these are still homotopy invariants, and in fact the mod 2 Hopf invariant of Section 1.3 is the integral cohomology class defined by $2F_2$. To argue that these are homotopy invariants, one needs to know that our construction extends to twisting sequences in $B_\infty$-algebras; that there is a $B_\infty$-interval algebra over $\mathbb Z$; and that tensor products of $B_\infty$ algebras carry a natural $B_\infty$ structure (defined inexplicitly using an acyclic carrier argument). This is necessary because it is not at all clear how to give the tensor product of two Hirsch algebras the structure of a Hirsch algebra. We avoid this by using the commutative interval algebra $\mathbb Q[t,dt]/(dt)^2$; when $B$ is Hirsch and $A$ is commutative, $A \otimes B$ has a canonical Hirsch algebra structure once more.
\end{remark}
\medskip

\bigskip

\section{Examples of Hirsch algebras}\label{examplesHirsch}

In this section we describe the two key examples of Hirsch algebras we will be interested in, namely simplicial and cubical cochains. While explicit formulas of such structures can be found in the literature (see \cite{KS} for the cubical case and \cite{Baues} for the simplicial case), we will provide a description in terms of the recent framework of \cite{MM1} and \cite{K-MM}. This allows both for explicit computations and for a straightforward comparison between the cubical and simplicial worlds. 

The structure operations $E_{p,q}$ are defined on the cochain algebras of an arbitrary simplicial or cubical set. To define them, one simply needs to define these operations on the basic algebras $C^*_\Delta(\Delta^n)$ and $C^*_\cube(I^n)$ and verify that they restrict appropriately to subsimplices or subcubes; that is, one needs to verify that the Hirsch algebra structures we construct are natural for the various face and degeneracy maps. 

In \cite{K-MM}, the authors construct a great many operations on simplicial and cubical cochains (giving what is called an $E_\infty$-structure on cochain algebras). They are defined as particular composites of two basic input operations, $\Delta$ and $*$; their boundary relations also involve an operation $\epsilon$. 

Not all of these composites are sufficiently natural and so do not define operations on the cochains of an arbitrary simplicial/cubical set (for instance, the join operation $*$ is only well-defined on cochains on $I^n$, not on arbitrary cubical sets). 

However, by representing these operations as particularly simple trivalent graphs, the authors determine when composites of these operations do give well-defined operations on all cubical cochain algebras. This will be the case for the operations $E_{p,q}$ we aim to construct.

\begin{remark}
It is not a priori clear whether the Hirsch algebra structure we describe below for simplicial and cubical complexes is exactly the one appearing in the literature; this will not be important for our purposes. 
\end{remark}

\begin{remark}
The signs in what follows (especially in regards to duals and tensor products) can be rather intricate. The authors found \cite{TLSigns} useful in keeping these straight and understanding the conceptual origin of the various sign conventions.
\end{remark}

\medskip

\subsection{Simplicial and cubical cochains as Hirsch algebras}
To begin with, we recall the definitions of the basic operations in both the simplicial and cubical settings. In both cases, we need to define the basic operations $\Delta$ (which we call the diagonal or coproduct), $*$ (which we call the join or degree 1 product), and $\epsilon$ (which we call the augmentation) and verify some simple relations between them.
\\
\par
\textbf{Simplicial setup.} 
In the case of simplicial chains, $\Delta$ is given by the classical Alexander-Whitney map:
\begin{equation*}
\Delta[v_0,\dots, v_q]=\sum_{i=0}^q[v_0,\dots, v_i]\otimes [v_{i+1},\dots, v_q],
\end{equation*}
the join $*$ is the usual join of simplices given by the Eilenberg-Zilber shuffle product, $$[v_0, \dots, v_i] * [v_{i+1}, \dots, v_p] = \begin{cases} (-1)^{i+|\pi|} [v_{\pi(0)}, \dots, v_{\pi(p)}] & v_i \ne v_j \text{ for } i \ne j \\ 0 & \text{else}\end{cases}$$ where here $\pi$ is the unique permutation so that the above list of vertices is in increasing order, and $(-1)^{|\pi|}$ is the sign of this permutation. Lastly, the augmentation is given by 
\begin{equation*}
\varepsilon[v_0,\dots v_q]=\begin{cases}
1 \text{ if }q=0\\
0 \text{ otherwise.}
\end{cases}
\end{equation*}

\textbf{Cubical setup.} For our purposes we will be interested in cubical chains first introduced by Serre in \cite{Ser} (with slightly different conventions from ours), see \cite{Mas} for an in-depth treatment of the singular version. A cubical set has a set $\Box_k(X)$ of k-cubes in $X$ for all k; these come equipped with a collection of degeneracy operations and a collection of pairs of face operators (front face and back face). The cubical cochains $C^k_{\cube}(X)$ on a cubical set $X$ are the set of functions $\Box_k(X) \to \mathbb Z$ which vanish on degenerate cubes. 

If the cubical set in question is $X_{\cube} = \text{Sing}_{\cube}(X)$ for some topological space $X$, this gives precisely the singular cubical cochain complex defined in Massey's book. Chains in $C_*^{\cube-\text{sing}}(X)$ consists of continuous maps $I^n\rightarrow X$; a chain is degenerate if it factors through a projection to a face $I^n\rightarrow I^{n-1}\rightarrow X$, and the normalized cubical chains are obtained by quotienting the complex of cubicial chains by the degenerate ones. Dually, cubical cochains are functions which assign to each continuous map $I^n \to X$ an integer, and assign zero to any map which factors through a projection $I^n \to I^{n-1}$.  

To give explicit formulas for operations on $C^*_\cube(I^n)$, we will use the following basis for the predual space of chains $C_*^\cube(I^n)$. The basis elements are length $n$ strings in the alphabet $\{0, 1, I\}$; these correspond geometrically to \emph{faces} of the cube $I^n$, where $0$ and $1$ indicate that the given coordinate is fixed, while $I$ indicates that the given coordinate is free. For instance, $0I$ denotes the left vertical edge on the standard unit square, while $I1$ corresponds to the top horizontal edge. The degree of a string is given by the number of $I$'s in it (so $II$ has degree $2$, while $I0I10I$ has degree $3$).

If $x$ is a basis vector for $C_*^\cube(I^n)$, we will write $e^x$ for the dual basis vector in $C^*_\cube(I^n)$; for instance, $e^{0I}$ evaluates to 1 on $0I$ and to zero on all other strings. 

The operation $\Delta$ is the Cartan-Serre diagonal, which we now define. First, on the $1$-cube $I$ we define
\begin{equation*}
\Delta(0)=0\otimes 0\quad\Delta(1)=1\otimes 1\quad \Delta(I)=0\otimes I+I\otimes 1. 
\end{equation*}
In general we write 
\begin{equation*}
\Delta(x_i)=\sum x_i^{(1)}\otimes x_i^{(2)}
\end{equation*}
using Sweedler's notation for coproducts, where we leave the index of summation ambiguous; experience shows this causes no real trouble.

We then define
\begin{equation*}
\Delta(x_1\cdots x_n)=\sum\pm(x_1^{(1)}\cdots x_n^{(1)})\otimes (x_1^{(2)}\cdots x_n^{(2)})
\end{equation*}
where the sign is determined by the Koszul convention: whenever we swap two consecutive elements $a$ and $b$, the sign $(-1)^{|a||b|}$ is introduced. For example, we have on $I^2$ that
\begin{equation*}
\Delta(II)=00\otimes II+I0\otimes 1I- 0I\otimes I1+II\otimes 11.
\end{equation*}
Taking duals, this corresponds to a description of the cup product of two cochains on the square. Explicitly, $$e^{00} \cup e^{II} = e^{I0} \cup e^{1I} = -e^{0I} \cup e^{I1} = e^{II} \cup e^{11} = e^{II},$$ and these are the only products of monomials which evaluate to $e^{II}$. As in the case of simplicial chains, the coproduct $\Delta$ is coassociative; dually, the cup product on cubical cochains makes $C^*_\cube(X)$ into an associative dg-algebra.
\par
The augmentation $\epsilon$ evaluates on a string $x = x_1 \cdots x_n$ as $$\epsilon(x_1 \cdots x_n) = \begin{cases} 0 & x_i = I \text{ for some } i \\ 1 & \text{ otherwise.} \end{cases}$$ That is, it sends all positive-degree strings to zero and sends each zero-degree string to 1.
\par
The join map (which has degree $+1$ on chains) is defined as follows. On the $1$-cube $I$, the only non-zero values are
\begin{equation*}
0\ast 1=I,\quad\quad 1\ast0=-I    
\end{equation*}
In general, given strings $x$ and $y$, we define their join to be
\begin{equation*}
(x_1\dots x_n)*(y_1\dots y_n)=(-1)^{|x|}\sum_{i=1}^n\varepsilon(y_{<i})\varepsilon( x_{>i}) x_{<i}(x_i*y_i)y_{>i}.
\end{equation*}
Here $x_{<i} (x_i * y_i) y_{>i}$ should be understood as concatenation of strings, where as strings we have
\begin{align*}
x_{<i}&=x_1\dots x_{i-1}\\
x_{>i}&=x_{i+1}\dots x_n  
\end{align*} and we interpret $x_{<1} = x_{>n}$ to be the empty string; similarly for $y_{<i}$ and $y_{>i}$. For example, on the square we have $00*1I=II$ while $00*I1$ is zero.
\begin{remark}\label{nonass}
The join operation is \textit{not} associative on the nose. For example
\begin{align*}
(00*10)*11&=I0*11=-II\\
00*(10*11)&=00*1I=II.
\end{align*}
However, as seen in this example, the join is associative up to a global sign. This is sometimes called \emph{anti-associativity} or \emph{graded-associativity}, and it is a phenomenon forced on us by the use of the Koszul sign convention because the join operation has odd degree. 
Precisely, given strings $x,y,z$, we have $$(x * y) * z = (-1)^{|x|+1} x * (y * z).$$ 
In fact, up to the overall sign depending on the degrees of $x$ and $y$, both iterated joins are given by 
\begin{equation*}
\sum_{1 \le i < j \le n}\varepsilon(x_{>i})\varepsilon(y_{<i})\varepsilon(y_{>j})\varepsilon(z_{<j}) x_{<i}(x_i*y_i)y_{i<j}(x_j*y_j)z_{>j}
\end{equation*}
where $y_{i<j}=y_{i+1}\dots y_{j-1}$. This is because the coefficients on all other terms vanish; for instance, in computing $(x\ast y)\ast z$, if the next join happens at position appears at a position $j\leq i$ then $\varepsilon(x\ast y)_{>j}=0$ because this string includes $x_i\ast y_i$, which is either $\pm I$ or 0. In what follows, we will always perform the join operation starting from the left to avoid confusion with signs.
\end{remark}

\medskip

\textbf{The operations $E_{p,q}$. }We will now discuss how to use the operations $\Delta$ and $*$ to construct the Hirsch algebra structure operations $E_{p,q}$. Indeed, we will define these to be dual to certain compositions of $\Delta$ and $*$ on singular chains. 

These are the easiest to describe in a pictorial fashion using the framework of \cite{MM1}. There, the author represents the basic operations $\Delta$, $*$ and $\varepsilon$ as graphs as in Figure \ref{basicops}. Our operations will be obtained by combining the first two; the third one will naturally appear in certain differentials. As a general fact about operads generated by some basic operations, we may visualize composites of these operations as \emph{immersed trivalent graphs with monotonically decreasing edges, considered up to homotopy through immersions which fixes the endpoints and the cyclic ordering of edges around each vertex}. (That is, the product and coproduct are not commutative: one may not flip the order of their inputs.)

\begin{figure}[H]
  \includegraphics[width=0.7\linewidth]{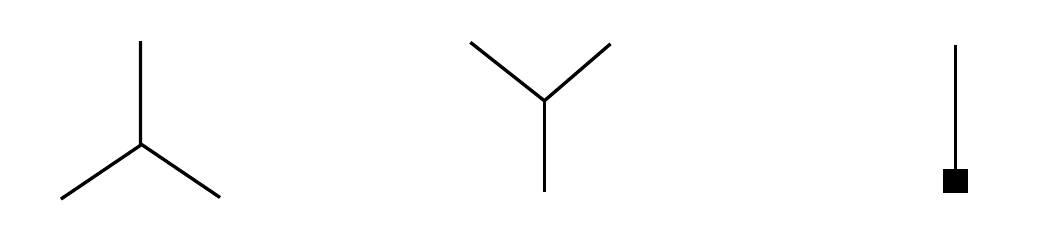}
  \caption{The basic operations $\Delta$, $*$ and $\varepsilon$. We use a square instead of a dot in the case of $\varepsilon$.}
  \label{basicops}
\end{figure}

\begin{remark}
It is important to remark that not all such graphs descend to well defined operations at the level of cochains. Indeed, the join operation is not well defined for general simplicial or cubical sets (for example, the one obtained from the interval $I$ by identifying the endpoints). On the other, the operations associated with graphs with only one incoming always determine a well defined operation on both simplicial and cubical complexes \cite{MM1}, \cite[Theorem $4$]{K-MM}.
\end{remark}

The operations we construct are defined first on chains, and come with a different sign than the one desired in Definition \ref{def:Hirsch} (only depending on the degree of the inputs). To avoid confusion, we will denote these chain operations by $e^{p,q}: C \to C^{\otimes(p+q)}$; before discussing the graphical calculus, we need to discuss the conventions for boundaries of operations and their duals.
\par
The differential of the basic operations $\Delta, *,\varepsilon$ as elements in Hom complexes $\text{Hom}(V, W)$ are described in Figure \ref{diffbasicops}; the only non-trivial computation is the middle one (see \cite{MM1} for the simplicial case and \cite[Lemma $3$]{K-MM} for the cubical case).

\begin{figure}
  \centering
\def\svgwidth{\textwidth}
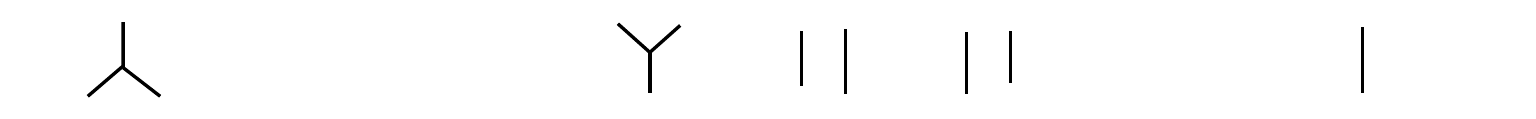
\caption{The differentials (in the Hom complex) of the basic operations.}
\label{diffbasicops}
\end{figure} 

Recall that the differential in $\text{Hom}(V, W)$ is defined to be
\begin{equation}\label{diffhom}
(\partial f)(v)=\partial (f(v))-(-1)^{|f|}f(\partial v).
\end{equation}
In particular, in $V^* = \text{Hom}(V, \Bbb Z)$, the differential is
\begin{equation*}
(\delta \varphi)(v) = -(-1)^{|\varphi|} \varphi(\partial v). 
\end{equation*}
In the Koszul sign convention we also have
\begin{equation}\label{signtens}
(f\otimes g)(a\otimes b)=(-1)^{|g||a|}f(a)g(b).
\end{equation}
Observe that there is a chain map 
\begin{equation}\label{dualmap}
\text{Hom}(V,W) \to \text{Hom}(W^*, V^*), \quad f \mapsto f^* \quad \text{where} \quad (f^* \varphi)(v) = (-1)^{|f||\varphi|} \varphi(fv).
\end{equation}

That this is a chain map amounts to the claim that $\delta f^* = (df)^*$, which is a straightforward calculation. Whenever we have an operation $e^{p,q}: C \to C^{\otimes p} \otimes C^{\otimes q}$, we define $$e_{p,q} = (e^{p,q})^*$$ to be the dual operation, including the sign in (\ref{dualmap}). The fact that passing to the dual map commutes with the differential shows that any relation we find for the boundary of the $e^{p,q}$ operations will translate with no sign change to a relation for the boundary of the $e_{p,q}$ operations.

Next, we should set up the graphical calculus. Some straightforward relations between the basic chain operations $\Delta, \tau, \epsilon, *$ can be found in Figure \ref{relbasicops}. Specifically, the first picture on the left asserts coassociativity of $\Delta$; the second picture is a consequence of the facts that the join of two elements has positive-degree (as the augmentation sends positive-degree terms to zero) and that $*\Delta = 0$; the third picture follows from the simple computation that for positive-degree $x$ we have
\begin{equation*}
\Delta x= x_0\otimes x+x\otimes x_0'+\sum x^{(1)}\otimes x^{(2)}
\end{equation*}
where $\varepsilon(x_0)=\varepsilon(x_0')=1$ (so that in particular $x_0,x_0'$ have degree zero) and each $x^{(1)}, x^{(2)}$ has positive degree.

\begin{figure}[H]
  \centering
\def\svgwidth{0.9\textwidth}
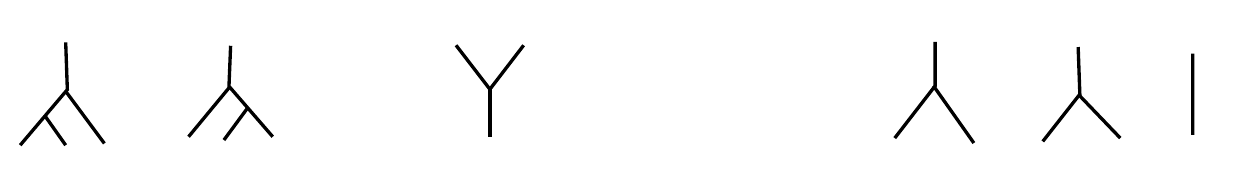
\caption{Some simple relations (in the Hom complex) among the basic operations.}
\label{relbasicops}
\end{figure} 

For convenience we will also depict its iterated $n$-fold composition with a "rake" with $n$ outgoing ends as in Figure \ref{rake}.

\begin{figure}[H]
  \centering
\def\svgwidth{0.9\textwidth}
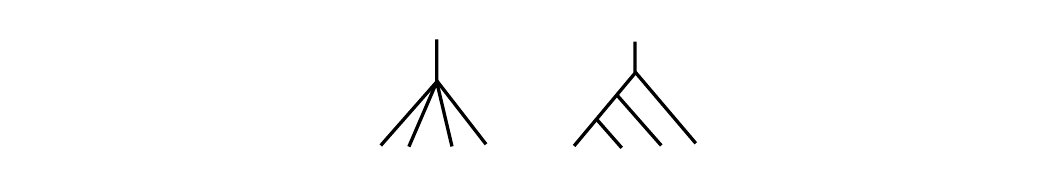
\caption{The definition of the rake; note that the order of splitting is irrelevant in cubical and singular chains, because the coproduct is associative on the nose.}
\label{rake}
\end{figure}

We are now ready to describe the graphs which define the operations $e^{p,q}$. We begin with the operation $e^{1,1}$, whose corresponding graph can be found in Figure \ref{e11}.
\begin{figure}[H]
  \centering
\def\svgwidth{0.9\textwidth}
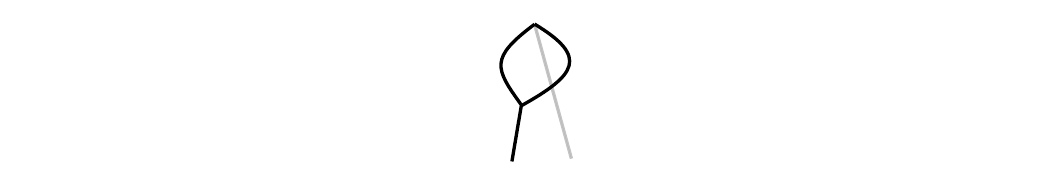
\caption{The operation $e^{1,1}$. This is a graphical encoding of a composite of basic operations and a swap map, $(* \otimes 1)(1 \otimes \tau)(\Delta \otimes 1)\Delta.$ For clarity, we will in general depict in black the strands that will end up in the first $p$ positions and in gray the strands that end up in the last $q$ positions. }
\label{e11}
\end{figure} 
We compute the differential of the operation $e^{1,1}$ as an element in $\text{Hom}(C, C^{\otimes 2})$ in Figure \ref{diffe11}, using the relations of Figure \ref{diffbasicops}. Here recall that (\ref{diffhom}) implies that for a composition of Hom elements we have

\begin{equation}\label{signcomp}
\partial(f\circ g)= (\partial f)\circ g +(-1)^{|f|}f\circ (\partial g).
\end{equation}
To go to the second row, we use the last identity in Figure \ref{relbasicops}. We now dualize to obtain an identity for $e_{1,1}$, which reads
\begin{equation*}
(\partial e_{1,1}) (a\otimes b)= -a\cdot b+(-1)^{|a||b|}b\cdot a.
\end{equation*}
(See Figure \ref{diffe11} below.) Recalling that the Koszul convention for the differential of a tensor product has
\begin{equation*}
\partial(a\otimes b)=\partial a\otimes b+(-1)^{|a|}a\otimes \partial b,
\end{equation*}
this, together with identity (\ref{diffhom}), means
\begin{equation*}
\partial e_{1,1}(a;b)+e_{1,1}(da;b)+(-1)^{|a|}e_{1,1}(a;db)=-a\cdot b+(-1)^{|a||b|}b \cdot a.
\end{equation*}
We then see that if we set
\begin{equation*}
 E_{1,1}(a;b)=(-1)^{|a|+1}e_{1,1}(a;b).
\end{equation*}
the identity (\ref{E11id}) holds. 
\begin{figure}[H]
  \centering
\def\svgwidth{0.9\textwidth}
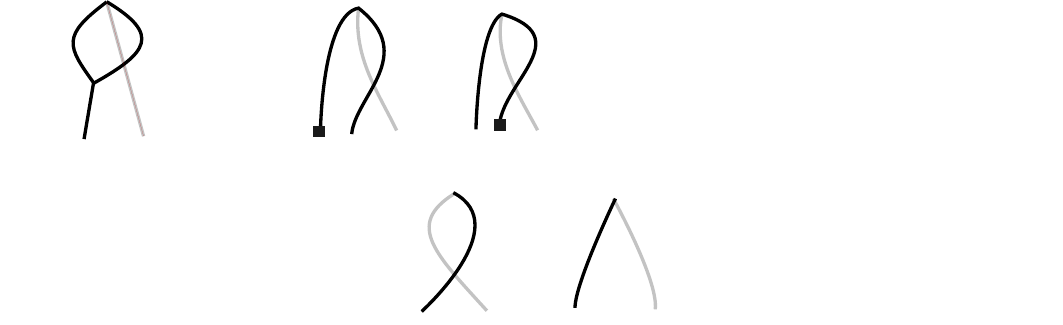
\caption{The differential in the Hom complex of the operation $e^{1,1}$. In the second identity, we use the fact (which directly follows from the definition of $\varepsilon$) that we can drag an edge ending with a square past any other edge.}
\label{diffe11}
\end{figure} 

Let us describe now in detail the operations $e_{1,2}$ and $e_{2,1}$; these will measure the failure of the right and left Hirsch formulas respectively. They are described by the graphs in Figure \ref{e12e21}.
\begin{figure}[H]
  \centering
\def\svgwidth{0.9\textwidth}
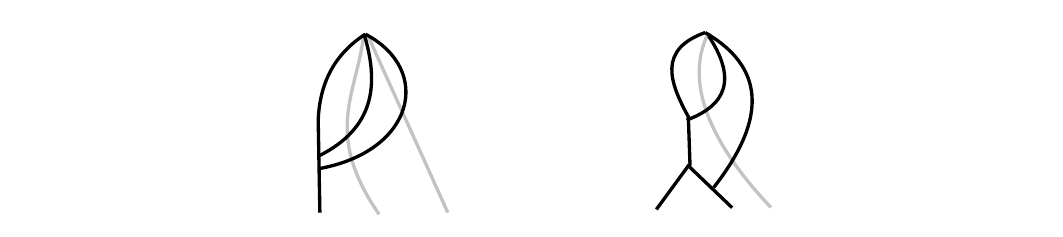
\caption{The operations $e^{1,2}$ and $e^{2,1}$.}
\label{e12e21}
\end{figure} 
The differential of $e^{1,2}$ is described in Figure 4.

\begin{figure}[H]
  \centering
\def\svgwidth{0.9\textwidth}
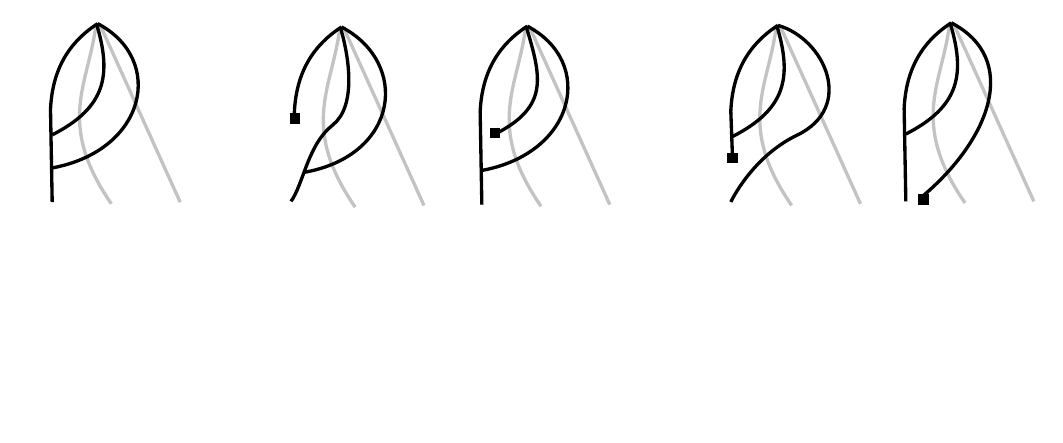
\caption{The differential of the operation $e^{1,2}$. The signs are introduced because of convention (\ref{signcomp}). Here we think of the homological operation, so that the figure is to be read from top to bottom. The convention  (\ref{signcomp}) means that the sign of the differential depends on the degree of the operation to be computed below the place we added a dot, and joins are the only basic operation of odd degree.}
\end{figure} 

The third element in the first row vanishes because of the second relation in Figure \ref{relbasicops}. Applying the last identity of Figure \ref{relbasicops} and an isotopy, we finally reach the final form in the last row. After dualizing, this reads
\begin{equation*}
(\partial e_{1,2})(a;b,c)=-(-1)^{(|a|+1)|b|}b \cdot e_{1,1}(a;c)+e_{1,1}(a;bc)-e_{1,1}(a;b) \cdot c
\end{equation*}
where in the first term the $|a|+1$ appears because $*$ is an operation of degree $1$, see Figure \ref{braiding}.
\begin{figure}[H]
  \centering
\def\svgwidth{0.9\textwidth}
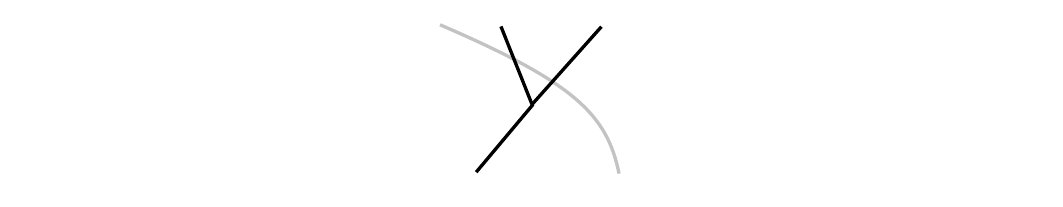
\caption{The braiding introduces a sign of $(-1)^{|a_1||b|+|a_2||b|}=(-1)^{(|a|+1)|b|}$ because the join has odd degree. Moving the gray strand below the join changes this sign by a factor of $(-1)^{|b|}$.}
\label{braiding}
\end{figure} 

Using that $e_{1,2}$ has degree $2$ so that $(\partial e_{1,2})(a \otimes b \otimes c) = \partial e_{1,2}(a,b;c) - e_{1,2}(\partial(a \otimes b \otimes c)),$ we see
\begin{align*}
\partial e_{1,2}(a;b,c)&= e_{1,2}(\partial a;b,c)+(-1)^{|a|}e_{1,2}(a;\partial b,c)+(-1)^{|a|+|b|}e_{1,2}(a;b,\partial c)\\
&-(-1)^{(|a|+1)|b|}b \cdot e_{1,1}(a;c)+e_{1,1}(a;bc)-e_{1,1}(a;b)c
\end{align*}

It follows that $e_{1,2}$ measures the failure of the left Hirsch formula to hold, and that the operation
\begin{equation*}
E_{1,2}(a;b,c)=(-1)^{|b|+1}e_{1,2}(a;b,c)
\end{equation*}
satisfies the identities of Definition \ref{def:Hirsch}.
\\
\par
The differential of $e_{2,1}$ is described in Figure \ref{diffe21}. The first term vanishes because of the middle relation in Figure \ref{relbasicops} (recall that $\Delta$ is coassociative, so that we can drag the bottom $\Delta$ to the top of the diagram). Again after some manipulations we obtain
$$(\partial e_{2,1})(a,b;c)=(-1)^{|a|}a\cdot e_{1,1}(b;c)+(-1)^{|b||c|}e_{1,1}(a;c)\cdot b-e_{1,1}(ab;c).$$
Here, the $(-1)^{|a|}$ in front of the first term on the right hand side comes form the relation (\ref{signtens}), because the join operation $\ast$ has degree $1$.
\par
We see that $e_{2,1}$ measures the failure of the right Hirsch identity, and that it corresponds to
\begin{equation*}
e_{2,1}(a,b;c)=(-1)^{|b|+1}E_{2,1}(a,b;c).
\end{equation*}
\begin{figure}[H]
  \centering
\def\svgwidth{0.9\textwidth}
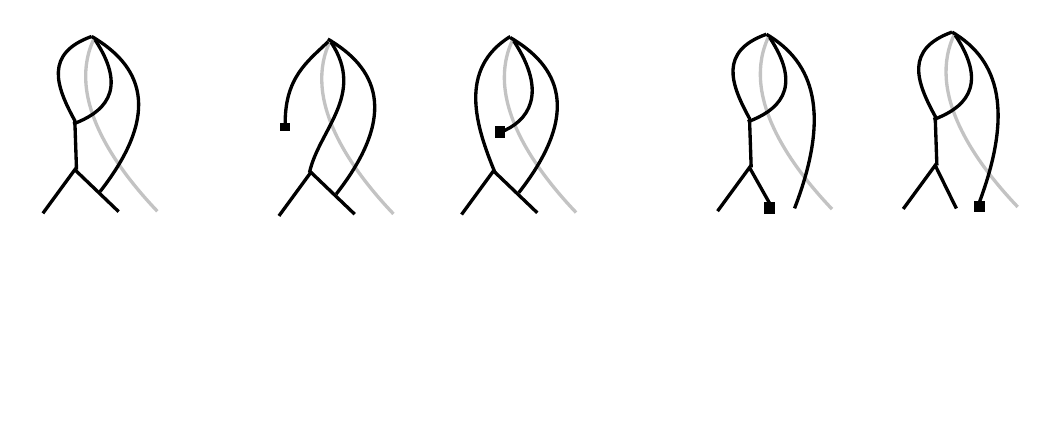
\caption{The differential in the Hom complex of the operation $e^{2,1}$. Here, the first term in the right-hand side of the first row vanishes because by coassociativity we can move the bottom $\Delta$ all the way to the top vertex; we then see that our graph contains a `square' (coming from the two rightmost edges) as in the middle relation of Figure \ref{relbasicops}.}
\label{diffe21}
\end{figure} 

We now describe a recursive way to determine all the (combinations) of graphs corresponding to the operations $e_{p,q}$. Consider an immersed downward-flowing graph $G$ with one incoming end and $p+q$ outgoing ends, corresponding to a composition of the two basic operations $\Delta $ and $*$. We define the upper and lower $(p,q)$-descendants of $G$ to be the graphs described in Figure $6$, which have $p+q+1$ outgoing ends. For each, we add two new strands, one which is fed back into another strand and one which gives rise to a new output. These descendants are called `upper' or `lower' depending on whether the new output strand begins its life above or below $G$. 

\begin{figure}[H]
  \centering
\def\svgwidth{0.9\textwidth}
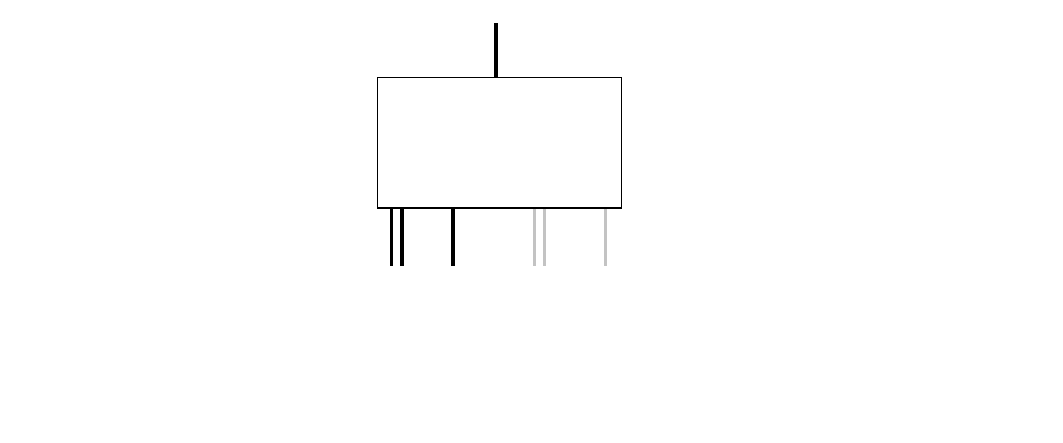
\caption{The lower and upper descendants of a graph $G$ (on the left and right respectively) with $p$ black and $q$ gray outgoing edges.}
\end{figure} 

Notice that the graphs of $e_{2,1}$ and $e_{1,2}$ are respectively the upper and the lower $(1,1)$-descendants of the graph of $e_{1,1}$. In general, we have the following.
\begin{definition}\label{epqdef}
The operation $e_{p,q}$ is defined recursively to correspond to the signed sum over graphs obtained by taking the union of the upper descendants of the graphs in $e_{p,q-1}$ and the lower descendants of the graphs in $e_{p-1,q}$, the latter with an extra sign of $(-1)^{q-1}$. See Figure \ref{descendantse11} for the cases with $p+q\leq4$.
\end{definition}

\begin{figure}[H]
  \centering
\def\svgwidth{0.9\textwidth}
\begingroup%
  \makeatletter%
  \providecommand\color[2][]{%
    \errmessage{(Inkscape) Color is used for the text in Inkscape, but the package 'color.sty' is not loaded}%
    \renewcommand\color[2][]{}%
  }%
  \providecommand\transparent[1]{%
    \errmessage{(Inkscape) Transparency is used (non-zero) for the text in Inkscape, but the package 'transparent.sty' is not loaded}%
    \renewcommand\transparent[1]{}%
  }%
  \providecommand\rotatebox[2]{#2}%
  \newcommand*\fsize{\dimexpr\f@size pt\relax}%
  \newcommand*\lineheight[1]{\fontsize{\fsize}{#1\fsize}\selectfont}%
  \ifx\svgwidth\undefined%
    \setlength{\unitlength}{510.23622047bp}%
    \ifx\svgscale\undefined%
      \relax%
    \else%
      \setlength{\unitlength}{\unitlength * \real{\svgscale}}%
    \fi%
  \else%
    \setlength{\unitlength}{\svgwidth}%
  \fi%
  \global\let\svgwidth\undefined%
  \global\let\svgscale\undefined%
  \makeatother%
  \begin{picture}(1,0.69444444)%
    \lineheight{1}%
    \setlength\tabcolsep{0pt}%
    \put(0,0){\includegraphics[width=\unitlength,page=1]{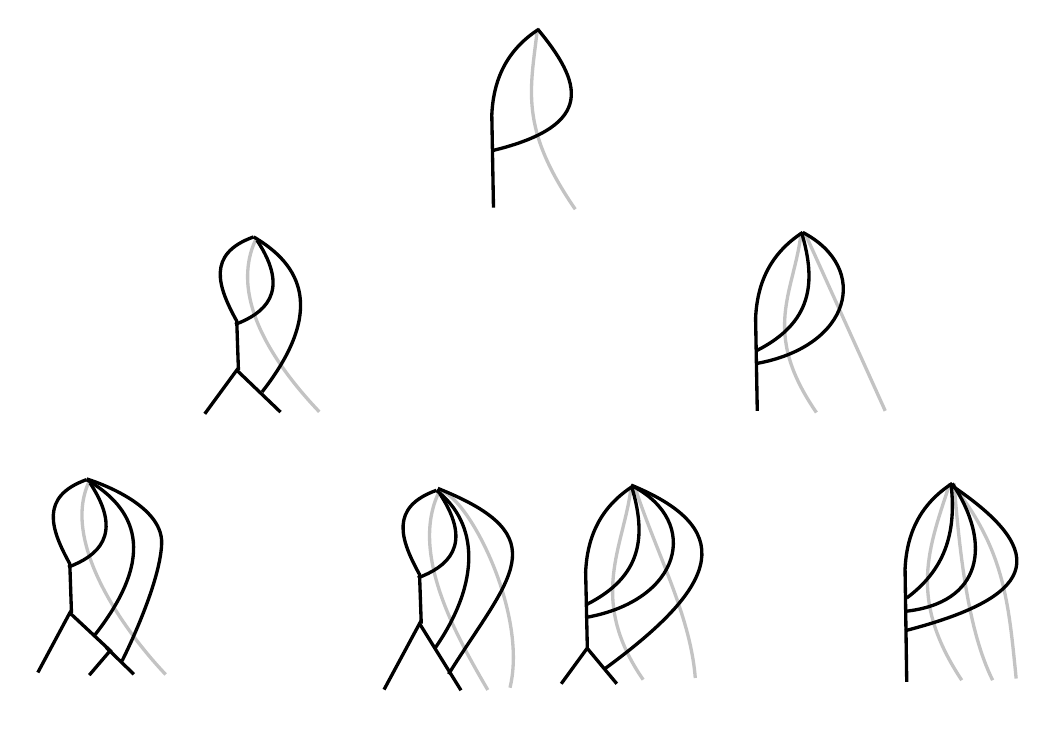}}%
    \put(0.49388889,0.12632813){\color[rgb]{0,0,0}\makebox(0,0)[lt]{\lineheight{0}\smash{\begin{tabular}[t]{l}$-$\end{tabular}}}}%
    \put(0.48620073,0.47201006){\color[rgb]{0,0,0}\makebox(0,0)[lt]{\lineheight{0}\smash{\begin{tabular}[t]{l}$e_{1,1}$\end{tabular}}}}%
    \put(0.7346624,0.26785772){\color[rgb]{0,0,0}\makebox(0,0)[lt]{\lineheight{0}\smash{\begin{tabular}[t]{l}$e_{1,2}$\end{tabular}}}}%
    \put(0.22625208,0.26542568){\color[rgb]{0,0,0}\makebox(0,0)[lt]{\lineheight{0}\smash{\begin{tabular}[t]{l}$e_{2,1}$\end{tabular}}}}%
    \put(0.8781646,0.01676605){\color[rgb]{0,0,0}\makebox(0,0)[lt]{\lineheight{0}\smash{\begin{tabular}[t]{l}$e_{1,3}$\end{tabular}}}}%
    \put(0.47356201,0.02058973){\color[rgb]{0,0,0}\makebox(0,0)[lt]{\lineheight{0}\smash{\begin{tabular}[t]{l}$e_{2,2}$\end{tabular}}}}%
    \put(0.06262416,0.01657389){\color[rgb]{0,0,0}\makebox(0,0)[lt]{\lineheight{0}\smash{\begin{tabular}[t]{l}$e_{3,1}$\end{tabular}}}}%
    \put(0,0){\includegraphics[width=\unitlength,page=2]{descendantse11.pdf}}%
  \end{picture}%
\endgroup%

\caption{The descendant operations on the graphs of $e_{p,q}$ with $p+q\leq4$.}
\label{descendantse11}
\end{figure} 

Up to an overall change of sign (only depending on the degrees of the entries), these new operations satisfy the relations of a Hirsch algebra.

\begin{prop}\label{prop:E-sign}
The operation corresponding to the sum of graphs $e_{p,q}$ satisfies, up to an overall sign change depending on the gradings of the inputs, the relations of the $E_{p,q}$. More precisely, one sets
\begin{equation}\label{signdiff-E}
E_{p,q}(x_1,\dots x_{p+q})=(-1)^{\epsilon_{p,q}^x}\cdot e_{p,q}(x_1,\dots x_{p+q})
\end{equation}
where
\begin{equation*}
\epsilon_{p,q}^x= \frac{(p+q)(p+q-1)}{2}+|x_{p+q-1}|+|x_{p+q-3}|+\dots;
\end{equation*}
here, the quantity $(p+q)(p+q-1)/2$ only depends on $p+q$ modulo $4$, and the sum ends at $|x_1|$ for $p+q$ even and at $|x_2|$ for $p+q$ odd.
\end{prop}

We carry out the proof in the next Section; the techniques needed in the proof will not be relevant elsewhere in the article.
\medskip

\subsection{Proof of the main relation}

Before moving on to the general case, let us discuss very explicitly the case of $e_{1,q}$ (see Figure \ref{e1p} on the left), as it is the case we will be interested in for the actual computations. In this case, the extra complications mentioned below do not arise, and generalizing the computation for $e_{1,1}$ and $e_{1,2}$ above, we obtain
\begin{align*}
(\partial e_{1,q})(a;b_1,\dots, b_p)&=-e_{1,q-1}(a;b_1,\dots, b_{q-1})\cdot b_q\\
&+(-1)^{(|a|+q-1)\cdot |b_1|+q-1}b_1\cdot e_{1,q-1}(a;b_2,\dots, b_{q})\\
&+\sum_{i=1}^{q-1}(-1)^{i+(q-1)}e_{1,q-1}(a; b_1,\dots b_i\cdot b_{i+1},\dots, b_q).
\end{align*}
Counting the $q$ join operations in the graph from the top, the term in the first row corresponds to deleting the right strand in the bottom join while the term in the second row corresponds to deleting the left strand in the top join (see Figure \ref{e1p} in the middle). Here the sign $(-1)^{(|a|+q-1)\cdot |b_1|}$ comes from the analogue of Figure \ref{braiding}, while the extra $(-1)^{q-1}$ comes from the Koszul sign convention for differentials of compositions (\ref{signcomp}). The terms in the last row are obtained by deleting the right strand in the first $q-1$ joins (see Figure \ref{e1p} on the right). It is tedious but straightforward to directly check that the after performing the indicated change of sign we obtain exactly the relations that $E_{1,q}$ satisfies.

\begin{figure}[H]
  \centering
\def\svgwidth{0.9\textwidth}
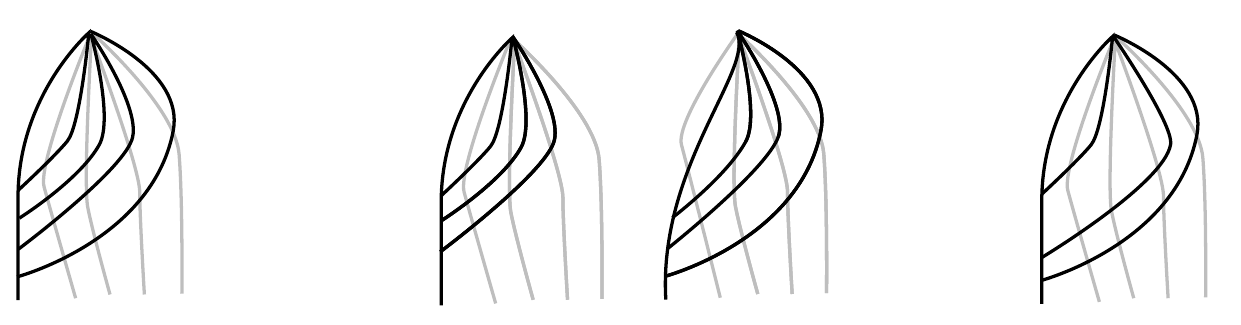
\caption{On the left, the graph corresponding to the operation $e_{1,q}$ for $q=4$. More in general, the operation starts with a rake with $2q+1$ outgoing edges. The edges in odd positions are black, and all merge together. In the middle, the graphs corresponding to the first two rows of the formula for $\partial e_{1,q}$. The terms in the last row are obtained by removing the interior edges; the case $i=2$ is depicted on the left.}
\label{e1p}
\end{figure}

To keep the proof of the general case readable, for the rest of the section we will work without keeping track of signs. These can be recovered in the same way as in the concrete examples we have discussed. In particular, one should be careful of signs introduced when braiding (as in Figure \ref{braiding}), and of signs introduced by (\ref{signtens}) when taking joins between strands which are not the two leftmost (as in the case of the computation for $e_{2,1}$), see also the definition of $f\cdot g$ below.
\begin{remark}
Notice that these signs do not appear in the our case of interest $e_{1,p}$, because always takes the join of the two leftmost strands.
\end{remark}
Because $E_{p,q}$ and $e_{p,q}$ only differ by a sign, we only need to prove the analogous recurrence relation for the differentials of the $e_{p,q}$ as elements of the relevant Hom complexes. To state the relation, it is helpful to name some operations one can carry out on these Hom complexes.
\par
We say a map $$f = f_{p,q}: C \to C^{\otimes p} \otimes C^{\otimes q}$$ is a \emph{$(p,q)$-operation}; the relevant operations for us are represented by trivalent graphs as in the previous section. Given a $(p,q)$-operation $f$, there are a handful of relevant ways of producing new operations (see Figure \ref{newops}): 
\begin{itemize}
\item Suppose $p \ge 1$. For $1 \le i \le p$, we write $fm_i$ for the $(p+1,q)$-operation whose dual is defined by $$(fm_i)^*(x_1, \dots, x_{p+1}; y_1, \dots, y_q) = f^*(x_1, \dots, x_i x_{i+1}, \dots, x_{p+1}; y_1, \dots, y_q);$$ pictorially, this is obtained by adding in the bottom of the picture for $f$ a coproduct on the $i$th black strand.
\item Suppose $q \ge 1$. For $1 \le j \le q$, we write $fn_j$ for the $(p,q+1)$-operation whose dual is defined by $$(fn_j)^*(x_1, \dots, x_p; y_1, \dots, y_{q+1}) = f^*(x_1, \dots, x_p; y_1, \dots, y_j y_{j+1}, \dots y_{q+1});$$ pictorially, this is obtained by adding in the bottom of the picture for $f$ a coproduct on the $j$th grey strand.
\item Suppose $f$ is a $(p,q)$-operation and $g$ is an $(r,s)$-operation for $p,q,r,s \ge 0$. We write $f\cdot g$ for the $(p+r,q+s)$-operation whose dual is defined by\footnote{We spell out here the sign for the interested reader, but will suppress it from now on.} $$(f\cdot g)^*(x_1, \dots, x_{p+r}; y_1, \dots, y_{q+s}) = (-1)^\epsilon f^*(x_1, \dots, x_p; y_1, \dots, y_q) \cdot g^*(x_{p+1}, \dots, x_{p+r}; y_{q+1}, \dots, y_{q+s}),$$ where $$\epsilon = \left(\sum_{i=p+1}^{p+r} |x_{p+i}|\right)\left(\sum_{j=1}^q |y_j|\right) + |g|\left(\sum_{i=1}^p |x_i| + \sum_{j=1}^q |y_j|\right)$$ is the sign of moving $(x_{p+1}, \cdots, x_{p+r})$ across $(y_1, \cdots, y_q)$ together with the Koszul sign from applying $f \otimes g$. Pictorially, this is obtained by placing the pictures of $f$ and $g$ next to each other, and braiding the strands in the bottom.
\end{itemize}

\begin{figure}[H]
  \centering
\def\svgwidth{0.9\textwidth}
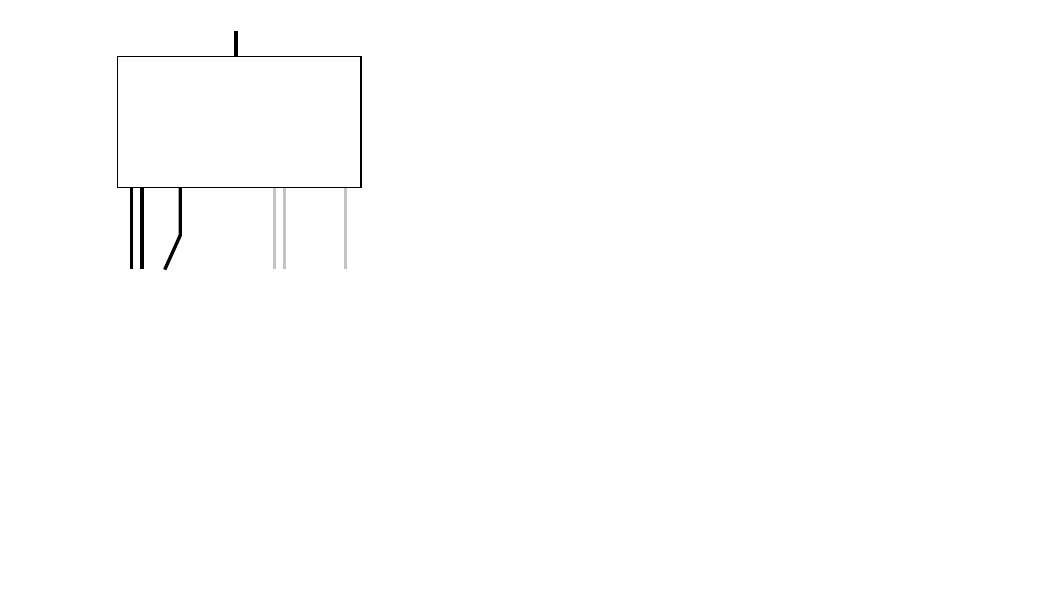
\caption{Some new operations.}
\label{newops}
\end{figure}
Finally, if $f$ is a $(p,q)$-operation for $p \ge 1$, we denote the operations corresponding to the upper and lower descendants by $f^u$ and $f^\ell$; these are respectively $(p,q+1)$- and $(p+1,q)$-operations. For example, we have $e_{1,1} = e_{1,0}^u$. The operation $e_{1,0}^\ell$ is also defined, but it is zero as it contains a bigon. We define $e_{p,q}$ for $(p,q) \ge (1,1)$ via the recursion relation $$e_{p,q} = e_{p,q-1}^u +  e_{p-1,q}^\ell,$$ 
see Figure \ref{Epq} (recall that we are suppressing signs for simplicity).
\begin{figure}[H]
  \centering
\def\svgwidth{0.9\textwidth}
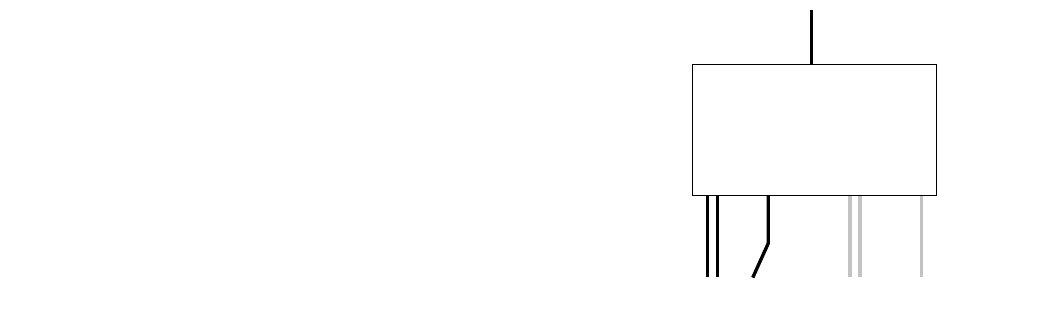
\caption{The recursive definition of $e_{p,q}$ in terms of $e_{p-1,q}$ and $e_{p,q-1}$: the operation $e_{p,q}$ is the sum of the operations corresponding to the depicted graphs. Recall that each of $e_{p-1,q}$ and $e_{p,q-1}$ corresponds to a sum of graphs itself.}
\label{Epq}
\end{figure}

Proposition \ref{prop:E-sign} then follows from the following.
\begin{prop}\label{prop:e-relation}
The boundary of the operation $e_{p,q}$ in \textup{Hom} satisfies the recurrence relation $$de_{p,q} = \sum_{i=1}^{p-1}  e_{p-1,q} m_i + \sum_{j=1}^{q-1}  e_{p,q-1} n_j + \sum_{(p,q) > (i,j) > (0,0)}  e_{p-i,q-j} \cdot e_{i,j}$$
where we suppress signs for simplicity.
\end{prop}

We will prove this by an inductive analysis. To set up the inductive analysis, we will need to understand the behavior of some the different operations with respect to taking upper and lower descendants.

\begin{lemma}\label{lemma:e-op-relations}
We have the following standard relations. 

\begin{enumerate}
    \item $d(e_{p,q-1}^u) = e_{p,q-1} \cdot e_{0,1}+  (de_{p,q-1})^u.$
    \item $d(e_{p-1,q}^\ell) = e_{p-1,q} \cdot e_{1,0}  +e_{p-1,q} m_{p-1} + (de_{p-1,q})^\ell$.
    \item For $1 \le i \le p-2$, we have $e_{p-1,q-1}^u m_i = (e_{p-1,q-1} m_i)^u$ and $e_{p-2,q}^\ell m_i = (e_{p-2,q} m_i)^\ell$.
    \item For $1 \le j \le q-2$, we have $e_{p,q-2}^u n_j = (e_{p, q-2} n_j)^u$. For $1 \le j \le q-1$, we have $e_{p-1,q-1}^\ell n_j = (e_{p-1, q-1} n_j)^\ell$.
    \item For $i > 0$ and $j-1 \ge 0$, we have $(e_{p-i, q-j} \cdot e_{i,j-1})^u =  e_{p-i, q-j} \cdot e_{i,j-1}^u.$
    \item For $i > 1$ and $j \ge 0$, we have $(e_{p-i, q-j} \cdot e_{i-1, j})^\ell =  e_{p-i,q-j} \cdot e_{i-1,j}^\ell.$
\end{enumerate}

We also have the following `edge' cases:

\begin{enumerate}[label=(\alph*)]
\item We have $(e_{p-1, q-1} m_{p-1})^u = (e_{p-1, q-1} \cdot e_{0,1})^\ell$. 
\item We have $(e_{p, q-2} \cdot e_{0,1})^u = e_{p,q-2}^u n_{q-1}.$
\end{enumerate}
\end{lemma}

Let us discuss how these identities can be proved pictorially. Regarding $(1)$ and $(2)$, when taking differentials one either resolves joins inside the box, or the new one introduced by taking the descendant. The former correspond to the descendants of the differential of the operation in the box, i.e. $(de_{p,q-1})^u$ and $(de_{p-1,q})^l$ respectively. The latter are identified with the claimed terms in Figures \ref{Epqdiff} and \ref{Epqdiffsimplified}.

\begin{figure}[H]
  \centering
\def\svgwidth{0.9\textwidth}
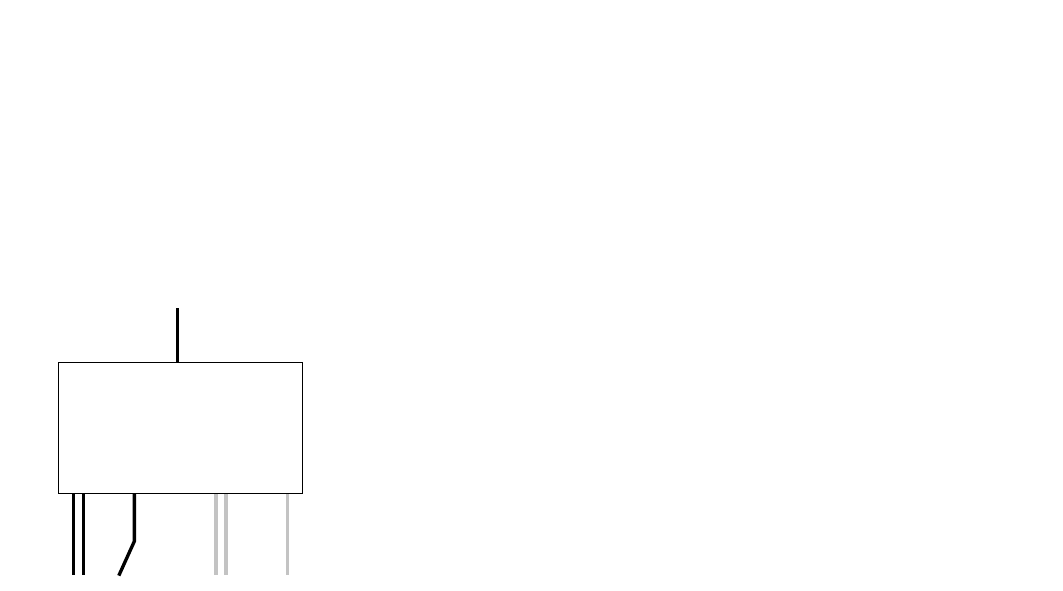
\caption{The terms in the differential of Figure \ref{Epq} obtained by resolving the join outside the box.}
\label{Epqdiff}
\end{figure}
\begin{figure}[H]
  \centering
\def\svgwidth{0.9\textwidth}
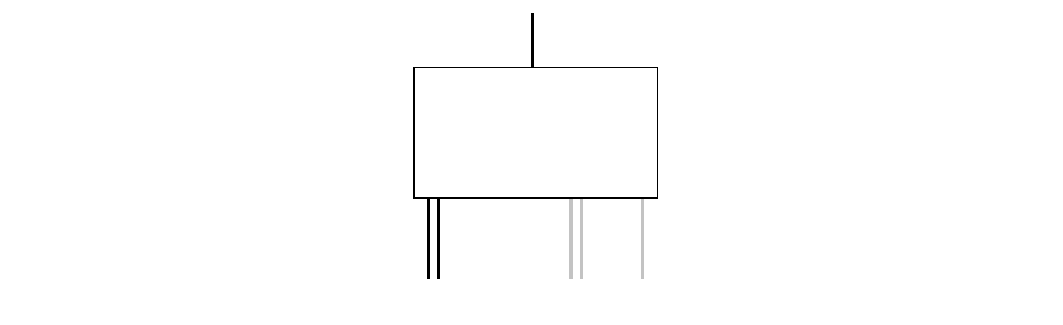
\caption{Simplifying the terms appearing in Figure \ref{Epqdiff}.}
\label{Epqdiffsimplified}
\end{figure}

The identity $(3)$ is proved via Figure \ref{Ep-1q} by moving the coproduct in the box to the bottom of the picture; notice that for the second, when $i = p-2$ one also uses coassociativity of the rake.
\begin{figure}[H]
  \centering
\def\svgwidth{0.9\textwidth}
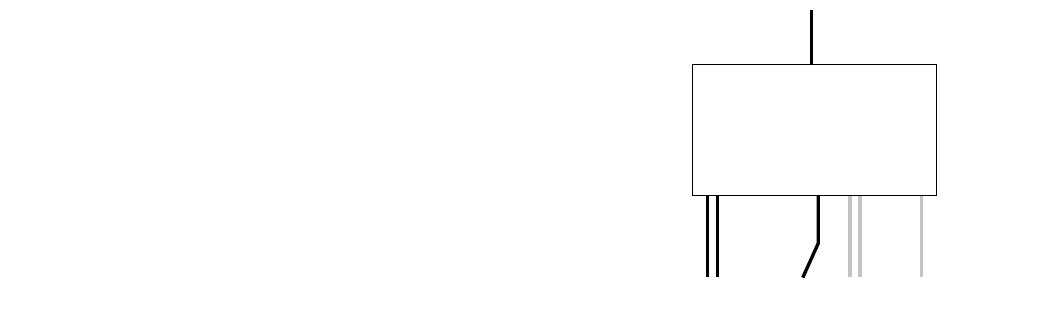
\caption{On the left $(e_{p-1,q-1} m_i)^u$ while on the right $(e_{p-2,q} m_i)^\ell$.}
\label{Ep-1q}
\end{figure}

Identity $(4)$ is proved in the same way via Figure \ref{Epq-1}.

\begin{figure}[H]
  \centering
\def\svgwidth{0.9\textwidth}
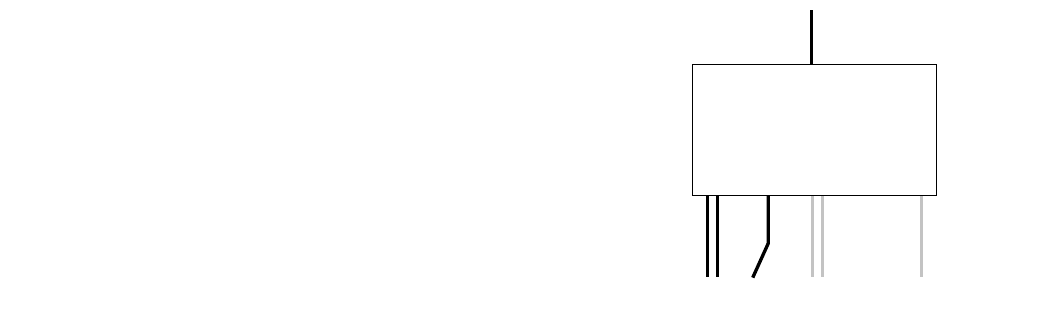
\caption{On the left, $(e_{p, q-2} n_j)^u$ while on the right $(e_{p-1, q-1} n_j)^\ell$. The identity is proved by moving the coproduct in the box to the bottom.}
\label{Epq-1}
\end{figure}

Identity $(5)$ and $(6)$ are proved via Figure \ref{prodtermdiff}.

 \begin{figure}[H]
  \centering
\def\svgwidth{0.9\textwidth}
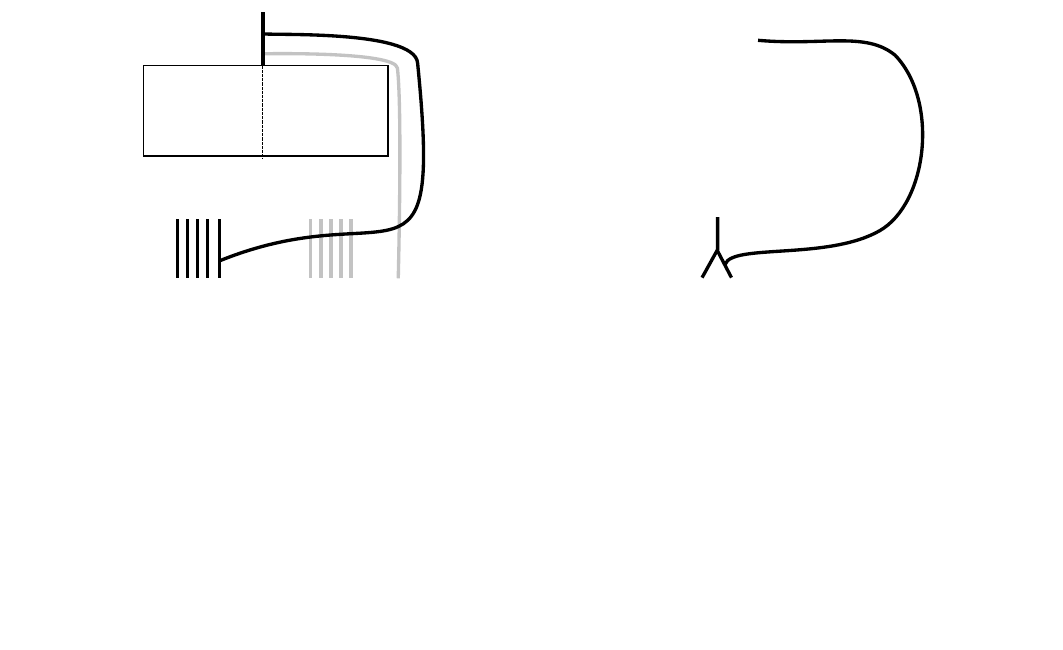
\caption{On the left $(e_{p-i, q-j} \cdot e_{i,j-1})^u=e_{p-i, q-j} \cdot e_{i,j-1}^u$, while on the right $(e_{p-i,q-j} \cdot e_{i-1,j})^\ell=e_{p-i,q-j} \cdot e_{i-1,j}^l.$ The empty boxes represent the braiding as in the bottom of Figure \ref{newops}. The identity is proved by moving the relevant part across the empty box.}
\label{prodtermdiff}
\end{figure}

Finally, the relations $(a)$ and $(b)$ are obtained by inspecting Figures \ref{leftover2} and \ref{leftover1} respectively.

\begin{figure}[H]
  \centering
\def\svgwidth{0.9\textwidth}
\begingroup%
  \makeatletter%
  \providecommand\color[2][]{%
    \errmessage{(Inkscape) Color is used for the text in Inkscape, but the package 'color.sty' is not loaded}%
    \renewcommand\color[2][]{}%
  }%
  \providecommand\transparent[1]{%
    \errmessage{(Inkscape) Transparency is used (non-zero) for the text in Inkscape, but the package 'transparent.sty' is not loaded}%
    \renewcommand\transparent[1]{}%
  }%
  \providecommand\rotatebox[2]{#2}%
  \newcommand*\fsize{\dimexpr\f@size pt\relax}%
  \newcommand*\lineheight[1]{\fontsize{\fsize}{#1\fsize}\selectfont}%
  \ifx\svgwidth\undefined%
    \setlength{\unitlength}{510.23622047bp}%
    \ifx\svgscale\undefined%
      \relax%
    \else%
      \setlength{\unitlength}{\unitlength * \real{\svgscale}}%
    \fi%
  \else%
    \setlength{\unitlength}{\svgwidth}%
  \fi%
  \global\let\svgwidth\undefined%
  \global\let\svgscale\undefined%
  \makeatother%
  \begin{picture}(1,0.30555556)%
    \lineheight{1}%
    \setlength\tabcolsep{0pt}%
    \put(0,0){\includegraphics[width=\unitlength,page=1]{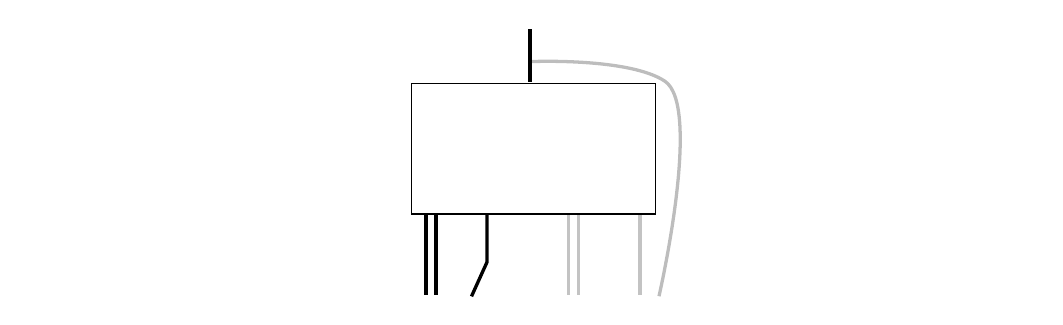}}%
    \put(0.49020752,0.15290324){\color[rgb]{0,0,0}\makebox(0,0)[lt]{\lineheight{0}\smash{\begin{tabular}[t]{l}$e_{p-1,q-1}$\end{tabular}}}}%
    \put(0,0){\includegraphics[width=\unitlength,page=2]{leftover2.pdf}}%
    \put(0.43864836,0.06737763){\color[rgb]{0,0,0}\makebox(0,0)[lt]{\lineheight{0}\smash{\begin{tabular}[t]{l}$\cdot$\end{tabular}}}}%
    \put(0.57113606,0.04632293){\color[rgb]{0,0,0}\makebox(0,0)[lt]{\lineheight{0}\smash{\begin{tabular}[t]{l}$\textcolor{gray}{\cdot}$\end{tabular}}}}%
    \put(0,0){\includegraphics[width=\unitlength,page=3]{leftover2.pdf}}%
  \end{picture}%
\endgroup%

\caption{This picture can be interpreted both as $(e_{p-1, q-1} m_{p-1})^u$ and $(e_{p-1, q-1} \cdot e_{0,1})^\ell$.}
\label{leftover2}
\end{figure}  

\begin{figure}[H]
  \centering
\def\svgwidth{0.9\textwidth}
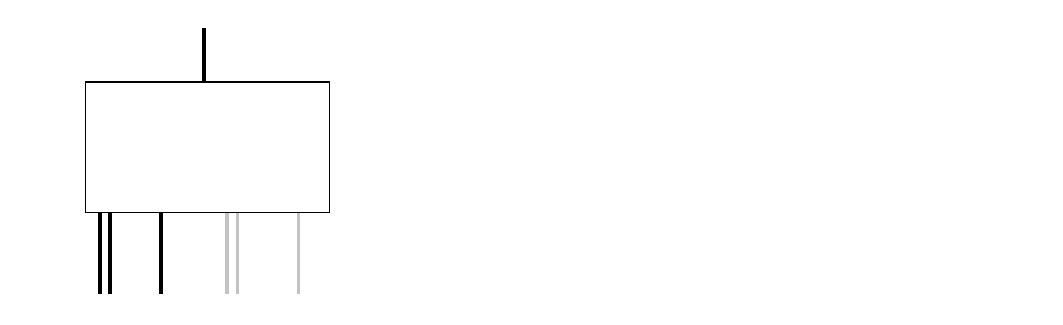
\caption{The pictures of $(e_{p,q-2}\cdot e_{0,1})^u n_{q-1}$ and $e_{p,q-2}^u n_{q-1}.$}
\label{leftover1}
\end{figure}

\begin{proof}[Proof of Proposition \ref{prop:e-relation}]
We prove this by induction on $(p,q)$. The claims for $(p,q) = (1,1), (2,1)$, and $(1,2)$ were shown in the preceding discussion. We will first establish the claim for $e_{p,1}$ and $e_{1,q}$, then $e_{p,q}$ for $p,q \ge 2$. Because it is straightforward to include signs but they obscure the argument, we suppress them.

For $e_{p,1}$ with $p \ge 3$, we have by Lemma \ref{lemma:e-op-relations}(2) that \begin{equation}\label{induct-p}
de_{p,1} = d(e_{p-1,1}^\ell) = e_{p-1,1} \cdot e_{1,0} + e_{p-1,1} m_{p-1} + (de_{p-1,1})^\ell.
\end{equation}
Using the inductive hypothesis, we may write $$(de_{p-1,1})^\ell = \sum_{i=1}^{p-2} (e_{p-2,1} m_i)^\ell + (e_{p-2,1} \cdot e_{1,0})^\ell + (e_{1,0} \cdot e_{p-2,1})^\ell.$$ Applying Lemma \ref{lemma:e-op-relations}(3) and (6) and the assumption $p \ge 3$ we obtain $$\sum_{i=1}^{p-2} e_{p-2,1}^\ell m_i + e_{p-2,1} \cdot e_{1,0}^\ell + e_{1,0} \cdot e_{p-2,1}^\ell.$$ The second term contains a bigon, hence vanishes; the other terms simplify to $\sum_{i=1}^{p-2} e_{p-1,1} m_i + e_{1,0} \cdot e_{p-1,1}$. Substituting this into identity (\ref{induct-p}), we obtain the desired relation for $de_{p,1}$. 

The argument for $e_{1,q}$ when $q \ge 3$ is similar, and in any case was discussed with signs in the previous section. Now suppose $(p,q) \ge (2,2)$, and the boundary relation for $de_{i,j}$ is known for all $(i,j) < (p,q)$ (meaning $i \le p$ and $j \le q$ and at least one of those inequalities is strict).

By Lemma \ref{lemma:e-op-relations}(1)-(2) and the definition $e_{p,q} = e_{p,q-1}^u + e_{p-1,q}^\ell$, we see that we have \begin{align*}
de_{p,q} &= de_{p,q-1}^u + (-1)^{q-1} de_{p-1,q}^\ell \\
        &= e_{p,q-1} \cdot e_{0,1} + e_{p-1,q} \cdot e_{1,0} + (-1)^q e_{p-1,q} m_{p-1} + (de_{p,q-1})^u +  (de_{p-1,q})^\ell.
\end{align*}
Using the inductive hypothesis, recall that $$de_{p,q-1} = \sum_{i=1}^{p-1} e_{p-1,q-1} m_i + \sum_{j=1}^{q-2} e_{p,q-2} n_j + \sum_{(p,q-1) > (i,j-1) > (0,0)} e_{p-i,q-j} \cdot e_{i,j-1}$$ and $$de_{p-1,q} = \sum_{i=1}^{p-2} e_{p-2,q} m_i + \sum_{j=1}^{q-1} e_{p-1,q-1} n_j + \sum_{(p-1,q) > (i-1,j) > (0,0)} e_{p-i,q-j} \cdot e_{i-1,j}.$$
Applying Lemma \ref{lemma:e-op-relations}(3) and (4), the first terms of $(de_{p,q-1})^u + (de_{p-1,q})^\ell$ simplify to $$(e_{p-1,q-1} m_{p-1})^u + \sum_{i=1}^{p-2} \left(e_{p-1,q-1}^u m_i + e_{p-2,q}^\ell m_i\right) = (e_{p-1,q-1} m_{p-1})^u + \sum_{i=1}^{p-2} e_{p-1,q} m_i,$$ while the second terms simplify to $$e_{p-1,q-1}^\ell n_{q-1} + \sum_{j=1}^{q-2} \left(e_{p,q-2}^u n_j + e_{p-1,q-1}^\ell n_j\right) = e_{p-1,q-1}^\ell n_{q-1} + \sum_{j=1}^{q-2} e_{p,q-1} n_j.$$ This gives almost precisely the terms in the first two parts of the desired relation (there is an extraneous $(e_{p-1, q-1} m_{p-1})^u$, a missing $e_{p-1,q} m_{p-1}$, and a missing $e_{p,q-2}^u n_{q-1}$; but notice that $e_{p-1,q} m_{p-1}$ arose elsewhere in the differential of $de_{p,q}$ itself).

Now apply Lemma \ref{lemma:e-op-relations}(5)-(6) to the upper and lower descendants of the last terms. We obtain $$(e_{p, q-2} \cdot e_{0,1})^u + (e_{p-1,q-1} \cdot e_{0,1})^\ell + \sum_{(p,q-1) > (i,j-1) \ge (1,0)} e_{p-i,q-j} \cdot e_{i,j-1}^u + \sum_{(p-1,q) > (i-1,j) \ge (1,0)} e_{p-i,q-j} \cdot e_{i-1,j}^\ell.$$
We investigate the terms contributing to $e_{p-i,q-j} \cdot e_{i,j}$ from this sum. Because whether or not $e_{i,j}$ is defined in terms of one descendant or both depends on whether or not $i,j \ge 2$, we have to do some case analysis.

\begin{itemize}
    \item When $(p,q) > (i,j) \ge (2,2)$, the expression $e_{p-i, q-j} \cdot (e_{i,j-1}^u + e_{i-1,j}^\ell) = e_{p-i,q-j} \cdot e_{i,j}$ appears in the sum above. 
    \item For $(i,j) = (1, j)$ and $j \ge 1$, we have $e_{1,j} = e_{1,j-1}^u$ and only this term appears in the sum above. 
    \item For $(i,j) = (i,1)$ and $i \ge 2$, we have $e_{i,1} = e_{i-1,1}^\ell$, but there are in principle two terms contributing to $e_{p-i,q-1} \cdot e_{i,1}$ in the sum above, which contains $e_{p-i,q-1} \cdot (e_{i,0}^u + e_{i-1,1}^\ell)$. However, because $i \ge 2$, we have $e_{i,0} = 0$ and thus the first term is zero. 
    \item Neither $e_{p-1,q} \cdot e_{1,0}$ nor $e_{p,q-1} \cdot e_{0,1}$ arise from the expressions in the larger sums. 
\end{itemize}
Thus this large sum simplifies to $\sum_{(p,q) > (i,j) \ge (1,1)} e_{p-i,q-j} \cdot e_{i,j}$. Further, recall from the expression for $de_{p,q}$ at the beginning of the argument that $e_{p,q-1} \cdot e_{0,1}$ and $e_{p-1,q} \cdot e_{1,0}$ arise elsewhere in the differential of $de_{p,q}$, accounting for all remaining $(i,j) > (0,0)$. We have shown so far that \begin{align*}
de_{p,q} &= \sum_{i=1}^{p-1} e_{p-1,q} m_i + \sum_{j=1}^{q-2} e_{p,q-1} n_j + \sum_{(p,q) > (i,j) > (0,0)} e_{p-i,q-j} \cdot e_{i,j} \\
&+ (e_{p-1, q-1} m_{p-1})^u + e_{p-1,q-1}^\ell n_{q-1} + (e_{p,q-2} \cdot e_{0,1})^u + (e_{p-1,q-1} \cdot e_{0,1})^\ell.
\end{align*}
Now apply Lemma \ref{lemma:e-op-relations}(a)-(b) to see that the first and last terms cancel out, while the second and third terms sum to give $e_{p,q-1} n_{q-1}$. Thus this gives us the desired relation $$de_{p,q} = \sum_{i=1}^{p-1} e_{p-1,q} m_i + \sum_{j=1}^{q-1} e_{p,q-1} n_j + \sum_{(p,q) > (i,j) > (0,0)} e_{p-i, q-j} \cdot e_{i,j}.$$
This completes the induction.
\end{proof}

\medskip

\subsection{Some explicit computations in the cubical setting}
In this section we give an explicit combinatorial interpretation of the operations $E_{1,p}$ in the case of cubical sets; this will be needed in Section 5.2 when discussing the Kraines construction in the cubical cochains on the torus.

To get first some concrete understanding, let us begin by computing the operations $E_{1,1}$ and $E_{1,2}$ in the case of the square. It will be useful to refer back to Figures \ref{e11} and \ref{e12e21}, which include diagrammatic representations of both of these operations. Recall that we have
\begin{equation*}
\Delta(II)=00|II-0I|I1+I0|1I+II|11;
\end{equation*}
here we write bars between tensor factors instead of $\otimes$ to save space.

Computing the diagonal on the first term again we obtain
\begin{align*}
(\Delta\otimes\mathrm{Id})\circ\Delta(II)&=00|00|II-00|0I|I1-0I|01|I1
+\underline{00|I0|1I}\\&+I0|10|1I+\underline{00|II|11}-0I|I1|11+\underline{I0|1I|11}+II|11|11.
\end{align*}
The underlined terms are the only ones for which the join of the first and third terms is non trivial; the result of the operation associated to the graph of Figure $2$ is then
\begin{equation}\label{e11II}
e^{1,1}(II)=(0I+I1)|II-II|(I0+1I)
\end{equation}
Here the extra minus sign for $II|I0$ is introduced because in order to compute the operation we need first to swap the second and third term; and for $00|I0|1I$ a sign is introduced as both are odd. Recalling now that $e_{1,1}$ is the dual of $e^{1,1}$, and we use exponentials to denote the dual basis, we see from (\ref{dualmap}) that
\begin{align*}
e_{1,1}(e^{0I},e^{II})(II)=&-(e^{0I}\otimes e^{II})(e^{1,1}(II))=\\
&-(e^{0I}\otimes e^{II})(0I\otimes II)=-1
\end{align*}
so that $e_{1,1}(e^{0I},e^{II})=-e^{II}$. Here the minus sign in the first row arises because $|e_{1,1}|\cdot(|e^{0I}|+|e^{II}|)$ is odd, and no sign introduced in the second row by (\ref{signtens}). An analogous computation shows
\begin{align*}
e_{1,1}(e^{0I};e^{II})=e_{1,1}(e^{I1};e^{II})&=-e^{II}\\
e_{1,1}(e^{II};e^{I0})=e_{1,1}(e^{II};e^{1I})&=e^{II}.
\end{align*}

Recalling the sign discrepancy $E_{1,1}(a;b)=(-1)^{|a|+1}e_{1,1}(a;b)$, we obtain that the only non trivial $E_{1,1}$ operations on the square with degree two output are

\begin{align*}
E_{1,1}(e^{0I};e^{II})=E_{1,1}(e^{I1};e^{II})=E_{1,1}(e^{II};e^{I0})=E_{1,1}(e^{II};e^{1I})=-e^{II}.
\end{align*}
Similarly, one can compute the effect of the graph associated to $e^{1,2}$ on $II$; there are many terms appearing after applying $\Delta$ five times, but the only term for which the iterated join of the first, third and fifth element are is non-trivial is
\begin{equation*}
00|I0|10|1I|11.
\end{equation*}
This corresponds to the output
\begin{equation*}
e^{1,2}(II)=-II|I0|1I,
\end{equation*} where we use the computation
\begin{equation*}
(00*10)*11=I0*11=-II,
\end{equation*}
and no extra sign is introduced when permuting the elements. This in turn means $$e_{1,2}(e^{II};e^{I0},e^{1I})=e^{II},$$ where we use that $e_{1,2}$ has even degree and that there is an extra minus sign coming from (\ref{signtens}); because of the sign discrepancy $E_{1,2}(a;b,c)=(-1)^{|b|+1}e_{1,2}(a;b,c)$, we have $E_{1,2}(e^{II};e^{I0},e^{1I})=e^{II}$.
\\
\par
To perform more involved computations it will be necessary to have a concrete understanding of these iterated coproducts in general. In what follows, to describe a basis element of $C_*^\cube(I^n)$ we will pass freely between the geometric language of \emph{faces} of the cube $I^n$ and length $n$ strings in $0,I,1$; a $k$-dimensional face corresponds to a string with $k$ appearances of $I$. Given a face $F$ corresponding to a string $x$, we say $$\text{min}(x) = \text{min}(x_1) \cdots \text{min}(x_n),$$ and similarly for $\text{max}(x)$, where $$\text{min}(0) = \text{max}(0) = 0, \quad \text{min}(1) = \text{max}(1) = 1, \quad \text{min}(I) = 0, \;\; \text{max}(I) = 1.$$ 

For instance, $$\text{min}(0I1I) = 0010, \quad \text{max}(0I1I) = 0111.$$

The following result is a direct consequence of the explicit formula for the coproduct by induction on $k$.
\begin{lemma}\label{itcoprod}
On the $n$-cube, the $k$-fold iterated coproduct of $I^n$ is given by the signed sum of all sequences of faces $F_1|\cdots| F_k$ for which:
\begin{itemize}
\item $\min F_0=0^n$ and $\max F_k=1^n$
\item $\max F_i=\min F_{i+1}$ for $i=1,\dots k-1.$
\end{itemize}
For each $j$, the $j$'th coordinate is equal to $I$ in exactly one of the $F_i$'s. Therefore we may associate a permutation $\pi: \{1, \dots, n\} \to \{1, \dots, n\}$ so that $\pi(1)$ is the position of the first $I$ in $F_1$, through $\pi(|F_1|)$ the final $I$ in $F_1$; then $\pi(|F_1| + 1)$ is the position of the first $I$ in $F_2$, and so on. Then $F_1 | \cdots | F_k$ appears with sign $(-1)^{|\pi|}$, the sign of the above permutation.
\end{lemma}

The last statement can also be rephrased as follows: for each face $F_i$, consider the ordered string $I_i$ of the positions at which $I$ appears; then the sign of $F_1|\cdots| F_k$ in the iterated coproduct is the sign of the concatenation of the strings $(I_1,\dots, I_k)$, thought of as a permutation of $\{1,\dots,n\}$. 
\\
\par
Joins can be described in terms of faces in the following way. For a $k$-dimensional face $F$ (corresponding to a string $x_1,...,x_n$ with $k$ entries equal to $I$), we define its \textit{back edges} to be the edges obtained by substituting to the substring of $I$s a string of zeroes, followed by exactly one $I$, followed by ones. For example, the back edges of the cube $I^3$ are $00I$, $0I1$ and $I11$, see Figure \ref{frontback}. Similarly, we define its front edges to be the edges obtained by substituting to the substring of $I$s a string of ones, followed by exactly one $I$, followed by zeroes. For example, the front edges of $III$ are $I00$, $1I0$ and $11I$.

\begin{figure}[H]
  \centering
\def\svgwidth{0.9\textwidth}
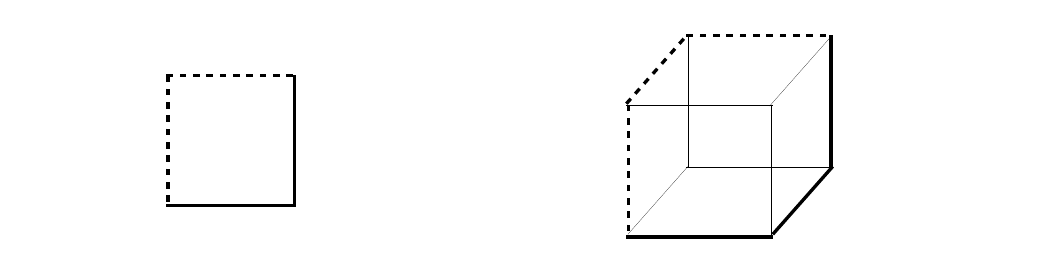
\caption{The front (solid) and back (dashed) edges for the square and the cube.}
\label{frontback}
\end{figure}

We may give a description of the join of two faces as follows (for instance, by induction on $n$).

\begin{lemma}\label{itjoin}
Consider two faces $F_0,F_1$ of $I^n$ for which we have $\max F_0\leq \min F_1$. When this is the case, $F_0 * F_1 = \sum (-1)^{|F_0|} G$, where the sum is taken over every face $G$ of dimension $|F_0|+|F_1|+1$ so that:
\begin{itemize}
\item $\min F_0\leq \min G$,
\item $\max G\leq \max F_1$,
\item If $F$ is the unique face with $\max F_0=\min F$ and $\max F=\min F_1$, then $G \cap F$ is an edge, which is a back edge of $F$ and a front edge of $G$.
\end{itemize}
\end{lemma}

We are now prepared to give a completely explicit description of the operations $E_{1,p}$.

\begin{prop}\label{E1p}
Given faces $G_0, G_1,\dots G_p$ of $I^n$, we have that $E_{1,p}(e^{G_0};e^{G_1},\dots,e^{G_p})=\pm e^{II\cdots I}$ if and only if the following hold:
\begin{enumerate}
\item $\max G_{i}\leq\min G_{i+1}$ for $i=1,\dots p-1$;
\item for each $i=1,\dots p$, $G_{0}\cap G_{i}$ is either empty or one dimensional; if it is nonempty, it is an edge which is a front edge for $G_0$ and a back edge of for $G_{i}$.
\item for each $i\geq 1$, there is exactly one position $k_i$ so that both strings $F_0$ and $G_i$ have $k_i$'th entry $I$, and furthermore the sequence $k_1,\dots k_p$ is increasing.
\item If $1\leq i<j$, $G_{i}$ and $G_{j}$ do not have an entry $I$ at the same position. Furthermore for each $k$ there is a face $G_i$ (possibly $i=0$) whose $k$th position is $I$.
\end{enumerate}
\textbf{Provided that the faces $G_0,G_1\dots, G_p$ are odd dimensional}, the sign is given by
\begin{equation*}
(-1)^{p}\cdot\mathrm{sign}(\pi),    
\end{equation*}
where $\pi$ is the permutation of $\{1,\dots, n\}$ obtained by replacing in the string of positions of the $I$'s in $G_0$, the position $k_i$ with the string of positions of the $I$'s in $G_i$, for all $i\geq 1$. For example, if $G_0$ has string of positions $2345$ while $G_1$ and $G_2 $ have strings of positions $267$ and $1489$, then $\mathrm{sign}(\pi)$ is the sign of the permutation $(267314895)$.
\end{prop}
\begin{remark}
The computation of the operations $E_{p,1}$ leads to a similar result with the role of back and front reversed. This symmetry (up to sign) is rather surprising because the graphs of $E_{p,1}$ and $E_{1,p}$ are very different. Indeed, it does not at all hold in general Hirsch algebras; in simplicial cochains, we have $e_{p,q} = 0$ for all $p > 1$, while the $e_{1,q}$ are in general nonzero.
\end{remark}

The proof is best elucidated by working out some more involved concrete examples on $I^3$ first. Let us focus on the non-trivial products $E_{1,2}(a;b,c)$ with $|a|=2,|b|=1,|c|=2$. One can see that these arise from two terms in the iterated diagonal, namely
\begin{equation*}
000|I00|100|1II|111,
\end{equation*}
for which iterated join is
\begin{equation*}
(000*100)*111=I00*111=-II1-I0I
\end{equation*}
and corresponds after dualizing to the products
\begin{equation*}
e_{1,2}(e^{II1}; e^{I00},e^{1II})=e_{1,2}(e^{I0I}; e^{I00},e^{1II})=-e^{III},
\end{equation*}
hence
\begin{equation*}
E_{1,2}(e^{II1}; e^{I00},e^{1II})=E_{1,2}(e^{I0I}; e^{I00},e^{1II})=-e^{III},
\end{equation*}
and
\begin{equation*}
-000|0I0|010|I1I|111,
\end{equation*}
for which the iterated join is
\begin{equation*}
(000*010)*111=0I0*111=-0II    
\end{equation*}
corresponding to $e_{1,2}(e^{0II}; e^{0I0}, e^{I1I})=e^{III}$, hence
\begin{equation*}
E_{1,2}(e^{0II}; e^{0I0}, e^{I1I})=e^{III}.
\end{equation*}
For the sign computation, notice that $e^{1,2}$ has even degree and no sign is introduced when permuting terms or evaluating as in (\ref{signtens}).
See Figure \ref{productse12} for a visualization of such products in terms of front and back edges.

\begin{figure}
  \centering
\def\svgwidth{0.9\textwidth}
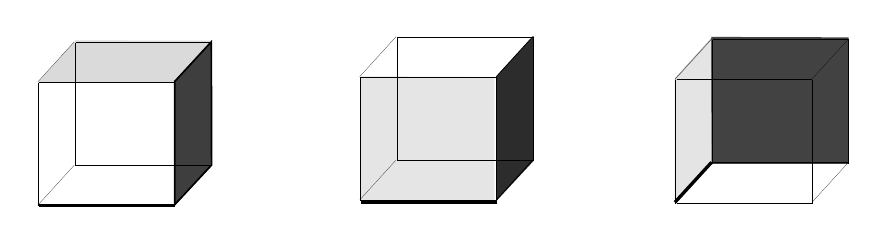
\caption{The non-trivial operations $E_{1,2}(a;b,c)$ with faces $|a|=2$ (the lightly shaded square) and $|b|=1,|c|=2$ (the dark edge and square respectively). When the lightly shaded square and a dark face intersect, the intersection is a front edge for the lightly shaded square and a back edge for the dark face.}
\label{productse12}
\end{figure}

\begin{proof}[Proof of Proposition \ref{E1p}]
Recall that the operation $E_{1,p}$ corresponds up to an overall sign to the graph obtaining by taking a rake with $2p+1$ outputs, and taking joins (in increasing order) of the elements in the odd positions.
\par
We will show that the stated conditions are necessary; sufficiency follows from a similar analysis. Notice that the image of $I^n$ under the $(2p+1)$-fold iterated diagonal is the sum over all sequences of $2p+1$ faces $G_0,F_1,G_1,F_2,G_2,\dots F_p,G_p$ for which
\begin{itemize}
    \item $\max F_{i}=\min G_{i}$ and $\max G_{i}=\min F_{i+i}$
    \item $\min G_0=0^n$
    \item $\max G_p=1^n$.
\end{itemize}
We need then to take (up to signs) the iterated joins of the $G_i$. From this $(1)$ readily follows; notice also that the sequence of the $F_i$ determines the whole sequence of faces. For $(2)$, the main observation is that because $\max G_{i-1}=\min F_{i}$ and $\max F_{i}=\min G_{i}$, by Lemma \ref{itjoin} each face $G$ of the join $G_{i-1}*G_i$ intersects $F_i$ in a face which is in the front for $G$ and in the back for $F_i$. Inductively, taking the iterated join with the following $G_{i+1}$ will either preserve this intersection property, or make the intersection empty (cf. Remark \ref{nonass}). To see that $(3)$ and $(4)$ hold, notice first that for each $k$ there is exactly one of the $G_i$ or $F_i$ whose $k$th element is $I$; the uniqueness of the intersection index $k_i$ follows from the previous discussion, and the fact that the sequence is increasing follows from the fact that the iterated join of the $G_i$ is non-zero.
\par
To see that the sign of the output is the claimed one, we need to be very explicit about all the signs that are introduced. First, suppose $E_{1,p}(e^{G_0}; e^{G_1}, \cdots, e^{G_p}) = \pm e^{II\cdot I}$. Reversing this, we may write $e^{1,p}(I^n)$ as follows. First we take the iterated coproduct $\Delta^{2p+1} I^n$, which gives $$\sum (-1)^{\epsilon_1} [F_0 | G_1 | \cdots | G_p | F_p].$$ Write $I_k$ and $J_k$ for the string of $I$'s in $F_k$ and $G_k$, respectively; then $\epsilon_1$ is the sign of the permutation needed to put $(I_0, J_1, \cdots, J_p, I_p)$ in order (where each $I_k$ and $J_k$ is already listed in its natural order). 

Next we braid the entries to send this to $(-1)^{\epsilon_1 + \epsilon_2} [F_0 | \cdots | F_p | G_1 | \cdots | G_p]$. Because each $G_i$ is assumed odd-degree, this costs a sign of $\epsilon_2 = \sum_{i=1}^p i |F_i|$. 

Next we take the iterated join of the first factors. Because each join costs a sign equal to the degree of the first input and we are iterating this procedure, this produces $(-1)^{\epsilon_1 + \epsilon_2 + \epsilon_3} [G_0 | G_1 | \cdots | G_p]$, where $$\epsilon_3 = |F_0| + (|F_0| + |F_1| + 1) + \cdots + (|F_0| + \cdots + |F_{p-1}| + p-1) = p(p-1)/2 + \sum_{j=0}^{p-1} (p-j)|F_j|.$$ Further, by the assumption that this iterated join is nonzero, we have $I_0 < \{j_1\} < \cdots < \{j_p\} < I_p$, where the string of positions of $G_0$ is $I_0 \cup \{j_1\} \cup \cdots \cup \{j_p\} \cup I_p$ and each $j_i \in J_i$. Thus $(-1)^{\epsilon_1}$ is precisely the term $\text{sign}(\pi)$ appearing in the statement of the Proposition. Next, $$\epsilon_2 + \epsilon_3 = p(p-1)/2 + \sum_{i=0}^p (p|F_i|).$$ By the assumption that $|G_0|$ has odd degree, we know that $|G_0| = p + \sum |F_i|$ is odd, so that $\sum |F_i| \equiv p-1$ mod $2$; thus this expression simplifies to $p(p-1)/2 + p(p-1) \equiv p(p-1)/2$.

We have justified that $e^{1,p}(I^n)$ gives a signed sum over terms $[G_0 | G_1 | \cdots | G_p]$, where the sign is $(-1)^{p(p-1)/2} \text{sign}(\pi)$ when all of the $|G_i|$ are odd. We should now pass to the dual operation. We have \begin{align*}[e_{1,p}(e^{G_0}; e^{G_1}, \cdots, e^{G_p})](I^n) &= (-1)^{\epsilon_4} (e^{G_0} \otimes e^{G_1} \otimes \cdots \otimes e^{G_p})(e^{1,p}(I^n)) \\
&= (-1)^{\epsilon_4+p(p-1)/2}\cdot \text{sign}(\pi) (e^{G_0} \otimes \cdots \otimes e^{G_p})(G_0 \otimes \cdots \otimes G_p).
\end{align*}
Here $\epsilon_4 = p(|G_0| + |G_1| + \cdots + |G_p|) = p(p+1) \equiv 0$ is introduced by the Koszul sign rule for the dual operation. Using the Koszul sign rule for applying tensor products of operators (\ref{signtens}), we find that this simplifies to $$(-1)^{p(p-1)/2 + \epsilon_5} \cdot \text{sign}(\pi) e^{G_0}(G_0) \cdots e^{G_p}(G_p),$$ where (because all of $|G_i|$ are odd) $\epsilon_5 = \sum_{i=1}^p i = p(p+1)/2$. This leaves us with precisely the claimed sign $(-1)^p \cdot \text{sign}(\pi)$.
\end{proof}

\medskip
\subsection{Comparing simplicial and cubical cochains}
We conclude by discussing how our two main examples, simplicial and cubical cochains, are related. This will be important because simplicial techniques are much more flexible geometrically (a simplicial approximation theorem holds, while a cubical approximation theorem is at the very least much more subtle), whereas cubical techniques provide the perfect setting for a minimal model of cochains on the torus.

Consider the cellular collapse map
\begin{align*}
\xi&: I^n\rightarrow \Delta^n\\
(x_1,\dots, x_n)&\mapsto (x_1,x_1x_2,x_1x_2x_3,\dots, x_1x_2\cdots x_n).
\end{align*}
Here we consider the model for $\Delta^n$ given by the elements $(y_1,\dots, y_n)\in I^n$ for which $y_1\geq y_2\geq\dots\geq y_n$. This induces the so called \cite{EMc} Cartan-Serre comparison map
\begin{align*}
\xi_*&:  C^{\Delta-\textrm{sing}}_*(X) \rightarrow  C^{\cube-\textrm{sing}}_*(X)\\
\sigma&\mapsto \sigma\circ \xi.
\end{align*}
We will be interested in its dual map
\begin{equation*}
\xi^*:  C^*_{\cube-\textrm{sing}}(X) \rightarrow  C^*_{\Delta-\textrm{sing}}(X).
\end{equation*}
By \cite[Lemma $12$]{K-MM} or \cite{Ser}, this is a quasi-isomorphism of algebras. They also prove that $\xi^*$ preserves the operations coming from `surjection-like graphs'. We can use this to prove the following statement.

\begin{theorem}\label{comparison}
The dual Cartan-Serre comparison map $\xi^*$ is a quasi-isomorphism of Hirsch algebras, where the Hirsch algebra structure is given by the operations $E_{p,q}$ described in the previous sections.
\end{theorem}
\begin{proof}
This is essentially proved in \cite[Section $5.5$]{K-MM}. There, the authors show that those operations which correspond to surjection-like graphs are preserved by $\xi_*$, and thus their composites do as well. The key point in their argument is that even though the Cartan-Serre collapse map $\xi$ does not preserve joins of faces in general \cite[Section $5.4$]{K-MM}, we at least have that $\xi(x*y)=\xi(x)*\xi(y)$ whenever $x\leq y$, where $\leq$ is the partial order on the faces induced (via tensor product) by $0<I<1$, \cite[Lemma $11$]{K-MM}. In particular, this implies that in our notation
\begin{equation*}
\xi(F*G)=\xi(F)*\xi(G)\text{ if }\max F\leq \min G.
\end{equation*}
Now the graphs which define our operations $e_{p,q}$ are in general not surjection-like, nor are they in general composites of surjection-like graphs; for instance, $e_{2,1}$ is not given by a composite of surjection-like graphs. Nevertheless, these operations only involve taking joins of faces $F,G$ with $\max F\leq \min G$: this follows via the observations about maxima and minima in Lemmas \ref{itcoprod} and \ref{itjoin} from the fact that for our graphs, the black subgraph (which is the only part involving joins) is embedded. 
\end{proof}

\bigskip

\section{Twisting sequences on the torus}

In this section, we specialize our construction to the case of the simplest cubical realization of the torus. Our machinery will then allow to compute certain twisted cohomology groups of the torus in a purely combinatorial way.

\subsection{A Hirsch structure on the exterior algebra}\label{mintoruscomp}

What we would like to assert is that \emph{the torus is formal}. If we have a zig-zag of algebra quasi-isomorphisms from $C^*_\Delta(\mathbb T)$ to its cohomology $\Lambda^* \mathbb Z^b$, for some simplicial torus $\mathbb T$, then we may transfer twisting sequences along this zig-zag. If all the relevant maps are Hirsch algebra maps, we can determine what the given class is transferred to. 

Unfortunately, the most obvious choice of zig-zag (coming from the K\"unneth theorem) consists of dg-algebra quasi-isomorphisms which are not Hirsch algebra maps. It follows from our results below that there \textbf{cannot} be a zigzag of Hirsch algebra maps from $C^*_\Delta(\mathbb T)$ to $\Lambda^*(\mathbb Z^b)$ unless the last term has a non-trivial Hirsch structure (the cup-1 product, for instance, must be nonzero). 

Instead of seeking some ad hoc sequence of Hirsch algebra maps, our philosophy is that there should be a zigzag of Hirsch algebra quasi-isomorphisms between any two models for cochains on the torus. Furthermore, there is a reasonable model for the torus whose cochains have zero differential, so return the exterior algebra itself.

\begin{definition}
We write $T^b_1$ for the cubical set obtained by pasting together opposing faces of the cube $\Box^b$ by the identity map. This cubical set has $\binom{b}{k}$ $k$-cubes, which are in bijection with $k$-element subsets of $\{1, \cdots, b\}$; in particular it has a single vertex and single top face. We call $T^b_1$ the \textit{minimal torus}.
\end{definition}

Notice that the algebra $C^*_\cube(T^b_1)$ is \textbf{canonically} isomorphic, as a dg-algebra, to $\Lambda^*(\mathbb Z^b)$. However, being the normalized cubical cochains on a cubical set, this carries significantly more structure: as discussed in Section \ref{examplesHirsch}, this cochain algebra is naturally a Hirsch algebra. 

In this section, we will describe parts of the Hirsch structure explicitly, as well as define some combinatorial operations which turn out to be related. In the next section we will explain the relationship.

Let $\Lambda^*(\mathbb Z^b)$ be the exterior algebra on $n$ elements $e^1, \cdots, e^b$. Identify a basis for this algebra with expressions of the form $e^I$, where $I \subset \{1, \cdots, b\}$ is a subset; if $I = \{i_1, \cdots, i_k\}$ with $i_1 < \cdots < i_k$, then $$e^I = e^{i_1} \wedge \cdots \wedge e^{i_k}.$$ Under the identification, we have the following.

\begin{prop}\label{E1p-formula}
The Hirsch algebra operation $E_{1,p}$ on $\Lambda^*(\mathbb Z^b)$ is given as follows: the operation $E_{1,p}(e^J; e^{I_1}, \cdots, e^{I_p})$ is nonzero if and only if:
\begin{itemize}
    \item $I_j \cap I_k = \varnothing$ for all $j \ne k$
    \item the intersection $J \cap I_k = \{j_k\}$ is a singleton for all $k$, and $j_1 < \cdots < j_p$.
\end{itemize}
In this case, the output is $\pm e^{J \cup I_1 \cup \cdots \cup I_p}$. To be more precise about the sign, \textbf{assume that $|I_i|$ is odd} for all $i$ and write $J = J_0 < \{j_1\} < J_1 < \cdots < \{j_p\} < J_p$. Then $$E_{1,p}(e^J; e^{I_1}, \cdots, e^{I_p}) = (-1)^p e^{J_0} \wedge e^{I_1} \wedge e^{J_1} \cdots \wedge e^{I_p} \wedge e^{J_p}.$$ That is, one replaces each $j_k$ with the set $I_k$, together with an overall sign of $(-1)^p$.
\end{prop}

\begin{proof}
This directly follows from the computations for cubes in Proposition \ref{E1p}. For example, the configurations in Figure \ref{productse12} correspond to
$$E_{1,2}(e^{12};e^1,e^{23})=E_{1,2}(e^{13};e^1,e^{23})=-E_{1,2}(e^{23};e^2,e^{13})= - e^{123}.$$
Going from the cube to the minimal torus, one essentially forgets about the entries $0,1$ of faces and only keeps track of the position of the $I$. The result then follows from the observation that for each sequence of subsets $I_i$ and $J$ as above, there is exactly one choice of faces $F_i$ and $F_0$ of the cube for which $I$ is at the right position and which satisfy the relations in Proposition \ref{E1p}. 
\end{proof}

In the operations $E_{1,p}$, the first input plays a very different role from the remaining $p$. On the other hand, it turns out that there is a symmetric operation --- which we will call the insertion product --- which captures the behavior of all of the $E_{1,p}$'s and their composites. This insertion product and its properties are then crucial to our analysis of the Kraines construction $K(a)$ on an element $a \in \Lambda^3(\mathbb Z^b)$, as $K(a)_{2n+1}$ is obtained by summing over all possible iterates of the operations $E_{1,p}$ (with appropriate scalar factors). To define this, we must make a brief diversion into combinatorics of subsets of $\{1, \cdots, n\}$.

\begin{definition}
Let $\{I_1, \dots, I_m\}$ be a collection of subsets of $\{1, \cdots, b\}$. Write the elements of $I_j$ as $\{x_j^1 <  \dots < x_j^{i_j}\}$. We say the \textbf{graph realization} $R(I_1, \cdots, I_m)$ is the following graph. Each $I_j$ is associated to the connected graph with $|I_j|$ vertices (labeled by $\{x_j^1, \dots, x_j^{i_j}\}$)  and $|I_j|-1$ edges (connecting $x_j^k$ and $x_j^{k+1}$); if $I_j \cap I_k$ is nonempty, we identify the corresponding vertices in the associated graphs.
\end{definition}

Notice that in $R(I_1, \cdots, I_m)$ a pair of vertices might be joined by multiple edges. The conditions in the definition of $E_{1,p}$ amount to saying that $R(J,I_1, \dots, I_p)$ is a tree which is obtained by pasting disjoint line segments (corresponding to the $I_k$) at various distinct points of a base line segment (corresponding to $J$), as well as a condition to guarantee that only one ordering of the $I_1, \dots, I_p$'s gives a nonzero output. This leads us to the following definitions.


\begin{definition}
Let $\{I_1, \dots, I_m\}$ be an unordered family of subsets $I_j \subset \{1, \dots, b\}$. We say that this family is \textbf{1-regular} if $R(I_1, \cdots, I_m)$ is a tree. \end{definition}

The condition of 1-regularity essentially means that there is \emph{some} reordering of this tuple so that the iterated cup-1 product of $e^{I_1}, \cdots, e^{I_m}$ is nonzero. Whether or not a family of subsets is 1-regular can be determined by a greedy algorithm. At each stage we have chosen $k$ distinct elements from this set so that $I_{\ell_1}, \cdots, I_{\ell_k}$ is 1-regular; we next determine if there is some $I_{\ell_{k+1}}$ which intersects the union of these exactly once. If not, the family is not 1-regular. If so, continue the algorithm. 

\begin{definition}
We define the \textbf{insertion products} to be the symmetric multilinear maps $$j_m: \Lambda^*(\mathbb Z^b)^{\otimes m} \to \Lambda^*(\mathbb Z^b)$$ of degree $1-m$, so that $$j_m(e^{I_1}, \cdots, e^{I_m}) = \begin{cases} \pm e^{I_1 \cup \cdots \cup I_m} & (I_1, \cdots, I_m) \text{ is 1-regular} \\ 0 & \text{else}.\end{cases}$$ 

To be explicit about the sign, one may reorder the $I_j$'s so that this is computed as an iterated $E_{1,1}$ product, with the signs of Proposition \ref{E1p-formula}; the result is independent of the choice of order. 
\end{definition}

\begin{remark}One may also define the sign inductively: suppose $R(I_2, \cdots, I_m)$ consists of $k$ components $$R(I_2, \cdots, I_m) = R(I_{1,1}, \cdots, I_{1,i_1}) \cup \cdots \cup R(I_{k,1}, \cdots, R_{I_{k,i_k}}).$$ The sign on the insertion product of these $k$ terms is already defined, by induction; after reordering them so that their intersections with $I_1$ are increasing, we have $$j_m(e^{I_1}, \cdots, e^{I_m}) = E_{1,k}(e^{I_1};j_{i_1}(e^{I_{1,1}}, \cdots, e^{I_{1,i_1}}), \cdots, j_{i_k}(e^{I_{k,1}}, \cdots, e^{I_{k,i_k}}))$$
using the higher operations in Proposition \ref{E1p-formula}.
\end{remark}

It follows from multilinearity that if $a = \sum_I a_I e^I$, we have $$j_m(a, \cdots, a) = \sum_{\substack{(I_1, \cdots, I_m) \text{ 1-regular} \\ I_j \ne I_k \text{ for all } j \ne k}} a_{I_1} \cdots a_{I_m} j_m(e^{I_1}, \cdots, e^{I_m}).$$ 
Here the sum is over all \textit{ordered sequences}. If $(I_1, \cdots, I_m)$ is 1-regular any permutation of this family of sets is again 1-regular and is distinct from the original family, because each $I_j$ labels a distinct subset of $\{1, \cdots, b\}$. Because $j_m$ is symmetric, it follows that we may rewrite this as $$j_m(a, \dots, a) = m! \sum_{\{I_1, \dots, I_m\} \text{ 1-regular}} a_{I_1} \dots a_{I_m} j_m(e^{I_1}, \dots, e^{I_m}),$$
where the sum runs over \textit{unordered sequences}. In particular, $j_m(a, \dots, a)/m!$ is always defined over the integers. It is worth giving a name to this operation.

\begin{definition}\label{def:insertionpower}
Let $a \in \Lambda^*(\mathbb Z^b)$ be an odd class. We say that the \textbf{insertion powers} of $a$ are $$a^{\circ m} = \frac{j_m(a, \dots, a)}{m!}.$$ Explicitly, if $a = \sum a_I e^I$, we have $$a^{\circ m} = \sum_{\{I_1, \cdots, I_m\} \text{ 1-regular}} a_{I_1} \cdots a_{I_m} j_m(e^{I_1}, \cdots, e^{I_m}).$$
\end{definition}

\medskip 
\subsection{The Kraines construction and the insertion product}
For the duration of this section, we write $\Lambda^*(\mathbb Z^b)$, together with its Hirsch algebra structure, as $\Lambda$, which we constructed in Section \ref{examplesHirsch}. It follows that the bar construction $B\Lambda$ carries a product $\mu: B\Lambda \otimes B\Lambda \to B\Lambda$ (cf. Theorem \ref{HirschBA}).  From this, and the inclusion $\Lambda \cong (B\Lambda)_1 \hookrightarrow B\Lambda$, we defined in Lemma \ref{mult-formula} operations corresponding to \emph{iterated multiplication}. In our case of interest these operations turn out to be symmetric in their inputs, and in fact equal to the insertion product of the previous section. 

For convenience, we state all results in the following theorem for odd classes. They hold for arbitrary elements of $\Lambda$ with suitable signs, but we do not need them. The main result of this section is the following characterization of iterated multiplication in $B\Lambda$. We give an explicit formula for the repeated product, and prove that it is both symmetric and associative.

\begin{theorem}\label{mu-is-insertion}
Define the map $\mu_n: \Lambda^{\otimes n} \to B\Lambda$ by iterated left-multiplication with the operation $\mu$: that is, $\mu_3(a_1, a_2, a_3) = \mu([a_1], \mu([a_2], [a_3]))$.  Then the component $$\mu_n^1: (\Lambda^{\textup{odd}})^{\otimes n} \to \Lambda \cong (B\Lambda)_1$$on odd-degree inputs is precisely the insertion product $j_n$.
\end{theorem}

\begin{proof}
We prove this by induction on $m$. The claim is tautological for $m = 1$, where $\mu_1^1(a) = a$ and $j_1(a) = a$. 

Supposing the claim is true for $\mu_m^1$ for all $m < n$, we will prove the claim for $\mu_n^1$. Because $\mu_n$ is multilinear, it suffices to verify it on basis elements $e^I$ for $I \subset \{1, \cdots, b\}$ a subset. We write $a_j = e^{I_j}$ for $1 \le j \le n$.

The key point is Lemma \ref{mult-formula}. To recall notation, if $J = \{j_1 < \cdots < j_k\} \subset \{2, \cdots, n\}$, we write $$\mu^1(J) = \mu_k^1(a_{j_1}, \cdots, a_{j_k});$$ if $J_1, \cdots, J_k$ is a list of subsets, then we write $\mu^1(J_1 \mid \cdots \mid J_k) = [\mu^1(J_1) \mid \cdots \mid \mu^1(J_k)].$ In this notation, Lemma \ref{mult-formula} tells us that $$\mu_{n-1}(a_2, \cdots, a_n) = \sum_{\substack{J_1, \cdots, J_k \subset \{2, \cdots, n\} \\ J_1 \cup \cdots \cup J_k = \{2, \cdots, n\} \\ J_i \cap J_j = \varnothing \text{ for } i \ne j}} \mu^1(J_1 \mid \cdots \mid J_k).$$ Next, our inductive hypothesis tells us what each expression $\mu^1(J_i)$ is. Each $J_i$ parameterizes a set of monomials $e^{I_k}$; we may compute their insertion product, and we have $\mu^1(J_i) = j(J_i)$ by inductive hypothesis. So we may write the above sum as $$\mu_{n-1}(a_2, \cdots, a_n) = \sum_{\substack{J_1, \cdots, J_k \subset \{2, \cdots, n\} \\ J_1 \cup \cdots \cup J_k = \{2, \cdots, n\} \\ J_i \cap J_j = \varnothing \text{ for } i \ne j}} [j(J_1) \mid \cdots \mid j(J_k)].$$
Next, if $a_1 = e^I$, when we apply $\mu([a_1], \mu_{n-1})$ we obtain $$\mu^1_n(a_1, a_2, \cdots, a_n) = \sum_{\substack{J_1, \cdots, J_k \subset \{2, \cdots, n\} \\ J_1 \cup \cdots \cup J_k = \{2, \cdots, n\} \\ J_i \cap J_j = \varnothing \text{ for } i \ne j}} E_{1,k}(e^{I}; j(J_1),  \cdots,  j(J_k)).$$
Here $j(J_i)$ is the insertion product of the elements $e^{I_k}$ for $k \in J_i$. This is given by $\pm e^{\cup J_i}$, where $\cup J_i = \cup_{k \in J_i} I_k$.

By Proposition \ref{E1p-formula}, if the above expression is nonzero, then the subsets $\cup J_1, \cdots, \cup J_k$ are necessarily disjoint with $I$ meeting each of them exactly once, so the graph realization $$R(I_2, \cdots, I_n) = R(J_1, \cdots, J_k)$$ must be a forest of exactly $k$ trees, and $I$ meets each tree exactly once --- so $R(I, I_2, \cdots, I_n)$ is a tree and the family is 1-regular. Further, $E_{1,p}$ is precisely given by the insertion product. 

Therefore, if a summand in the above expression is nonzero, the family $(I, I_2, \cdots, I_n)$ is 1-regular and that summand is the insertion product. The only way such a summand can be nonzero is if the $J_i$ each correspond to the components of $R(I_2, \cdots, I_n)$, and further the intersections $\{j_i\} = I \cap (\cup J_i)$ have $j_1 < \cdots < j_k$. Hence exactly one such summand is nonzero: the one with partition given as above, and with the unique ordering in which the above inequality holds. Thus the expression evaluates to $j(I, I_2, \cdots, I_n)$, as desired. This completes the induction.
\end{proof}

This immediately gives us an explicit formula for the Kraines construction $K(a)$ on the minimal torus. 

\begin{cor}\label{K-special}
Let $a \in \Lambda^3(\mathbb Z^n)$ be any element of degree 3. Writing $a^{\circ n}$ for the insertion powers of Definition \ref{def:insertionpower}, the Kraines construction on $a$ is the twisting sequence with $K(a)_{2n+1} = a^{\circ n}$ for all $n\geq1$; this is defined as a twisting sequence over $\Lambda$ (without rationalizing).
\end{cor}

A modified statement is true for elements of any odd degree (or, if one is more careful with signs, for arbitrary elements).

\medskip
\subsection{The main computation}
Combining the results of the previous two sections, we are now able to prove the following structure theorem for twisted homology of twisting sequences on the torus. The following phrasing, involving simplicial tori, is intended towards our application in monopole Floer homology; the result may be stated more generally about any reasonable model for cochains on the torus.

\begin{theorem}\label{toruscomp}
Let $\mathbb T$ be any finite simplicial complex whose realization is homeomorphic to the $b$-dimensional torus $T^b$. Then there is a zigzag of Hirsch algebra quasi-isomorphisms from the simplicial cochain (Hirsch) algebra $C^*_\Delta(\mathbb T)$ to the exterior algebra $C^*_\cube(T^b_1) \cong \Lambda^*(\mathbb Z^b)$, with the following property.

Suppose $\xi_\bullet$ is a twisting sequence in $C^*_\Delta(\mathbb T)$ with $F_n(\xi_\bullet) = 0$ for all $n > 1$; write $F_1(\xi_\bullet) = [\xi_3]$ as $a \in \Lambda^3$ for convenience. Then under the induced bijection $$hT(C^*_\Delta \mathbb T) \to hT(\Lambda^* \mathbb Z^n) = \Lambda^3 \oplus \Lambda^5 \oplus \cdots,$$ the element $[\xi_\bullet]$ is sent to the twisting sequence $K(a)$. Explicitly, we have $K(a)_{2n+1}$ given by the $n$-th insertion power $$K(a)_{2n+1} = a^{\circ n} = \frac{1}{n!} j_n(a, \cdots, a)$$ defined in Subsection 5.1.
\end{theorem}

\begin{proof}
Fix a homeomorphism $|\mathbb T| \cong T^b$ from the realization of the simplicial torus $\mathbb T$ to the standard $b$-dimensional torus, which also has a canonical homeomorphism to the realization of the minimal torus $T^n_1$ of section 5.1.

The appropriate zigzag of Hirsch algebras takes the form $$C^*_\Delta(\mathbb T) \leftarrow C^*_{\Delta-\text{sing}}(T^b) \leftarrow C^*_{\cube-\text{sing}}(T^b) \to C^*_\cube(T^b_1) = \Lambda^*(\mathbb Z^b).$$ The first map is obtained by restricting domain from singular chains to the simplicial chains; it is represented by a simplicial map $\mathbb T \to \text{Sing}_\Delta |\mathbb T|$ of simplicial sets and hence defines a map on Hirsch algebras. Similarly with the last map. 

The middle map from singular cubical cochains to singular simplicial cochains is the Cartan-Serre comparison map, which is an isomorphism of Hirsch algebras by Theorem \ref{comparison}. Hence this is a zig-zag of Hirsch algebra quasi-isomorphisms.

Because the characteristic classes of twisting sequences on Hirsch algebras are homotopy invariants, which are natural for Hirsch algebra maps, it follows that $F_n(z_* \xi_\bullet) = 0$ for all $n > 1$ and $F_1(z_* \xi_\bullet) = [\xi_3]$. Furthermore, Theorem \ref{charclass-facts} implies that there is exactly one twisting sequence in $\Lambda^* \mathbb Z^b$ with these properties, and it must be $K([\xi_3])$ --- which is therefore defined over the integers (though this follows from our explicit calculation of this element). The given formula for $K(a)$ was determined in Corollary \ref{K-special}.
\end{proof}

Applying Lemmas \ref{isom} and \ref{typeab}, we immediately obtain an isomorphism between the twisted cohomology $H^*_{\text{tw}}(\mathbb T; \xi_\bullet)$ and the (algebraic, and concretely computable) twisted homology of $\Lambda^* \mathbb Z^n$ with respect to the twisting sequence $(a, a^{\circ 2}, a^{\circ 3}, \cdots)$.

\begin{remark}
If one is careful, one observes that in the Lemmas \ref{isom} and \ref{typeab} we referred to a version of the twisted cohomology groups where one uses $\mathbb Z[T^{-1}, T\rrbracket$, completing in the $T$-direction, whereas the statement above refers to an uncompleted version. The key point is that when the chain complexes we apply this to are supported in bounded degrees --- as is the case for $C^*_\Delta(\mathbb T)$ and $\Lambda^*(\mathbb Z^n)$ --- this twisted cohomology is defined via Laurent polynomials, for in a fixed degree $d$ sufficiently high powers of $T$ must have coefficient zero).
\end{remark} 
\vspace{0.5cm}
\subsection{The case of local systems}\label{sec:localsys}
We would now like to extend our calculation of twisted (co)homology of the torus to the setting of local systems. When applying this to the setting of monopole Floer homology of a spin$^c$ three-manifold with $c_1(\mathfrak s)$ non-torsion, such a local system is most naturally \emph{a local system of $\mathbb Z[T,T^{-1}]$-modules}. 

Let $\mathbb T$ be a finite simplicial complex whose realization is homeomorphic to the $b$-dimensional torus, and suppose $\mathbb T$ is equipped with a local coefficient system $A$ of $\mathbb Z[T, T^{-1}]$-modules; this is described by some representation $\rho: \pi_1(\mathbb T) \to \text{Aut}(A)$, where we take automorphisms as a $\mathbb Z[T, T^{-1}]$-module.

Let $\xi_\bullet \in TS(C^*_\Delta \mathbb T)$ be a twisting sequence. We may define the twisted homology with respected to the local system $A$ as the homology of $C^\Delta_*(\mathbb T; A)$ with respect to the differential $$x \mapsto \partial_A x  + (\xi_3 \cap x) T + (\xi_5 \cap x) T^2 + \cdots$$
where $\partial_A$ is the differential with local coefficients (see for example \cite[Chapter $5$]{DK}). As $C^\Delta_*(\mathbb T)$ is finite-dimensional, this is a finite sum. 

\begin{theorem}\label{toruslocthm}
In the situation above, suppose we have $F_n(\xi_\bullet) = 0$ for all $n > 1$, and $[\xi_3] = a_\xi \in \Lambda^3(\mathbb Z^n)$ after choosing a basis of $\pi_1 \mathbb T$ (and thus an isomorphism of $H^* \mathbb T$ to $\Lambda^* \mathbb Z^b)$, then the twisted homology $H_*^{\textup{tw}}(\mathbb T; A, \xi_\bullet)$ is isomorphic to the twisted homology of the local system over the minimal torus $C^*_{\cube}(T^b_1)[T, T^{-1}] \otimes_{\mathbb Z[T]} A$ with respect to the differential $$x \mapsto \partial_A x + (a_\xi \cap x) T + (a_\xi^{\circ 2} \cap x) T^2 + \cdots $$
In this situation, the differential $\partial_A$ is computed explicitly as follows. If $x = e^{i_1} \wedge \cdots \wedge e^{i_k} \otimes a$, then $$\partial_A x = \sum_{1 \le j \le k} e^{i_1} \wedge \cdots \wedge e^{i_{j-1}} \wedge e^{i_{j+1}} \wedge \cdots \wedge e^{i_k} \otimes \rho(e_{i_j})(a).$$ Here $\rho(e_k)$ is $\rho$, applied to the $k$'th basis vector of $\mathbb Z^b$.
\end{theorem}
\begin{proof}
The zigzag of Hirsch algebras $$C^*_\Delta(\mathbb T) \leftarrow C^*_{\Delta-\text{sing}}(T^b) \leftarrow C^*_{\cube-\text{sing}}(T^b) \to C^*_\cube(T^b_1) = \Lambda^*(\mathbb Z^b)$$ has a corresponding form for chains (valued in any local system) $$C_{-*}^\Delta(\mathbb T; A) \to C_{-*}^{\Delta-\text{sing}}(T^b; A) \rightarrow C_{-*}^{\cube-\text{sing}}(T^b; A) \leftarrow C_{-*}^\cube(T^b_1; A) = \Lambda^{-*}(\mathbb Z^b) \otimes A.$$
The negative degrees appear because our conventions use cohomological gradings. To be clear, the grading of $\sigma \otimes a$ is given by $|a| - |\sigma|$.

These maps are all quasi-isomorphisms (they are comparison maps between various types of chains on the torus) and type (b) module maps (cf. Lemma \ref{typeab}). 

We would like to argue as follows: there is a filtration on each of these complexes so that the twisted differential preserves te filtration, and the associated graded of the twisted homology complex is precisely the complex $C_{-*}(\mathbb T; A)$ in its various guises; because the associated graded maps are all quasi-isomorphisms, and the spectral sequence converges, the maps between the various twisted homology groups are also quasi-isomorphisms.

The issue is in giving the correct spectral sequence. One should not filter by $T$-degree in $A$, as this is not complete. The most natural choice is by \emph{chain-degree}. If $\sigma \otimes a\in C_{-p}(\mathbb T; A)$ --- so that $|a| - |\sigma| = p$ --- we say that $\sigma \otimes a$ has chain-degree $-|\sigma|$, and we set $$F_p C_{-*}(\mathbb T; A) = \bigoplus_{k \le p} C_{-k}(\mathbb T) \otimes A.$$ 
Then as long as the models for chains on $\mathbb T$ are \emph{bounded in degree}, this is a complete filtration, and the argument above runs through. However, singular chains are unbounded complexes.

To remedy this, observe the following. Suppose $A$ is a nonnegatively-graded dg-algebra, and $M$ is a (cohomologically graded) bounded above dg-module with cohomology supported in degrees $[-n, 0]$. Then there is an $A$-module $M'$ which is supported in degrees $[-n-1, 0]$ and comes equipped $A$-module quasi-isomorphism $M \to M'$; furthermore, this construction is natural in pairs $(A,M)$ of the above form. One may simply set $$M'_k = \begin{cases} M_k & k \ge -n \\ M_{-n-1}/Z_{-n-1} & k = -n-1 \\ 0 & \text{else.} \end{cases}$$
Applying this natural truncation construction we now have a zigzag $$C_*^\Delta(\mathbb T; A) \to C_*^{\Delta-\text{sing}}(T^b; A)' \to C_*^{\cube-\text{sing}}(T^b; A)' \leftarrow C_*^\cube(T^b_1; A) = \Lambda^*(\mathbb Z^b) \otimes A$$ where each dg-module here has a finite filtration, and all maps are quasi-isomorphisms. Applying the argument above, we have a zigzag of isomorphisms between the relevant twisted homology groups, as desired.
\par
Finally, the determination of $\partial_A$ on $H^* \mathbb T$ follows directly from the definitions (cf. the discussion in Section \ref{HMblocsys} below). 
\end{proof}

\bigskip

\section{Applications to the monopole Floer homology group $\HMb_*$}\label{hmbformula}
In the first part of this section, we recall precisely what is proved in \cite{KM}. Then will finally prove Theorem \ref{thmmain}. We will conclude by discussing and prove generalizations to the case of non-torsion spin$^c$ structures and local systems.

In Chapters 33 and 34 of \cite{KM}, which we refer to heavily below, the authors investigate the monopole Floer homology groups $\HMb_*(Y,\mathfrak s; \Gamma_0)$ and compare them to \emph{coupled Morse homology groups} of the torus of reducible solutions; we write these groups as $$CMH_*(\mathbb T_{Y,\mathfrak s}, D_B; \Gamma_0).$$ Here $\Gamma_0$ is some local system over the torus $\mathbb T_{Y, \mathfrak s}$, hence an $H^1(Y;\mathbb Z)$-module. The complex $CMC_*(\mathbb T_{Y, \mathfrak s})$ computing this homology group is freely generated by pairs $(x, \lambda_i)$ where $x$ is a critical point of a Morse function on $\mathbb T$, and $\lambda_i$ labels an eigenspace of the operator $D_x$, listed so that $$\cdots < \lambda_0 < \lambda_1 < \lambda_2 < \cdots$$

Both $\HMb$ and $CMH$ carry the intrinsic structure of a module over $\Lambda^*(H_1 Y/\text{tors})[U, U^{-1}]$, both defined in similar ways (counting points in moduli spaces cut down by an appropriate cohomology class, either represented geometrically as a section of a vector bundle or as a \v{C}ech cocycle). Notice that while there is an obvious degree -2 map on $CMC_*(\mathbb T_{Y, \mathfrak s})$, given by sending $(x, \lambda_i)$ to $(x, \lambda_{i-1})$, this is not necessarily the same as the action of $U$, which is defined on \cite[Page 657]{KM}. 

While the proof of \cite[Theorem $35.1.1$]{KM} gives an isomorphism $$r_{Y, \mathfrak s}: \HMb_*(Y,\mathfrak s;\Gamma_0) \cong CMH_*(\mathbb T_{Y, \mathfrak s}, D_B; \Gamma_0)$$ as abelian groups, it is straightforward to see $r_{Y,\mathfrak s}$ is in fact a $\Lambda(H_1 Y/\text{tors})[U, U^{-1}]$-module isomorphism, because the module structures are defined in essentially the same way. This holds regardless of whether or not $\mathfrak s$ is torsion or $\Gamma_0$ is a local system with nontrivial monodromy and the proof is the same. In the non-torsion case, the domain is instead $\HMb_*(Y, \mathfrak s, c_b; \Gamma_0)$, where one uses an appropriate balanced perturbation to the Seiberg-Witten equations.
\par
They then compare these to appropriate twisted simplicial homology groups, by use of Proposition $34.2.1$ in the torsion case and $34.4.1$ in the non-torsion case. While these Propositions are phrased in terms of \v{C}ech representatives, there is no difficulty in phrasing their results simplicially, as in the proof of their Theorem $34.3.1$. 

What they show, precisely, is that when $\mathbb T$ is given a simplicial structure and $D_B$ is homotoped to a simplicial map $L: \mathbb T \to U(2)$, for an appropriate simplicial structure on the latter, then one may choose the Morse function on $\mathbb T$ so that there is a group isomorphism $$\varphi_{\Delta}: CMC_*(\mathbb T, L; \Gamma_0) \cong C_*^\Delta(\mathbb T; \Gamma_0)[T, T^{-1}]$$ given by sending the generator $(x, \lambda_i)$ to $x T^i$, and so that the coupled Morse differential on the domain is sent to the twisted differential $$x\mapsto d_{\Gamma_0} x + (\xi_3 \cap x)T,$$ where $\xi_3$ is a simplicial 3-cocycle on $\mathbb T$ pulled back from the 3-sphere. Here $T$ is a formal variable of degree $2$. This contrasts with the main body of the text above, where we used cohomological grading conventions, so instead $T$ had degree $-2$. Restated, $\varphi_\Delta$ gives an isomorphism of chain complexes from $CMC$ to $C_{\text{tw}}^\Delta(\mathbb T; \xi_{KM})$. 
\\
\par
Now we should discuss how $\varphi_\Delta$ behaves with respect to the additional structure. When $L$ factors through $SU(2)$ --- so $\mathfrak s$ is a torsion spin$^c$ structure --- the argument of Proposition $34.2.4$ shows that $\varphi_\Delta$ has $$\varphi_\Delta(Ux) = T^{-1} \varphi_\Delta(x).$$ 
However, these arguments do not help with the $\Lambda$-module structure. Without further work, all one can say is that for example if $\alpha \in H^1(\mathbb T)$, we have $$\varphi_\Delta(\alpha \cap x T^i) = \alpha \cap \varphi_\Delta(x) T^i + x_{-3} T^{i+1} + x_{-5} T^{i+2} + \cdots$$
where $x$ and $x_{-(2i+1)}$ differ in grading by $2i+1$. Furthermore, when $\mathfrak s$ is non-torsion, the arguments of \cite{KM} only show $$\varphi_\Delta(Ux) = T^{-1} \varphi_\Delta(x) + T y$$ for some appropriate $y=\sum a_i y_i$ where $y_i$ are critical points for which $ \mathrm{gr}_{z_i}(x,y_i)=4$ for some homotopy class of paths $z_i$; see \cite{CPPZ} for more details. (Notice that this is slightly stronger than simply knowing that $U$ is sent to $T^{-1}$ up to higher filtration.)
We summarize the above discussion in the following Proposition. 

\begin{prop}[Kronheimer--Mrowka]\label{varphiKM} Filter $\HMb_*(Y, \mathfrak s,c_b;\Gamma_0)$ by powers of $U$ and filter $H_{\text{tw}}^\Delta(\mathbb T; \xi_{KM})$ by powers of $T^{-1}$. There is a filtered isomorphism of relatively graded groups $$\varphi_{KM} = \varphi_\Delta \circ r_{Y,\mathfrak s} :\HMb_*(Y, \mathfrak s,c_b;\Gamma_0) \cong H_{\textup{tw}}^\Delta(\mathbb T; \xi_{KM}).$$ The associated graded map is a $\Lambda^*(H_1 Y/\textup{tors})$-module isomorphism sending $U$ to $T^{-1}$. When $\mathfrak s$ is torsion, so that we consider the usual group $\HMb_*(Y, \mathfrak s;\Gamma_0)$, the map $\varphi_{KM}$ sends $U$ to $T^{-1}$ even before passing to the associated graded modules.
\end{prop}

It is likely that one can give a completely explicit description of how $\varphi_{KM}$ intertwines the two $\Lambda$-module structures on both sides in general, as well as how $U$ relates to $T^{-1}$, see Section \ref{problems}. What follows below builds on Proposition \ref{varphiKM} and is carried out entirely on the twisted homology side.

\medskip

\subsection{Proof of the main theorem}
Let us first recall here the statement for the reader's convenience.

\begin{theorem}[Theorem \ref{thmmain}]
If $\mathfrak{s}$ is a torsion spin$^c$ structure, the monopole Floer homology group $\HMb_*(Y,\mathfrak s)$ is isomorphic to the extended cup homology $\overline{HC}_*^{\infty}(Y)$ as relatively $\mathbb{Z}$-graded $\mathbb Z[U, U^{-1}]$-modules.
\end{theorem}

Here, recall that after choosing an identification $H^1(Y;\mathbb{Z})=\mathbb{Z}^{b_1(Y)}$, the extended cup homology $\overline{HC}_*^{\infty}(Y)$ is the homology of $\Lambda^*(\mathbb{Z}^{b_1})[U, U^{-1}]$ with respect to the differential given by \begin{equation}\label{difffinal}
x\otimes U^{n}\mapsto (a_Y\cap x) U^{n-1} + (a_Y^{\circ 2}\cap x) U^{n-2} + (a_Y^{\circ 3}\cap x) U^{n-3} + \cdots\big.
\end{equation}
where $a_Y\in \Lambda^3\mathbb{Z}^{b_1}$ corresponds to the triple cup product $\cup^3_Y$ of $Y$, and the cap products are to be interpreted via the identification $\mathbb{Z}^{b_1(Y)}=H^*(\mathbb{T})$.

\begin{proof}
We need to show that $$H_{\text{tw}}^\Delta(\mathbb T; \xi_{KM}) \cong \overline{HC}_*^\infty(Y)$$ as $Z[T, T^{-1}]$-modules. Now $\xi_{KM} = (\xi_3, 0, \cdots)$ is a twisting sequence pulled back from a twisting sequence $\xi_{S^3}$ on $SU(2)=S^3$. All the higher characteristic classes $F_n(\xi_{S^3})$ vanish because the cohomology of $S^3$ vanishes in degree $\geq 4$. By functoriality, all characteristic classes $F_{n}(\xi_{KM})\in H^{2n+1}(\mathbb T;\mathbb{Q})$ also vanish, and $F_1(\xi_\bullet)$ corresponds to the pullback of the generator of $H^3(S^3;\mathbb{Z})$ under the classifying map.
\par

By the index theorem for families \cite[Lemma $35.1.2$]{KM}, the classifying map of the such family pulls back the generator of $H^3(S^3;\mathbb{Z})$ to the triple cup product
\begin{equation*}
\cup^3_Y\in \left(\Lambda^3 H^1(Y;\mathbb{Z})\right)^*=\Lambda^3 H^1(\mathbb{T};\mathbb{Z})=H^3(\mathbb{T};\mathbb{Z})
\end{equation*}
Theorem \ref{toruscomp} implies therefore that $H^*_{\text{tw}}(\mathbb T; \xi_{KM})$ is the twisted cohomology $H^*_{\text{tw}}\left(H^*(\mathbb{T});K(\cup_Y^3)\right)$ of $H^*(\mathbb{T})$ (thought of as a dga with trivial differential) with respect to the twisting sequence $K(\cup_Y^3)$. To conclude the homological statement, notice that
\begin{align*}
\HMb_*(Y,\mathfrak s)&\cong \HMb^*(-Y,\mathfrak s)\cong H^*_{\text{tw}}\left(H^*(\mathbb{T});K(-\cup_Y^3)\right)\\
&\cong H^*_{\text{tw}}\left(H^*(\mathbb{T});K(\cup_Y^3)\right)\cong\overline{HC}_*^{\infty}(Y)
\end{align*}
where the first isomorphism is Poincar\'e duality in Floer homology; we have that $\cup_{-Y}^3=-\cup_Y^3$. Now the twisting sequences $K(\cup_Y^3)$ and $K(-\cup_Y^3)$ differ by a sign in degrees $4n+1$, and the corresponding twisted cohomologies are seen to be isomorphic via the map that acts on $H^*(\mathbb{T})$ by multiplying by $-1$ in degrees congruent to $0$ and $3$ mod $4$. The last isomorphism is Hodge duality on $\Lambda^*(\Bbb Z^b)$, where we use that wedging with a form is Hodge dual to contraction with it (i.e. the differential in extended cup homology) up to a sign that does not affect the resulting homology.
\end{proof}

\medskip

\subsection{Local systems and non-torsion spin$^c$ structures}\label{HMblocsys}
Next we discuss the case of a non-torsion spin$^c$ structure $\mathfrak{s}$; the relevant group for our purposes is the one associated to balanced non-exact perturbation $c_b$, which is denoted by $\HMb_*(Y,\mathfrak{s},c_b)$ \cite[Chapter $30$]{KM}. The spin$^c$ structure determines a homomorphism $\varphi_{\mathfrak{s}}$ given by
\begin{align*}
\varphi_{\mathfrak s}&: H^1(Y;\mathbb Z) \to \mathbb Z\\
a& \mapsto \frac 12 \langle a \cup c_1(\mathfrak s), [Y]\rangle
\end{align*}
We know that $\HMb_*(Y,\mathfrak{s},c_b; \Bbb Z)$ can be identified with $$H_*^{\textup{tw}}(\mathbb T; \xi_\bullet, \Gamma_{\mathfrak s}),$$ where one uses a local system on the torus $\mathbb{T} = \mathbb T(Y,\mathfrak s)$ of reducible solutions with fiber $\mathbb{Z}[T,T^{-1}]$ whose monodromy around the loop $a$ in $\mathbb{T}$ is given
multiplication by $T^{\varphi_{\mathfrak{s}}(a)}$. Notice that the monodromy map does not preserve the $\mathbb{Z}$-grading, but only the $\mathbb{Z}/2N \mathbb{Z}$-grading, where $N$ is a generator of the image of $\varphi_{\mathfrak{s}}$. A similar isomorphism holds for $\HMb_*(Y, \mathfrak s, c_b; \Gamma_0)$, where $\Gamma_0$ is some additional local system over $\mathbb T$. As discussed in the opening of this section, this identification is a $U$ and $\Lambda$-module isomorphism on associated graded modules, but more work is necessary to determine how this identification intertwines the module structures in higher filtration. Henceforth we exclusively discuss the twisted homology group.\\
\par
Choose an identification $\mathbb T \cong (S^1)^b$, where $b = b_1(Y)$. The latter space has a canonical cubical decomposition, $T^b_1$, discussed extensively in Section 5; this cubical decomposition has $C^*_\cube(T^b_1) \cong \Lambda^*(\mathbb Z^b)$. We choose such an identification once and for all.
\par
If $\Gamma_0$ is a local system on $\mathbb{T}$, we can define the chain complex with local coefficients $C_*^\cube(T^b_1;\Gamma_0)$ as follows (see for example \cite[Chapter $5$]{DK}). The cubical decomposition of $T^b_1$ induces a cellular decomposition on its universal cover $\tilde{T}^{b_1}$; this is the usual lattice in $\mathbb R^n$. 

The complex $C_*^{\cube}(\tilde{T}^{b_1};\mathbb{Z})$ is naturally a module over the group ring $\mathbb{Z}[\pi_1(\mathbb{T})]$ via the action given by deck transformations. If we interpret the local system as a module $\Gamma_0$ over $\mathbb{Z}[\pi_1(\mathbb{T})] \cong \mathbb Z[\mathbb Z^b]$, we can then define
\begin{equation}\label{loccoeff}
C_*^\cube(T^b_1;\Gamma_0)=C_*^{\cube}(\tilde{T}^{b_1};\mathbb{Z})\otimes_{\mathbb{Z}[\pi_1(\mathbb{T})]}\Gamma_0,
\end{equation}
equipped with the induced differential.
\par
Similarly, consider the local system on $\mathbb T$ with fiber $\mathbb{Z}[T,T^{-1}]$ and monodromy around the loop $a$ given by $\rho(a) \otimes T^{\varphi_{\mathfrak s}(a)}$; we will denote this by $\Gamma^{\mathfrak{s}}$. This induces a local system over $T^b_1$; we use the same notation for this local system. 

The construction above, applied now to the local system $\Gamma_0\otimes \Gamma^{\mathfrak{s}}$, gives rise to a twisted chain complex $C_*^\cube(T^b_1;\Gamma_0\otimes \Gamma^{\mathfrak{s}})$. Very explicity, after unwinding the definition above, we have
\begin{equation}\label{loccoeff1}
C_*^{\cube}(\mathbb{T};\Gamma_0\otimes \Gamma^{\mathfrak{s}})=C_*^{\cube}(\mathbb{T};\mathbb{Z})\otimes \Gamma_0\otimes \mathbb{Z}[T,T^{-1}]
\end{equation}
equipped with the differential $d^1$ given by
\begin{equation}\label{loccoeff2}
d^1(a_{i_1} \wedge \cdots \wedge a_{i_n} \otimes r \otimes T^k) = \sum_{j=1}^n (-1)^j (a_{i_1} \wedge \cdots \widehat a_{i_j} \cdots \wedge a_{i_n}) \otimes \rho(a_{i_j}) r \otimes T^{k + \varphi_{\mathfrak s}(a_{i_j})},
\end{equation}
where
$$\rho: H^1(Y;\mathbb Z)=H_1(\mathbb{T};\mathbb Z) \to \text{Aut}(\Gamma_0)$$
is the monodromy of the local system $\Gamma_0$.
\\
\par
With this in mind, we are ready to define the `twisted' version of extended cup homology which is relevant for our purposes.
\begin{definition}
Suppose $(Y, \mathfrak s)$ is a closed oriented 3-manifold equipped with a (possibly non-torsion) spin$^c$ structure, and let $\Gamma_0$ be a local system on $\mathcal B^\sigma(Y, \mathfrak s)$ with monodromy $\rho$. Denote by $\mathbb{T}$ the $b_1(Y)$-dimensional torus of reducible solutions, with a fixed identification $\mathbb{T}=(S^1)^b$ (and induced cubical decomposition) as above. The \emph{extended cup complex} $C_*^\infty(Y, \mathfrak s; \Gamma_0)$ is the twisted chain complex $$C_*^\cube(T^b_1;\Gamma_0\otimes \Gamma^{\mathfrak{s}})$$ in (\ref{loccoeff1}) in equipped with the differential $d^1 + \bar{d}_{\cup}$, where $d^1$ is given in (\ref{loccoeff2}) and $\bar{d}_{\cup}$ is the differential in (\ref{difffinal}). The \textit{extended cup homology} $\overline{HC}_*^\infty(Y, \mathfrak s; \Gamma_0)$ is, by definition, the homology of this complex.
\end{definition}

Notice that when $\mathfrak{s}$ is torsion we have $\varphi_{\mathfrak s} = 0$. In this case, if we take the trivial local system, we obtain the standard cup homology studied earlier. Furthermore, the invariant associated to a balanced perturbation coincides with the standard one. With this in mind, the following result can be thought as a generalization of Theorem \ref{thmmain}.

\begin{theorem}\label{thm:localsys}
Let $(Y, \mathfrak s)$ be a closed oriented spin$^c$ 3-manifold and $\Gamma_0$ a local system on $\mathcal B^\sigma(Y, \mathfrak s)$. Then $$\HMb_*(Y, \mathfrak s, c_b;\Gamma_0) \cong \overline{HC}_*^\infty(Y, \mathfrak s;\Gamma_0)$$ 
as relatively $\mathbb{Z}/2N \mathbb{Z}$-graded abelian groups, where $N$ is a generator of $\textup{Im}(\varphi_{\mathfrak{s}}) \subset \mathbb Z$. In the torsion case, this map is also an isomorphism of $\mathbb{Z}[T,T^{-1}]$ modules.
\end{theorem}
\begin{remark}\label{Unontors}
In the non-torsion case, while we cannot determine the $\mathbb{Z}[T,T^{-1}]$-module structure, we can still say that the isomorphism sends $U$ to $T^{-1}$ up to higher filtration.
\end{remark}
\begin{proof}
Referring to the discussion preceding Proposition \ref{varphiKM}, consider the local system $\Gamma = \Gamma_0 \otimes \Gamma^{\mathfrak{s}}$; recall that it has fiber $\Gamma_0[T,T^{-1}]$ and monodromy around the loop $a$ given by multiplication by $T^{\varphi_{\mathfrak s}(a)}$. Then, for a sufficiently fine simplicial structure structure on $\mathbb{T}$, there is a $3$-cocycle $\xi_3$ whose cohomology class represents the triple cup product of $Y$ so that simplicial chains with local coefficients $C_*^{\Delta}(\mathbb T; \Gamma)$ with twisted differential
\begin{equation*}
\sigma\mapsto d_{\Gamma}\sigma+(\xi_3\cap \sigma) T
\end{equation*}
have homology groups isomorphic to $\HMb_*(Y,\mathfrak{s},c_b;\Gamma_0)$. In the non-torsion case this isomorphism is only shown to hold as $\mathbb{Z}$-modules in \cite{KM}.

Finally, $\xi_3$ is pulled back from a $3$-cocycle on $U(2)$; as the latter has vanishing cohomology in degrees $\geq 5$, functoriality implies once more that all characteristic classes $F_n(\xi_\bullet)$ with $n>1$ vanish while $F_1(\xi_\bullet) = [\xi_3] = \cup_Y^3$. The proof then follows as in the previous case, using the computation with local coefficients in Theorem \ref{toruslocthm}.
\end{proof}

\medskip
\section{Some open questions}\label{problems}
This article answers a specific question in the $\HMb_*$ story. The theory is much richer and many parts of this story remain open; we collect some interesting questions below.
\\
\par
First, there is the combinatorial algebra question of comparing cup homology to extended cup homology already mentioned in the Introduction.

\begin{quest}\label{cupequalsext}
Let $a \in \Lambda^3(\mathbb Z^b)$ be a degree 3 class; we may define its cup homology and extended cup homology modules as before. Is there an isomorphism of $\mathbb Z[T, T^{-1}]$-modules  between the associated cup homology $HC^\infty$ and its extended version $\overline{HC}^\infty$?
\end{quest}

Referring to the discussion at the beginning of Section \ref{hmbformula}, and in particular Proposition \ref{varphiKM}, we know that $\HMb_*(Y;\mathfrak{s})$ is also a module over $\Lambda^{*}(H_1(Y;\mathbb{Z})/\textup{tors})$; from the coupled Morse homology picture, via the isomorphism of Theorem \ref{thmmain} this coincides with the natural action by contraction on the right on $\overline{HC}^\infty$ only up to lower filtration terms. It is then natural to ask the following.

\begin{quest}
Can one identify the $\Lambda^{*}(H_1(Y;\mathbb{Z})/\textup{tors})$-action induced on $\overline{HC}^\infty$ by the isomorphism in Theorem \ref{thmmain}? If the isomorphism in Question \ref{cupequalsext} holds, can one describe the corresponding action on $HC^\infty$?
\end{quest}

One may also ask the above question when there is the additional structure of a local system, and possibly a local system of $\mathbb Z[T, T^{-1}]$-modules, as in the case of $\HMb_*(Y, \mathfrak s, c_b)$ when $c_1(\mathfrak s)$ is non-torsion. In the latter case, it is also unclear how to describe the $\mathbb{Z}[U^{-1},U]$-module structure (see also Remark \ref{Unontors}), so we ask the following.

\begin{quest}
For a non-torsion spin$^c$ structure, can one describe the $\mathbb{Z}[U^{-1},U]$-action on $\overline{HC}^\infty$ induced the isomorphism in Theorem \ref{thm:localsys}?
\end{quest}

Notice that while Question \ref{cupequalsext} is purely combinatorial, Questions 2 and 3 require a better understanding of both the algebraic picture and the description of the invariant in terms of couple Morse homology; their answer is not clear even in the algebraically simpler setting of twisted de Rham cohomology.
\par
To conclude, on the side of coupled Morse homology itself, it would be interesting to understand the coupled Morse homology of families that do not factor through $U(2)$. If one had the pipe-dream goal of completely computing the functor $\HMb_*(Y, \mathfrak s)$ on the cobordism category in some explicit algebraic fashion, this may be necessary, as it is unlikely there is a \emph{coherent} way to homotope the classifying maps $D_B: \mathbb T(Y, \mathfrak s) \to U(\infty)$ to $U(2)$, in a way compatible with cobordisms. A starting point for this is the following.

\begin{quest}Is there a (simplicial) twisting sequence $(\zeta_3, \dots, \zeta_{2n-1}, 0, \dots)$ on $SU(n)$ so that the following property holds? Whenever $Q$ is a compact manifold and $L: Q \to SU(n)$ is a simplicial map classifying a family of self-adjoint Fredholm operators, we have an isomorphism $$CMH_*(Q, L) \cong H^{\textup{tw}}_*(Q; L^* \zeta_\bullet).$$
\end{quest}

If the answer to this question is positive, then one should also develop the corresponding version with spectral flow (maps to $U(n)$). One would then need to answer the analogue of Question 3 in this more general setting to get a full understanding of the aforementioned functoriality problem. The case $n=3$ is already very interesting as $SU(3)$ is not the product $S^3\times S^5$ even though they have the same cohomology ring (there is a non-trivial $\mathrm{Sq}^2$ operation relating the generators in degree $3$ and $5$).

\vspace{2cm}
\bibliographystyle{alpha}
\bibliography{main}
\end{document}